\documentclass[a4paper,11pt,reqno]{amsart}

\usepackage{amssymb}
\usepackage{times}

\usepackage{tikz}
\usetikzlibrary{decorations.pathreplacing}
\usepackage{caption}
\usepackage{subcaption}

\let\emptyset \undefined
\newsymbol\emptyset    203F

\definecolor{darkred}{rgb}{0.9,0.1,0.1}

\theoremstyle{plain}
\newtheorem{theorem}{Theorem}[section]

\newtheorem{lemma}[theorem]{Lemma}
\newtheorem{proposition}[theorem]{Proposition}

\newtheorem{definition}[theorem]{Definition}
\newtheorem{assumption}[theorem]{Assumption}
\newtheorem{notation}[theorem]{Notation}

\theoremstyle{remark}
\newtheorem{remark}[theorem]{Remark}

\numberwithin{equation}{section}
\numberwithin{figure}{section}

\def\bul{$\bullet\;$}
\def\begeq{\begin{equation}}
\def\endeq{\end{equation}}

\newcommand{\ls}{\lesssim}

\newcommand{\E}{\mathbb{E}}

\newcommand{\eps}{\varepsilon}

\newcommand{\phim}{\varphi_{-1}}
\newcommand{\phip}{\varphi_{+1}}

\newcommand{\R}{\mathbb{R}}
\newcommand{\Z}{\mathbb{Z}}
\newcommand{\F}{\mathcal{F}}

\newcommand{\Aa}{\mathcal{A}}
\newcommand{\Bb}{\mathcal{B}}
\newcommand{\A}{\mathcal{A}}

\newcommand{\N}{\mathbb{N}}
\newcommand{\Je}{J_{y,\eps}}
\newcommand{\DE}{\Delta E}
\newcommand{\Ixu}{I^{u_{\pm}}_{x_{\pm}}}
\newcommand{\Hp}{J_{z,+}^\eps}
\newcommand{\Hm}{J_{y,-}^\eps}

\newcommand{\klm}{k_{y,-}^\eps}
\newcommand{\Be}{\mathcal{J}_{y,\eps}}

\newcommand{\de}{d_\eps}

\newcommand{\Rx}{\mathsf{R}_{y, z}}
\newcommand{\kxepsm}{k_{y,-}^\eps}
\newcommand{\kxepsp}{k_{y,+}^\eps}
\newcommand{\kzepsm}{k_{z,-}^\eps}
\newcommand{\kzepsp}{k_{z,+}^\eps}
\newcommand{\I}{\mathcal{I}}

\newcommand{\Ayt}{\Aa_{Y,2}}
\newcommand{\Ayth}{\Aa_{Y,3}}
\newcommand{\JY}{\mathcal{J}_Y}
\newcommand{\Rrl}{R^{\chi_r}_{\chi_l} }

\DeclareMathOperator{\sign}{sign}
\DeclareMathOperator{\Id}{Id}

\newcommand{\ue}{u_\eps}

\newcommand{\Le}{L_\eps}

\newcommand{\hx}{h _{(x_-,x_+)}^{u_-,u_+}}

\begin{document}

\title[Invariant measures]
{Invariant measure of the stochastic Allen-Cahn equation: the regime of small noise \\and large system size}
\author{Felix Otto}
\address{Felix Otto, MPI for Mathematics in the Sciences}
\email{felix.otto@mis.mpg.de}

\author{Hendrik Weber}
\address{Hendrik Weber, University of Warwick
}
\email{hendrik.weber@warwick.ac.uk}

\author{Maria G. Westdickenberg}
\address{Maria G. Westdickenberg, RWTH Aachen University}
\email{ maria@math1.rwth-aachen.de}



 \begin{abstract}
We study the invariant measure of the one-dimensional stochastic Allen-Cahn equation for a small noise strength and a large but finite system.  We endow the system with inhomogeneous Dirichlet boundary conditions that enforce  at least one transition from $-1$ to $1$. (Our methods can be applied to other boundary conditions as well.) We are interested in the competition between the ``energy'' that should be minimized due to the small noise strength and the ``entropy'' that is induced by the large system size.

Our methods handle system sizes that are exponential with respect to the inverse noise strength, up to the ``critical'' exponential size predicted by the heuristics. We capture the competition between energy and entropy through  upper and lower bounds on the  probability of extra transitions between $\pm 1$. These bounds are sharp on the exponential scale and imply in particular that the probability of having one and only one transition from $-1$ to $+1$ is exponentially close to one. In addition, we show that the position of the transition layer is uniformly distributed over the system on scales larger than the logarithm of the inverse noise strength.

Our arguments rely on local large deviation bounds, the strong Markov property, the symmetry of the potential, and measure-preserving reflections. \end{abstract}


\date\today

\maketitle

\section{Introduction}
\label{s:Intro}
In this paper we study the unique invariant measure of the stochastically perturbed Allen-Cahn equation
\begin{equation}\label{e:ACE}
\partial_t \ue(t,x)   \, = \,  \partial_{x}^2\,  \ue(t,x) - V'(\ue(t,x) )    + \sqrt{2 \eps} \,  \eta(t,x),
\end{equation}
where $\ue$ is a one-dimensional order parameter defined for all non-negative times $t \in \R_+$ and  $x \in (-\Le,\Le )$.  Here $\eta$ is a formal expression denoting space-time white noise and $V$ is  a symmetric double-well potential.  The canonical choice for $V$ is
\begin{equation*}
V(u)\,=\,\frac{1}{4}(1-u^2)^2,
\end{equation*}
although more general choices are possible (see Assumption~\ref{ass:V} below).  We are interested in the properties of the invariant measure for large system sizes,
$$\Le\gg 1.$$

It is well-known that for $\eps\downarrow 0$ and fixed system size $L$, the invariant measure of the Allen-Cahn equation concentrates on minimizers of the energy
\begin{align*}
\int_{- L}^ {L}  \left(\frac{1}{2}(\partial_x u)^2+V(u)\right)\,dx.
\end{align*}
This follows from large deviation theory.  In fact, even for system sizes $\Le$ that grow with $\eps$, the same is true.  Indeed, in \cite{W10} the second author proved this fact for $\Le\sim\eps^{-\alpha}$ for any $\alpha < \frac{2}{3}$.

Our main goal in the current paper is to go up to interval sizes that are \emph{exponential} with respect to $\eps^{-1}$ and, specifically, to understand the \emph{competition between energy and entropy} that emerges in this regime.  Let us first consider the effect of \underline{\emph{energy}} on the measure.
The intuition is that the invariant measure can be viewed as a Gibbs measure with the given energy, i.e., that it is in some heuristic sense proportional to
$$\exp\left(-\frac{1}{\eps}\int_{- \Le}^ {\Le}  \left(\frac{1}{2}(\partial_x u)^2+V(u)\right)\,dx\right).$$
The heuristic picture then says that, because of the potential term in the energy, functions $u$ supported on this measure are most likely to be close to one or the other  minimum of $V$ on most of $[-\Le,\Le]$.  On the other hand, because of the gradient term in the energy, there is an energetic ``cost'' $c_0$ for each transition between these two preferred states, making such transitions unlikely.  When the system size is order-one, this is  the end of the story.

Now let us consider the competing effect of \underline{\emph{entropy}} on the measure when the system size is large.
Namely, the probability of finding a transition between the minima of $V$ is increased by the fact that it is possible for the transition to occur anyplace in the system.  Hence, the folklore is that the probability of finding $n$ transition layers scales like
\begin{equation}\label{e:expprefact}
(\Le)^n\exp\left(-\frac{nc_0}{\eps}\right).
\end{equation}
This competition between entropy and energy is captured (on the exponential level) in our first theorem, Theorem~\ref{t:layers} below. Our second  result, Theorem~\ref{t:uniform}, then shows  the uniform distribution of transitions within the domain.

As far as our methods, the central idea is that one can decompose the measure into conditional measures and the corresponding marginals in order to reduce to order-one intervals  on which one can apply large deviation theory.
Along the way, it is important for us to use measure-preserving \emph{reflection arguments} that allow us to transform the underlying Brownian paths. The detailed structure of the (deterministic) energy functional is also critical in our proofs.

We will state our results in  detail in Subsection~\ref{ss:MR} after first explaining our set-up and notation.

\subsection{Set-up and notation}

For the potential $V$ in~\eqref{e:ACE}, we need a symmetric double-well potential with at least  superlinear growth at infinity.  For simplicity, we assume that the two minima of $V$ are normalized to be at $\pm 1$ and that the minimum value of the potential is zero.  To be precise, our assumptions are:

\begin{assumption}\label{ass:V}
$V$ is a smooth, even potential such that, on $(0,\infty)$, $V$ satisfies
\begin{align}
&V(u) \geq 0 \quad \text{and} \quad V(u)=0 \quad \text{iff}\quad u= 1, \notag \\
&V'(u)=0 \quad\text{if and only if}\quad u=1 \notag,\\
&V''(1)>0,\notag\\
&V(u)    \geq  u^{1+\beta}/C \quad\text{for} \quad u\geq C \quad \text{for some $C<\infty$ and $\beta>0$}. \label{e:assumptions}
\end{align}
\end{assumption}

\begin{remark}
If we assume superquadratic growth on $V$ at infinity (recall that we have \emph{quartic} growth of the standard double well potential $V(u)=(1-u^2)^2/4$), some of our technical lemmas simplify slightly. In particular, one can remove the dependence of the minimal system size $\ell_*$ on $M$ in Lemmas~\ref{le:lazy} and \ref{l:cl}.
\end{remark}

Because of the normalization of our potential, the transitions that we are interested in are transitions between $\pm 1$.  We make the notion precise in the following definition.
\begin{definition}[Up/down transition layers]\label{def:layer}
We say that $u$ has an up transition layer on $(x_-,x_+)$ if
$$ u(x_{\pm})= \pm  1 \qquad\text{and}\qquad  |u(x)|< 1\quad\text{for all}\;x\in(x_-,x_+).$$
We say that $u$ has a down transition layer on $(x_-,x_+)$ if the same condition holds with signs reversed, and that $u$ has a transition layer if it has an up or down transition layer.
\end{definition}

For the boundary conditions on our PDE, we will work with the popular inhomogeneous Dirichlet boundary conditions
\begin{equation}
\ue \big(t,\pm \Le \big) = \pm \frac{1}{\eps} \int^{\ell_0}_{-\ell_0} V\big(\eps^{1/2} (\hat{u} -1) +1 \big)1.\label{bcj}
\end{equation}
Because of the boundary conditions, there is necessarily \emph{one} up transition layer, and the question is whether there are additional transition layers.  Moreover, if there are additional layers, they come as a pair of an up layer and a down layer.  Note that our methods can also handle other boundary conditions, for instance periodic boundary conditions or Dirichlet boundary conditions that do not force a transition layer to be present.

We will denote the invariant measure of~\eqref{e:ACE} subject to the boundary conditions~\eqref{bcj} by
$\mu^{-1,1}_{\eps,(-\Le,\Le)}$ and the corresponding expectation by $\mathbb{E}^{\mu_\eps,-1,1}_{(-\Le,\Le)}(\cdot)$.  We will often use the fact that
the  measure $  \mu^{-1,1}_{\eps,(-\Le,\Le)}$ can be written as a Gaussian measure with density \cite{Z88}.  Namely, one can express the expectation of any test function $\Phi$ as
\begin{align}
\mathbb{E}^{\mu_\eps,-1,1}_{(-\Le,\Le)}(\Phi)=\frac{\mathbb{E}^{\mathcal{W}_\eps,-1,1}_{( -\Le,\Le)}\left[\Phi(u)\exp\Big(-\frac{1}{\eps}\int_{-\Le}^{\Le} V(u)\,dx\Big)\right]}{\mathbb{E}^{\mathcal{W}_\eps,-1,1}_{( -\Le,\Le)}\left[\exp\Big(-\frac{1}{\eps}\int_{-\Le}^{\Le} V(u)\,dx\Big)\right]}. \label{m30.1}
\end{align}
Here $ \mathbb{E}_{(-\Le,\Le)}^{\mathcal{W}_\eps,-1,1} $ denotes the expectation with respect to the measure  $  \mathcal{W}^{-1,1}_{\eps, (-\Le,\Le)}$, which is the distribution of a Brownian bridge on $(-\Le, \Le)$ from $-1$ to $+1$ with variance proportional to $\eps$. Properties of  $\mathcal{W}^{-1,+1}_{\eps, (-\Le,\Le)}$ will be discussed in detail in Section   \ref{s:Prlm}.

The deterministic Allen-Cahn equation (set $\eta=0$ in~\eqref{e:ACE}) is the $L^2$-gradient flow of the energy
\begin{align}
E(u):=\int_{-\Le}^{\Le}\left(\frac{1}{2}(\partial_x u)^2+V(u)\right)\,dx.\label{energy}
\end{align}
When we need to refer to the energy on all of $\R$ or the localized energy on a subinterval, we will denote this with a subscript:
\begin{align}
E_{(-\infty,\infty)}(u)&=\int_{-\infty}^{\infty}\left(\frac{1}{2}(\partial_x u)^2+V(u)\right)\,dx,\notag\\
E_{(-\ell,\ell)}(u)&=\int_{-\ell}^{\ell}\left( \frac{1}{2}(\partial_x u)^2+V(u)\right)\,dx.
\label{enloc}
\end{align}
As mentioned above, the energy functional will be important for understanding the invariant measure of the stochastic equation.  In particular, the probability of finding transition layers will depend on the energetic ``cost'' of a transition layer on $\R$, that is:
\begin{equation}
c_0 \, := \, \inf \big\{  E_{(-\infty,\infty)}(u) \colon \, u(\pm \infty) = \pm 1 \big\}.\label{c00}
\end{equation}
It is well known \cite{MM} that this cost can be computed explicitly as
\begin{equation}
c_0\,=\,\int_{-1}^{1}\sqrt{2V(u)}\,du\,\overset{\bf{Ass.}~\ref{ass:V}}=\,2\int_{0}^{1}\sqrt{2V(u)}\,du;\label{c0}
\end{equation}
see the beginning of Section~\ref{s:detac} for an explanation.

We will often refer to scaling regimes in our results.  To this end, we define the following notation.

\begin{notation}\label{N:Notation}
The well-established theory of large deviations applies on intervals whose length is order-one with respect to $\eps$. A main point of this paper, however, is to obtain estimates on intervals that are exponentially large with respect to $\eps$ and for which, consequently, the established theory does not apply. We therefore use a subscript of $\eps$ in order to distinguish interval lengths that are large with respect to $\eps$ from quantities that are order-one with respect to $\eps$.

To specify bounds with respect to $\eps$, we sometimes make use of the shorthand notation $\ll$, $\lesssim$, and $\lessapprox$. To explain:
We write $$A_\eps\ll B_\eps$$ if  for every $C <\infty$, we have $A_\eps/B_\eps \leq 1/C$ for $\eps$ sufficiently small.

We write $$A_\eps\lesssim B_\eps $$ if there exists a universal constant $C<\infty$ such that $A_\eps\leq C\,B_\eps$, and similarly for $A_\eps\gtrsim B_\eps$. If both inequalities hold, then we write $A_\eps\sim B_\eps$.

We write $$A_\eps \lessapprox B_\eps$$ if for every $\alpha>0$ 
we have $A_\eps\leq B_\eps +\alpha$ for $\eps$ sufficiently small, and similarly for $A_\eps\gtrapprox B_\eps$.  If both inequalities hold, then we write $A_\eps\approx B_\eps$.

We use numbered constants $C_1$, $C_2$, et cetera, to denote specific constants that we refer to later in the paper. On the other hand, we use $C$ to denote a generic order-one constant whose value may change from place to place. Throughout the article, $C$ or a numbered constant $C_i$ is a constant that is universal except for a possible dependence on the potential $V$.
\end{notation}

We are now ready to state our results.

\subsection{Main results}
\label{ss:MR}

Recall that the boundary conditions imply that there must be at least one up layer and that any additional layers come in pairs. We will always consider the regime where the system size $\Le$ satisfies
\begin{align}
 1\ll \Le\lesssim \exp\left(\frac{c_0'}{\eps}\right)\quad\text{for some}\;\; c_0'<c_0.\label{lbd}
\end{align}
(Recall that $c_0$ is the energy cost defined in~\eqref{c00}.) This is the regime in which one expects the probability of extra transitions to go to zero and in particular to obey the energetic and entropic scaling expressed in~\eqref{e:expprefact}. Our first result captures this behavior on the exponential level.
\begin{theorem}\label{t:layers}
Suppose that $\Le$ satisfies~\eqref{lbd}.
Then for every $n \in \N$ and $\gamma>0$ sufficiently small, there exists $\eps_0>0$ such that for  $\eps \leq \eps_0$, one has the upper bound
\begin{align*}
\mu_{\eps,(-\Le,\Le)}^{-1,1} \big(&  \text{ $u$ has $(2n+1)$ transition layers  } \big) \\
&\leq (\Le)^{2n}\,\exp\left(-\frac{ 2n c_0 - \gamma}{\eps}\right),
\end{align*}
and the lower bound
\begin{align*}
\mu_{\eps,(-\Le,\Le)}^{-1,1} \big(&  \text{ $u$ has $(2n+1)$ transition layers  } \big)\\
 &\geq (\Le)^{2n}\,\exp\left(-\frac{ 2n c_0 + \gamma}{\eps}\right).
\end{align*}
\end{theorem}

\begin{remark}
One should note that because of the error term $\gamma$, our result sees only information on the exponential level. In particular, if one has an exponential system size such that
\begin{equation*}
\eps \,  \log \, \Le \approx   c_0' < c_0,
\end{equation*}
then what our result says is that for any $n \in \N$ we have
\begin{equation*}
\eps \log \mu_{\eps,(-\Le,\Le)}^{-1,1} \big(  \text{ $u$ has $(2n+1)$ transition layers  } \big) \approx  - 2n ( c_0 - c_0' ).
\end{equation*}
\end{remark}

\begin{remark}
Throughout the paper, when we say ``$u$ has $2n+1$ layers,'' we mean that $u$ has \underline{at least} $2n+1$ layers.
\end{remark}

\begin{remark}
As mentioned above, our techniques can also handle different boundary conditions, e.g.,  periodic boundary conditions or Dirichlet boundary conditions that do not enforce a transition layer.  For instance, for periodic boundary conditions or Dirichlet conditions $u(\pm \Le)=1$, the probability of $2n$ transition layers is bounded above and below by
\begin{align*}
(\Le)^{2n}\,\exp\left(-\frac{ 2n c_0\mp\gamma}{\eps}\right),
\end{align*}
respectively, while for homogeneous Dirichlet boundary conditions, the probability of $n$ transition layers is bounded above and below by
\begin{align*}
(\Le)^{n}\,\exp\left(-\frac{ n c_0\mp\gamma}{\eps}\right),
\end{align*}
respectively.
\end{remark}

Our second main result states that, on scales larger than logarithmic in $1/\eps$, the layer location is uniformly distributed in the following sense.

\begin{theorem}\label{t:uniform}
Consider $\mu^{-1,1}_{\eps,(-\Le,\Le)}$  in the regime
\begin{align}
 |\log \eps| \ll  \Le\lesssim \exp\left(\frac{c_0'}{\eps}\right)\quad\text{for some}\;\; c_0'<c_0.
\end{align}
Let $d_\eps>0$ be such that
\begin{equation*}
|\log \eps| \ll d_\eps \leq \Le.
\end{equation*}
Then uniformly for any $x$ such that  $[x- d_\eps,x + d_\eps] \subseteq [-\Le,\Le]$, we have
\begin{align}
 \frac{\Le}{\de} \; \mu^{-1,1}_{\eps,(-\Le,\Le)}\big(&\text{there is an up layer contained in }[x-d_\eps,x+d_\eps] \big)
\approx 1.\label{unif}
\end{align}
\end{theorem}
The theorem says that the probability of finding an up transition layer in a subinterval of length $2d_\eps$ given a system size $2\Le$ is approximately $d_\eps/\Le$ in the sense expressed in~\eqref{unif}, independent of the location of the subinterval. (The existence of an up transition layer somewhere in the system is forced by the boundary conditions.) In this sense, the layer locations are approximately uniformly distributed. The theorem is strongest when considering $d_\eps$ at the lower range of validity: It shows that the uniform distribution holds not only on macroscopic intervals but also down to the logarithmic scale.

We remark that the uniform distribution of the layer location in our regime is very different from the characterization of the layer distribution in the case $\Le=|\log \eps|/4$ studied in~\cite{BBB08}; see Subsection~\ref{ss:back} below for more discussion.
\subsection{Methods:  Markovianity, compact sets, and reflections}\label{ss:methods}

Our approach for Theorem ~\ref{t:layers} relies on a simple idea. Namely, while we cannot use large deviation theory directly on $(-\Le,\Le)$, we can use the Markovianity of the underlying reference measure to reduce to order-one subintervals on which we can. In particular, by taking large (but order-one) subintervals and conditioning on the boundary values of a larger, surrounding subinterval, we can take advantage of \emph{large deviation bounds} with a cost that is \emph{to leading order independent of the subinterval size}. This method is similar in spirit to Freidlin and Wentzell's approach of calculating the expected exit time from a metastable domain for a diffusion process with small noise (\cite{FW}, see Subsection \ref{ss:back} for a more detailed account of the related literature).

To illustrate the idea, suppose that we want to estimate the probability that there is a transition layer contained within $[-\ell,\ell]$ for some $\ell$ large. (Transition layers are introduced in Definition~\ref{def:layer} above; roughly, they are layers connecting $\pm 1$.) The Markov property (Lemma~\ref{le:Markov1b}) implies that this probability can be written as
\begin{align}
\lefteqn{\mu_{\eps,(-\Le,\Le)}^{-1,1}  \big(\text{transition in} (-\ell,\ell)\big)} \notag\\
=&\int_{-\infty}^\infty \int_{-\infty}^\infty   \,   \nu(du_-, du_+) \, \mu_{\eps,(-2\ell,2\ell)}^{u_-,u_+}  \big( \text{transition in} (-\ell,\ell)    \big) \label{e:Outline1}.
\end{align}
Here $\nu$ denotes the marginal distribution of the pair $(u(-2\ell),u(2\ell))$, and $ \mu^{u_-,u_+}_{\eps,(-2\ell,2\ell)}$ denotes the distribution of paths on $(-2\ell,2\ell)$ with boundary conditions $u_{\pm}$ (see Section \ref{s:Prlm} for a precise definition of this measure).

In Subsection \ref{ss:LD} we establish large deviation estimates for the measures $ \mu_{\eps,(-2\ell,2\ell)}^{u_-,u_+}$ that hold locally uniformly in the boundary values $u_\pm$. Hence for $u_{\pm}$ in some large compact set, we can integrate over these bounds in~\eqref{e:Outline1}. On the other hand, the probability that the boundary values $u_\pm$ fall outside of the compact set $[-M,M]$ for $M\gg 1$ decays exponentially with $M$ (see Lemma~\ref{le:onept} below).

For  boundary values within the compact set $[-M,M]$,  large deviation theory gives the uniform estimate
\begin{align*}
 \mu_{\eps,(-2\ell,2\ell)}^{u_-,u_+}  \big( &\text{ transition in} (-\ell,\ell)    \big) \\
  &=  \exp  \Big(-\frac{1}{\eps}\big(  \DE(\text{transition}) + o(1) \big)  \Big).
\end{align*}
Here $ \DE(\text{transition})$ denotes the difference between the minimal energy of  paths that perform a transition in $(-\ell,\ell)$ and the minimal energy of any profile $u$ that satisfies the boundary conditions $u(\pm 2\ell) = u_{\pm}$. (See Subsection \ref{ss:LD} for a more complete discussion.)

Now we arrive at the second problem, which is more subtle. The issue is that the energy difference $\DE(\text{transition})$ depends strongly on the boundary conditions. The cost that we are expecting to recover is $c_0$, defined in~\eqref{c00}. However, if $u_-\approx -1$ and $u_+\approx 1$, for instance, then the energy \emph{difference} is approximately zero! In this case, the information about the probability of a transition is encoded in the distribution $\nu$.

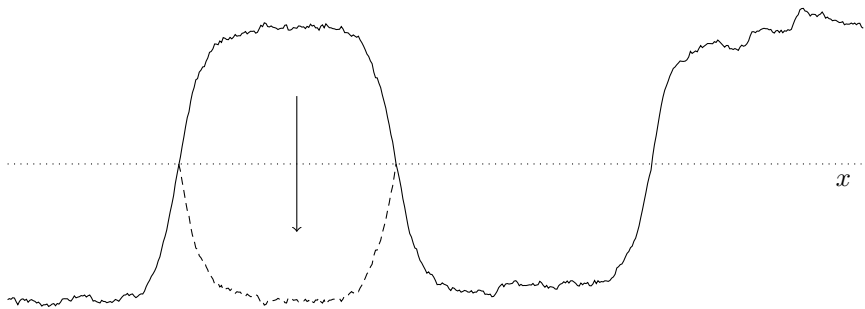
\begin{figure}
\centering
\begin{tikzpicture}[xscale=0.0225, yscale=1.8]

\draw[black, style=dotted]  (1,0) -- (500,0) node[anchor=north east]{\small $x$};

\draw[black] plot file{./Graphs/Plot1.txt};
\draw[black, style = densely dashed] plot file{./Graphs/Plot2.txt};

\draw[->] (170,.5) --(170,-.5)   ;

\end{tikzpicture}
\caption{A vertical reflection turns a  transition layer into a ``wasted excursion'' in which (roughly speaking) the path goes from $-1$ to $0$ and then back to $-1$. The probability of a wasted excursion on $(-\ell,\ell)$ is approximately \emph{independent} of the boundary conditions at $\pm 2\ell$. }
\label{fig:1}
\end{figure}

Our idea to handle the problem of dependence on the boundary conditions relies on Markovianity and the global symmetries of $\mu_{\eps, (-\Le,\Le)}^{-1,1}$. What we want to do is to transform a transition event into an event that does not feel the influence of the boundary conditions. Roughly, the new event will be that there are points $x<y<z\in(-\ell,\ell)$ such that $u(x)\approx u(z)\approx -1$ while $u(y)= 0$. (See Figure \ref{fig:1} for an illustration and Definitions~\ref{def:wasted} and~\ref{def:wastedplus} for formal definitions of these ``wasted excursions.'') The expected cost for such an event is also $c_0$, and a little thought reveals that this should be the energy difference \emph{regardless of the boundary conditions at $\pm 2\ell$}. (For a result in this direction, see Lemma~\ref{l:cl}.)

In order to transform transitions into wasted excursions, we use the strong Markov property (see Lemma~\ref{p:Markov}) and the symmetry of $V$. Specifically, we reflect paths vertically between certain hitting points of zero in such a way that leaves $\mu_{\eps, (-\Le,\Le)}^{-1,1}$ invariant. For details, see for instance~\eqref{r} and the subsequent calculations in the proof of Theorem~\ref{t:layers}.

A different reflection operator turns out to be useful when we come to the proof of the uniform distribution of the layer location in Theorem~\ref{t:uniform}. Again the Markovianity and the symmetry of $\mu_{\eps, (-\Le,\Le)}^{-1,1}$ are crucial. Here the rough idea is to show that the probability of finding the transition layer in any interval $[y-\de,y+\de]$ is approximately the same as that of finding the layer in any other interval $[z-\de,z+\de]$. In Section~\ref{s:Uniform}, we construct a measure-preserving reflection operator that transforms paths with a transition in $[y-\de,y+\de]$ into paths with a transition in (or near) $[z-\de,z+\de]$. We build this reflection operator using certain hitting points of $-1$ and $+1$ to the left and right of the transition layer. (This is illustrated in Figure \ref{fig:2}.) Hence a key point is to prove that, on the set of paths with a transition in $[y-\de,y+\de]$, such hitting points exist with high probability. This fact is developed in Lemmas~\ref{l:smallu}
and~\ref{l:hittingzero} using an iterated rescaling argument and large deviation bounds.

\begin{figure}
\centering
\begin{tikzpicture}[xscale=0.025, yscale=2, >=stealth]

\draw[black, style=dotted] (1,0) -- (500,0) node[anchor=north east]{\small $x$};

\draw[black, ultra thick] (60,0) -- (120,0);
\draw[black, ultra thick] (60,-.05) -- (60,0.05);
\draw[black, ultra thick] (120,-.05) -- (120,0.05);
\draw (90,0.05) node[above] {$J_{1}$};

\draw[black, ultra thick] (330,0) -- (390,0);
\draw[black, ultra thick] (330,-.05) -- (330,0.05);
\draw[black, ultra thick] (390,-.05) -- (390,0.05);
\draw (360,0.05) node[above] {$J_{2}$};

\draw[black] plot file{./Graphs/P2G1.txt};
\draw[black, style = densely dashed] plot file {./Graphs/P2G2.txt};


 \draw[ ->] (250,.25) arc [start angle= 30,   end angle=290,
x radius=1120pt, y radius=14 pt ] ;

\end{tikzpicture}
\caption{A point reflection between a hitting point of $-1$ and a hitting point of $+1$ moves the transition from the interval $J_1$ into the interval $J_2$. As the point reflection preserves the measure, both events have the same probability.}
\label{fig:2}
\end{figure}
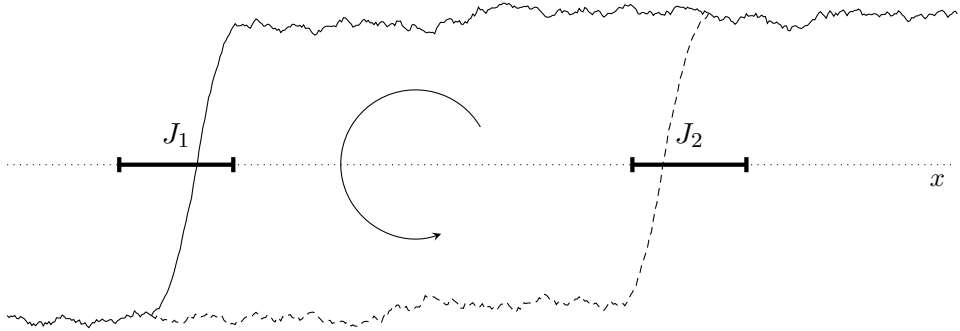

\subsection{Background literature and related results}\label{ss:back}

The study of the effect of a small noise on a physical system has a rich history in the chemistry, physics, and mathematics literature. With roots in the fluctuation theory of Einstein~\cite{einstein} (1910), the path integral formulations of Wiener~\cite{wiener} (1930) and Feynman~\cite{feynman} (1948) lie at the heart of the large deviation theory for diffusion processes and the characterization of the corresponding invariant measure. One of the aspects to receive the most applied interest and significant mathematical attention is the question of the first exit time from a metastable basin. The exponential dependence of the mean exit time on the energy barrier goes back to Van't Hoff and Arrhenius~\cite{VH,arrhenius} (1889). Refining this picture, the so-called Kramers formula determines the prefactor in terms of the curvature of the potential at the critical points and was made famous in the 1940 paper by Kramers~\cite{kramers}, although the result (for the overdamped dynamics) had been derived as early
as 1927 by Farkas~\cite{farkas}. See the review paper by H\"anggi, Talkner, and Borkovec~\cite{HTB} for a thorough historical survey.
The higher dimensional case was analyzed by Landauer \& Swanson~\cite{LS} in 1961 and further pursued by Langer (see for instance~\cite{L}, 1969).

In the mathematics literature, metastability for diffusion processes that depend only on time (i.e., constant in space) was explored early on in the paper by Pontryagin, Andronov, and Vitt~\cite{PAV} (1933). The mathematical theory of large deviations was subsequently developed in the 1970s in papers by Wentzell and Freidlin (see for instance~\cite{wf_early}) and Kifer~\cite{kifer}, and a landmark text is the book of Freidlin-Wentzell~\cite{FW} (published in Russian in 1979 and first published in English in 1984). On the level of the mean exit time, the Freidlin-Wentzell theory confirmed the exponential factor in the Kramers formula. The prefactor in Kramers' law for $d>1$ was established via formal asymptotic expansions in the famous paper by Matkowsky and Schuss~\cite{MS} in 1977. A rigorous derivation was given  by Sugiura in \cite{S95} and independently and with a different method by  Bovier, Eckhoff, Gayrard, and Klein~\cite{BEGK04,BGK05}.

The small noise problem for stochastic \emph{partial} differential equations appears more recently in the mathematics community. A seminal paper in extending the Freidlin-Wentzell theory to spatially varying diffusions is the paper of Faris and Jona-Lasinio~\cite{FJ} from 1982, which specifically established and studied the action functional of the stochastic Allen-Cahn differential equation on a bounded system $[0,L]$. The invariant measure of stochastically perturbed reaction diffusion systems (including the Allen-Cahn equation) on a bounded domain is studied by Freidlin in~\cite{F} in 1988. Recently, Barret, Bovier, and M\'el\'eard \cite{BBM, B12} and  Berglund and Gentz~\cite{BG} have established the mean exit time estimate including the prefactor for a class of equations including the Allen-Cahn equation.

As we have emphasized in the beginning of the introduction, in this paper we are concerned with the interplay between \emph{small noise} and \emph{large domain size}. Specifically, we are interested in system sizes that are exponential with respect to the inverse noise strength. Before turning to the invariant measure for unbounded systems, we remark that there is already an \emph{entropic, system-size dependent component} of the mean switching time when there is a ``flat'' or ``degenerate'' saddle point, e.g., for the Allen-Cahn equation in the periodic case. Specifically, the prefactor picks up a factor that is proportional to the volume of the degenerate set. This fact was observed already by Glasstone, Laidler, and Eyring~\cite{GLE} (1941) in the context of transition state theory, and the estimates in the setting of overdamped diffusions were developed by Langer~\cite{L} (1969) and Matkowsky \& Schuss~\cite{MS} (1977). See also~\cite{VW} for an independent, also formal, derivation. %

The dynamics of the stochastic Allen-Cahn equation~\eqref{e:ACE} have been considered by several authors. In particular, in the groundbreaking works of Funaki \cite{Fu95}  and Brassesco, De Masi, and Presutti~\cite{P95}, the dynamics of very similar  equations were studied.  In \cite{Fu95},  the equation~\eqref{e:ACE} is considered on the whole line with boundary conditions that enforce one transition. The noise term $\sqrt{2 \eps} \eta$ is multiplied by a function with compact support. In terms of our notation, the noise acts on an interval $\Le$ of length polynomial in $\eps^{-1}$. In \cite{P95}, the equation~\eqref{e:ACE} is considered for $\Le = \eps^{-1}$ with Neumann boundary conditions. In both articles, the initial condition is chosen close to the optimal profile of a single transition, and it is shown that the solution stays close to an optimal profile on timescales that are polynomial in $\eps^{-1}$. The evolution of the midpoint of the transition layer is also characterized: In \cite{Fu95}, the
interface dynamic is given by a stochastic differential equation that reflects the spatially dependent noise strength.
In \cite{P95}, it is shown that the midpoint performs a Brownian motion. The dynamic behavior observed in both of these articles is consistent with our results on the invariant measure. In particular, the Brownian motion of interfaces is consistent with the uniform distribution of layer location that we observe in Theorem \ref{t:uniform}.

Now let us consider the interplay between small noise and large domain size.  The idea of understanding large deviation events on large spatial systems via a decomposition into subintervals (intermediate in size between the logarithmic and exponential scale) is used in the paper~\cite{VW} to heuristically derive the nucleation and propagation dynamics in the setting of an unequal-well potential. In rigorous work on the invariant measure for the equal-well case,
the second author derived a concentration result for the measures $\mu^{-1,1}_{\eps,(-\Le,\Le)}$ in~\cite{W10} for system sizes that are large but algebraically bounded: specifically, $\Le \leq \eps^{-\alpha}$ for $\alpha < 2/3$. The technique used there is completely different from the one employed in the present article, however. In~\cite{W10}, the measure is discretized to make rigorous the heuristic intuition that $  \mu^{-1,1}_{\eps,(-\Le,\Le)}$ is a Gibbs measure. Explicit bounds on the energy landscape and Gaussian concentration inequalities are then used to derive bounds on this discretized measure. This technique does not appear to be applicable for longer intervals because the discretization errors become too large.

In the articles~\cite{BBB} and~\cite{BBB08}, the special case of intervals growing like $\Le =  \frac{1}{4} |\log \eps|$ is studied. (The prefactor $1/4$ depends on a specific choice of double-well potential.) The article uses the fact
pointed out in~\cite{RVE05} that the measure  $\mu^{-1,1}_{\eps,(-\Le,\Le)}$ can be realized as the distribution of a diffusion process
\begin{equation}\label{e:SDE}
du(x) \, = \, a_\eps\big(u(x) \big) \, dx + \eps^{1/2} dw(x) \qquad u(-\Le)= -1,
\end{equation}
\emph{conditioned on} the event $u(\Le) = 1$. The drift term $a_\eps$ is the logarithmic derivative of the ground state of the Schr\"odinger operator $-\eps^2 \Delta +V$. (In most cases, the drift $a_\eps$ cannot be given explicitly.) This is the extension to bounded intervals of the well-known equivalence for the measure on the real line, cf. \cite{simon}.

Building on the connection between the invariant measure of the PDE and the process in~\eqref{e:SDE},~\cite{BBB08} derives a concentration result around the one-parameter family of energy minimizers. Furthermore, the authors characterize the asymptotic distribution of the position of the interfacial layer. It is nonuniform due to the energetic repulsion from the boundary of the interval. To see this nonuniformity, the moderate scaling  $\Le \approx | \log \eps|$ is necessary. Incidentally, this shows that our lower bound $d_\eps \gg |\log \eps|$ in Theorem~\ref{t:uniform} is optimal: Below the scale of $|\log \eps|$, nonuniformity occurs. Loosely speaking, the results in \cite{BBB08} and ours are complimentary. They obtain finer results on logarithmically large system sizes, we obtain coarser results on exponentially large system sizes.

Results  similar to (but different from) ours  were obtained in \cite{COP93} for a one-dimensional Ising model with ferromagnetic Kac potential. This is a spin model whose spins interact not only with their nearest neighbors, but with all spins in a given range. The authors study the limit in which this range diverges. This corresponds to the limit $\eps \downarrow 0$ that we investigate. Their main argument relies on a large deviation statement for the whole system in a \emph{local topology}. This large deviation result implies, for example, that the the local spin averages concentrate around $\pm 1$ and that probability to see a transition from $-1$ to $+1$ in any given compact interval is exponentially small. The exponential rate is given by the energetic cost of a transition (similar to the constant $c_0$ in this work). The significant difference between 
their large deviation bounds and ours is the dependence on the boundary condition. Their bounds state that the exponential decay of the probability of observing a certain behavior on an order-one interval is governed by the energy. We only get bounds for the measures conditioned on the boundary values on that interval. The difference is easy to appreciate on the level of the results. As mentioned, the probability of seeing a transition on a given order-one subinterval in their setting is exponentially small, while---because of our boundary conditions---a similar statement cannot possibly hold in our case:
Indeed, if it were to hold, we could sum over order-one subintervals and deduce that the probability to see a transition in the full system goes to zero with the noise, while in fact it is identically equal to one.

Finally, let us touch on the appearance of measures similar to   $\mu^{-1,1}_{\eps,(-\Le,\Le)}$  in the study of  Schr\"odinger operators. The Feynman-Kac formula gives a way to solve the imaginary time Schr\"odinger equation  (i.e., the heat equation with a potential) in terms of measures that are absolutely continuous with respect to Wiener measure. In this context, our model is often referred to as the $\phi^4_1$ model and the limit $\eps \downarrow 0$ corresponds to the \emph{semiclassical limit} in which the Planck constant $\hbar$ is sent to zero. Lemma \ref{le:onept}, for instance, is closely related (but not equivalent to) a statement about the decay of the ground state for the Schr\"odinger operator $\eps^2 \Delta + V$ as $\eps \downarrow 0$.

\subsection{Organization}
We begin with preliminaries:  In Section~2 we collect some properties of the energy functional, and in Section~3 we collect some probabilistic properties of  $\mu_{\eps,(-\Le,\Le)}^{-1,1}$ and of the underlying Gaussian measures.  With these preliminaries in hand, we turn in Section~4 to the proof of our first result, Theorem~\ref{t:layers}. In Section~5 we prove Theorem~\ref{t:uniform}, the uniform distribution of the layer location.  Finally, in Section~6 we prove the various technical lemmas that have been used in support of the main theorems.

\section{Deterministic preliminaries}
\label{s:detac}

In this section we discuss some more details about the energy functional $E$ (cf.~\eqref{energy}).  Our goal is to familiarize the reader with the common intuition about this energy, as well as to present some facts that will guide our method and appear later in proofs.

As described above, the potential term in the energy favors the states $\pm 1$ and the gradient term in the energy leads to an energetic cost for transitions between these states.  Given our large system and the boundary conditions~\eqref{bcj}, it is natural to consider the problem
\begin{align*}
\inf\{E_{(-\infty,\infty)}(u)\;:\; u(\pm\infty)=\pm 1\}.
\end{align*}
As we mentioned, the minimum cost $c_0$ can be calculated explicitly (cf.~\eqref{c0}). The calculations underlying this fact appear repeatedly in the proofs of our energy lemmas, so we begin by recalling them. The so-called Modica-Mortola trick (cf. \cite{MM}) uses the elementary inequality $a^2+b^2\geq 2ab$ to observe:
\begin{align*}
&\inf\{E_{(-\infty,\infty)}(u)\colon u(\pm \infty)=\pm 1\}\\
&=\inf\left\{\int_{-\infty}^{\infty}\left(\frac{1}{2}(\partial_x u)^2+V(u)\right)\,dx\colon u(\pm\infty)=\pm 1\right\}\\
&\geq \inf\left\{\int_{-\infty}^{\infty} \sqrt{2V(u)}(\partial_x u)\,dx\colon u(\pm\infty)=\pm 1\right\}\\
&=\int_{-1}^{1}\sqrt{2V(u)}\,dx,
\end{align*}
which gives a lower bound on the energetic cost. For the matching upper bound, one observes that the equality $a^2+b^2=2ab$ holds if and only if $a=b$, so that the minimum energetic cost is achieved precisely when
\begin{align}
|\partial _x u|=\sqrt{2V(u)}.\label{absode}
\end{align}
For our boundary conditions, it is easy to see that the minimum is achieved for the strictly increasing function that satisfies
\begin{align}
 \partial_x u=\sqrt{2V(u)}. \label{ode}
\end{align}
We denote by $m$ the minimizer that is normalized so that $m(0)=0$. This function $m$ is then the unique, centered, stationary solution of the Allen-Cahn equation on $\R$ subject to the given boundary conditions, i.e., the solution of
\begin{equation*}
\partial_x^2 m - V'(m) = 0 \qquad m(0)=0 \qquad \text{and} \quad m(\pm \infty) = \pm 1.
\end{equation*}
In the case of the standard double-well potential $V(u) = (1 - u^2)^2/4 $, one has $m(x) = \tanh(x/ \sqrt{2})$.

For general potentials satisfying Assumption~\ref{ass:V}, the energy minimizer has similar qualitative properties to the hyperbolic tangent.  In particular, what will be important for us is that the minimizer converges exponentially to $\pm 1$ as $x\to\pm\infty$.
\begin{lemma}[Exponential decay of minimizer]\label{l:expdecmin}
Under Assumption~\ref{ass:V} on the potential $V$, there exists $C<\infty$ such that the global energy minimizer $m$ satisfies
\begin{equation*}
| m(x) - \sign(x) |  \leq  C\,\exp\left(-\sqrt{\frac{V''(1)}{2}}\;x\right).
\end{equation*}
\end{lemma}
The exponential convergence to $\pm 1$ follows directly from~\eqref{ode} and  the quadratic behavior of $V$ near the minima (cf., Assumption~\ref{ass:V}).

In addition to the exponential convergence to $\pm 1$, we see from~\eqref{ode} and Assumption~\ref{ass:V} that outside of a neighborhood of $\pm 1$, the slope of $m$ is bounded away from zero. Consequently, there is a characteristic length-scale associated to a transition layer.  We will use this length-scale in an essential way.  That is, since we cannot apply large deviation theory on the full system scale $\Le$, we will decompose into subsystems of bounded size, typically called $2\ell$ or $4 \ell$.  We will choose the subsystem size so that (with very large probability) a typical transition layer fits inside, which requires $\ell$ to be large.  In order to make these ideas precise, we begin by introducing the idea of a $\delta^-$~transition layer. Simply put, instead of connecting $\pm 1$, it connects $-1+\delta$ with $1-\delta$.
\begin{definition}[$\delta^-$~transition layer]\label{def:layerA}
Fix $\delta \in (0,1/2)$ and suppose $x_-<x_+$.  We say that $u$ has a $\delta^-$ up transition layer between $x_-$ and $x_+$ if
$$ u(x_{\pm})= \pm  (1-\delta) \qquad\text{and}\qquad  |u(x)|< 1-\delta\;\;\text{for all}\;x\in(x_-,x_+).$$

We say that $u$ has a $\delta^-$ down transition layer on $(x_-,x_+)$ if the same condition holds true with signs reversed, and that $u$ has a $\delta^-$ transition layer if it has a $\delta^-$ up or a $\delta^-$ down transition layer.
\end{definition}

\medskip

Since it is of course true that
\begin{align*}
&\mu_{\eps,(-\Le,\Le)}^{-1,1} \big(  \text{ $u$ has $(2n+1)$ transition layers  } \big)\\ &\;\;\leq \mu_{\eps,(-\Le,\Le)}^{-1,1} \big(  \text{ $u$ has $(2n+1)$ $\delta^-$~transition layers  } \big),
\end{align*}
the proof of the upper bound in Theorem~\ref{t:layers} will be established if we can show that for any $\gamma>0$ and for sufficiently small $\delta>0$, there is an $\eps_0>0$ such that, for all $\eps\leq\eps_0$, we have
\begin{align}
\lefteqn{\mu_{\eps,(-\Le,\Le)}^{-1,1} \big(  \text{ $u$ has $(2n+1)$ $\delta^-$~transition layers  } \big)}\notag\\
 &\lesssim (\Le)^{2n}\,\exp\left(-\frac{ 2n c_0-\gamma}{\eps}\right).\label{fras}
\end{align}
The main ingredient for establishing~\eqref{fras} is the uniform large deviation estimate from Proposition~\ref{pr:LD1}, below, which essentially reduces the problem to one of \emph{energy estimates}. We will control the energy of suitable classes of functions up to a small $\delta$-dependence and ultimately absorb this error term into the large deviation error $\gamma$ from the proposition.

One of the first steps will be to understand the length-scale associated to $\delta^-$ transition layers.  For any $\delta\in (0,1/2)$, the optimal transition layer captured by the energy minimizer $m$ goes from $-1+\delta$ to $1-\delta$ over a finite length-scale, and ``typical layers'' perform the transition on a similar length-scale.
A question that we will have to address is how likely it is for a transition to  take unusually long to complete a $\delta^-$~transition. In the following lemma, we show that the difference of energies expressed in Proposition~\ref{pr:LD1} is large for functions that perform unusually long transitions (uniformly with respect to the boundary values).
\begin{lemma}[Long transitions]\label{le:lazy}
There exists a $C_1 < \infty$ (depending only on $V$) such that, for any $M<\infty$ and any $\delta \in (0,1/2)$, there exists an $\ell_* <\infty$ with the following property. For any $\ell \geq \ell_*$ and $u_\pm\in [-M,M]$, set
\begin{align*}
\mathcal{A}^{\rm bc}&:=\{u\in C([-2\ell, 2\ell]) \colon u(-2\ell)=u_{-}  \, \text{and } u(2 \ell) = u_+   \},\\
\mathcal{A}_0^{\rm bc}&:=\{u\in \mathcal{A}^{\rm bc}  \colon  \text{for all $x \in [-\ell, \ell ]$, }u(x) \in [-1+ \delta, 1-\delta ]     \}.
\end{align*}
Then we have
\begin{equation}\label{e:lazy}
\inf_{u \in \mathcal{A}_0^{\rm bc} } E_{(-2\ell, 2 \ell)}(u) - \inf_{u \in \mathcal{A}^{\rm bc} }  E_{(-2\ell, 2 \ell)}(u) \geq\frac{2\delta^2\,\ell}{C_1}.
\end{equation}
\end{lemma}
The proof of Lemma~\ref{le:lazy} is given in Subsection~\ref{ss:enlem}. This lemma together with the large deviation bound from Proposition~\ref{pr:LD1} will imply that for $\gamma$ small with respect to $\delta^2\ell$, the probability of finding such a layer is bounded above by
$$\exp\left(-\frac{2\delta^2\ell/C_1-\gamma}{\eps}\right)\leq \exp\left(-\frac{\delta^2\ell}{C_1\eps}\right),$$
which we can make negligible by choosing $\ell$ sufficiently large.

Now we would like to show that the exponential factor in the probability of finding a $\delta^-$ layer is close to $c_0$, defined in~\eqref{c00}.  Specifically, we expect it to be approximately
$$\int_{-1+\delta}^{1-\delta}\sqrt{2V(s)}\,ds.$$
The problem, which we already alluded to at the end of Subsection~\ref{ss:methods}, is that the boundary values (for instance $u(-2\ell)\approx -1$, $u(2\ell)\approx 1$) may make it likely to find a layer.  Hence, we will employ reflection operators to transform $\delta^-$~transition layers into events that are
unlikely \emph{regardless of the boundary conditions}.  We will call such events wasted $\delta^-$ excursions:
\begin{definition}[Wasted $\delta^-$ excursion]\label{def:wasted}
For any $\delta\in (0,1/2)$, we will say that $u$ has a wasted $\delta^-$ excursion on $(-\ell,\ell)$ if there exist points
$$-\ell\leq x_-<x_0<x_+\leq \ell$$
such that
$$|u(x_0)|\leq \delta $$
and
$$\text{either}\quad |u(x_\pm)-1|\leq \delta\qquad\text{or}\qquad |u(x_\pm)+1|\leq\delta.$$
\end{definition}
As described above for long transitions, we will estimate the probability of such events using the large deviation estimate from Proposition~\ref{pr:LD1}.  We note that the proposition requires minimizing energy over a ball (in the space of continuous functions) around the set of interest. Because of the way we have defined wasted excursions, a ball of radius $\delta$ around the set of functions with a $\delta^-$ excursion in a given interval is equal to the set of functions with a $(2\delta)^-$ excursion in that interval. Hence, our large deviation estimate together with an energetic estimate will bound the probability that we are after. The following lemma contains the necessary energetic estimate: namely, that the difference of energies described in our large deviation estimate is bounded below by $c_0$ plus a small term.
\begin{lemma}\label{l:cl}
There exists a constant $C< \infty$ such that for every $M<\infty$ and $\delta\in(0,1/2)$, there exists a constant $\ell_*<\infty$ with the following property. For any  $\ell\geq \ell_*$ and any boundary conditions $u_{\pm}\in[-M,M]$, set
\begin{align*}
\mathcal{A}^{\rm bc}&:=\{u\in C([-2\ell,2\ell])\colon u(\pm 2\ell)=u_{\pm}\},\\
\mathcal{A}_0^{\rm bc}&:=\{u\in \mathcal{A}^{\rm bc}\colon u\,\text{has a wasted $\delta^-$ excursion in } (-\ell,\ell)\}.
\end{align*}
Define the optimal cost
\begin{align}
c_\ell:=\inf_{\mathcal{A}_0^{\rm bc}}E_{(-2\ell,2\ell)}(u)
-\inf_{\mathcal{A}^{\rm bc}}E_{(-2\ell,2\ell)}(u).\label{mmm.1}
\end{align}
Then we have
\begin{align}
c_\ell-c_0\geq -C\,\delta.\label{fdy}
\end{align}
\end{lemma}
The proof of Lemma~\ref{l:cl} is given in Subsection~\ref{ss:enlem}.  It gives us the exponential factor in the desired estimate~\eqref{fras}, above.

\medskip

For the lower bound in Theorem~\ref{t:layers}, we will work  with so-called $\delta^+$ transition layers between $-1-\delta$ and $1+\delta$.
\begin{definition}[$\delta^+$~transition layer]\label{def:layerB}
Fix $\delta \in (0,1/2)$.  We say that $u$ has a $\delta^+$ up transition layer within the interval $(-\ell,\ell)$ if there exist points
$$-\ell\leq x_-< x_+\leq\ell$$
such that
\begin{align*}
&u(x_\pm)=\pm (1+\delta).
\end{align*}
We say that $u$ has a $\delta^+$ down transition layer on $(-\ell,\ell)$ if the same condition holds true with signs reversed, and that $u$ has a $\delta^+$ transition layer if it has a $\delta^+$ up or a $\delta^+$ down transition layer.
\end{definition}

In analogy with the $\delta^-$~transition layers that we use for the upper bound, $\delta^+$~transition layers will be convenient for the lower bound. Since the probability of having $(2n+1)$~transition layers is greater than the probability of having $(2n+1)$~$\delta^+$~transition layers, it will  suffice to show that
\begin{align*}
\mu_{\eps,(-\Le,\Le)}^{-1,1}& \big(  \text{ $u$ has $(2n+1)$ $\delta^+$ transition layers  } \big)\\
&\;\;\gtrsim (\Le)^{2n}\,\exp\left(-\frac{ 2n c_0 -\gamma}{\eps}\right).
\end{align*}
We will establish this bound by reflecting in order to transform the $\delta^+$ transition layers into some kind of ``wasted excursions'' whose probability we can bound, independently of the boundary conditions.
\begin{definition}[Wasted $\delta^+$ excursion]\label{def:wastedplus}
For any $\delta\in (0,1/2)$, we will say that $u$ has a wasted $\delta^+$ excursion on $(-\ell,\ell)$ if there exist points
$$-\ell\leq x_-<x_0<x_+\leq \ell$$
such that
\begin{align*}
&u(x_\pm)\leq -1-\delta,\quad u(x_0)=0.
\end{align*}
\end{definition}
(We will use only the wasted $\delta^+$ excursions that come from below, but of course it would be straightforward to define the analogue with $u(x_\pm)\geq 1+\delta$, and they would obey the same energetic and probabilistic bounds.)

As in the case of the upper bound, we need an energetic lemma that will control the contribution to the large deviation estimate for wasted $\delta^+$ excursions. Because of the form of the large deviation estimate that we will develop in Section~\ref{s:Prlm} (see Proposition~\ref{pr:LD2} below), it will be convenient for us to introduce the energy bound on the following set of functions:
\begin{align}
\mathcal{A}_{\delta,pre}^{\rm bc}:=\Big\{&u\in\mathcal{A}^{\rm bc}\colon \text{ there exist points }-\ell\leq x_-< x_0<x_+\leq \ell\notag\\
&\text{with }u(x_-)\leq -1-2\delta,\, u(x_+)\leq -1-2\delta,\,u(x_0)\geq\delta\Big\}.\label{preset}
\end{align}
It is easy to see that  a $\delta$ ball (with respect to the $\sup$ norm) around $\mathcal{A}_{\delta,pre}^{\rm bc}$ is equal to the set of functions with wasted $\delta^+$ excursions on $(-\ell,\ell)$. This fact is what will later be useful for the lower bound. For now, we record the following energetic fact, which plays the role for the lower bound that Lemma~\ref{l:cl} played for the upper bound.
\begin{lemma}\label{l:cll}
There exists a constant $C< \infty$ such that for every $M<\infty$ and $\delta \in (0,1/2)$, there exists a constant $\ell_*<\infty$ with the following property. For any $\ell\geq \ell_*$  and $u_\pm\in[-M,0]$, set
\begin{align*}
&\mathcal{A}^{\rm bc}:=\{u\in C([-2\ell,2\ell])\colon u(\pm 2\ell)=u_{\pm}\}\\
\text{and}\quad &\mathcal{A}_{\delta,pre}^{\rm bc}\text{ as above in}~\eqref{preset}.
\end{align*}
Define the optimal cost
\begin{align*}
c_\ell:=\inf_{\mathcal{A}_{\delta,pre}^{\rm bc}}E_{(-2\ell,2\ell)}(u)
-\inf_{\mathcal{A}^{\rm bc}}E_{(-2\ell,2\ell)}(u).
\end{align*}
Then we have
\begin{align*}
c_\ell-c_0\leq C\, \delta.
\end{align*}
\end{lemma}

We will need to consider some additional properties of the energy as we prove the main theorems, but we defer their discussion to a later time when their motivation and hypotheses will be clearer.  With the central facts about the energy in hand, we now turn to the  probabilistic background for our paper.

\section{Probabilistic preliminaries}
\label{s:Prlm}

In this section, we collect some probabilistic facts about the Gaussian measures $ \mathcal{W}^{u_{-}, u_+}_{\eps, (x_{-},x_+)}$ and the measures $  \mu^{u_-,u_+}_{\eps,(x_-,x_+)}$. After stating a precise definition and some elementary symmetry properties, we will discuss Markov properties satisfied by these measures in Subsection \ref{ss:GM} and large deviation bounds in Subsection \ref{ss:LD}.

For every $x_- < x_+$, we denote by $  \mathcal{W}^{0,0}_{\eps, (x_{-},x_+)}$ the distribution of a Brownian bridge with homogeneous boundary conditions on $[x_-,x_+]$ whose variance is proportional to $\eps$. To be more precise, $  \mathcal{W}^{0,0}_{\eps, (x_{-},x_+)}$ is the unique centered Gaussian measure on the space of continuous functions $C([x_-, x_+])$ such that, for all $x_1, x_2 \in [ x_- , x_+]$, one has
\begin{align}
 \mathbb{E}_{(x_{-},x_{+})}^{\mathcal{W}_\eps,0,0}& \Big( u(x_1) \, u(x_2)  \Big) \notag\\
  &= \, \frac{\eps}{x_+ - x_-}  \Big(  (x_1 - x_-)(x_+ - x_2)  \wedge (x_2  - x_-)(x_+ -x_1) \Big) .\label{e:Cov}
\end{align}

Equivalently, one can say that $  \mathcal{W}^{0,0}_{\eps, (x_{-},x_+)}$ is the centered Gaussian measure whose Cameron-Martin space is given by the Sobolev space $H^{1}_0 ([x_{-},x_+])$ with vanishing boundary conditions equipped with the homogeneous scalar product
\begin{equation*}
\frac{1}{\eps} \int_{x_{-}}^{x_+}  \partial_x u \, \partial_x  v  \, d x .
\end{equation*}

Indeed, the right-hand side of \eqref{e:Cov} is the Green's function for $\frac{1}{\eps}\partial_x^2$ with Dirichlet boundary conditions.

In the sequel, we often use the notation
\begin{equation}\label{e:I}
I_{x_-,x_+}(u)  : =  \frac{1}{2} \int_{x_{-}}^{x_+} \big(  \partial_x u   \big)^2 \,  d x
\end{equation}
to denote the Gaussian part of the energy of a function $u$ on the interval $(x_-,x_+)$.

It is common to think of $  \mathcal{W}^{0,0}_{\eps, (x_{-},x_+)}$ as a Gibbs measure
\begin{equation}\label{e:Fey}
 \mathcal{W}^{0,0}_{\eps, (x_{-},x_+)} \propto \exp \Big(  -\frac{1}{\eps} I_{x_-,x_+} (u) \Big) du
\end{equation}
with energy  $I_{x_-,x_+}$ and noise strength $\propto \eps$. Of course,~\eqref{e:Fey} does not make rigorous sense because there is no ``flat measure" $du$ on path space, and $I_{x_-,x_+} (u)$ is almost surely infinite under $  \mathcal{W}^{0,0}_{\eps, (x_{-},x_+)}$. The heuristic formula~\eqref{e:Fey} is motivated by finite dimensional approximations and it gives the right intuition for the large deviation bounds.

For more general boundary conditions $u_-,u_+ \in \R$, we can define  $\mathcal{W}^{u_{-},u_+}_{\eps, (x_{-},x_+)}$ as the image measure of $  \mathcal{W}^{0,0}_{\eps, (x_{-},x_+)}$ under the shift map
\begin{equation*}
u(x)  \mapsto u(x) + \hx(x),
\end{equation*}
where $h$ is the affine function interpolating the boundary conditions:
\begin{equation}\label{e:Defh}
 \hx(x) \,: = \,  \frac{x -x_-}{x_+ - x_-}u_+ +  \frac{x_+ -x}{x_+ - x_-}u_- .
\end{equation}
Similarly to~\eqref{m30.1}, for any choice of boundary condition $u_\pm$ and on any interval $(x_-,x_+)$, we denote by $  \mu^{u_-,u_+}_{\eps,(x_-,x_+)}$ the probability measure whose density with respect to $  \mathcal{W}^{u_{-},u_+}_{\eps, (x_{-},x_+)}$ can be expressed as
\begin{equation}\label{eq:Densmu}
\frac{d  \mu^{u_-,u_+}_{\eps,(x_-,x_+)}}{ d  \mathcal{W}^{u_{-},u_+}_{\eps, (x_{-},x_+)}}(u) =  \frac{1}{  \mathcal{Z}^{u_-,u_+}_{\eps, (x_-,x_+)}} \exp \bigg(-\frac{1}{\eps}  \int_{x_-}^{x_+}  V( u ) \, dx \bigg).
\end{equation}
Here we have introduced the notation
\begin{equation*}
 \mathcal{Z}^{u_-,u_+}_{\eps, (x_-,x_+)} :=   \mathbb{E}_{(x_{-},x_{+})}^{\mathcal{W}_\eps,u_{-},u_{+}} \Big( \exp\Big( - \frac{1}{\eps} \int_{x_-}^{x_+} V( u ) \, dx \Big) \Big)
\end{equation*}
for the normalization constant that ensures that $  \mu^{u_-,u_+}_{\eps,(x_-,x_+)}$ is indeed a probability measure.

As we have indicated in the introduction, there are symmetry properties of the measures $ \mathcal{W}^{u_{-},u_+}_{\eps, (x_{-},x_+)}$ and $  \mu^{u_-,u_+}_{\eps,(x_-,x_+)}$ that will play an important role in our argument. Observe for example that both $ \mathcal{W}^{0,0}_{\eps, (x_{-},x_+)}$ and $  \mu^{0,0}_{\eps,(x_-,x_+)}$ are invariant under the \emph{vertical reflection} $u \mapsto Ru$ and the \emph{horizontal reflection} $u \mapsto Su$ where
\begin{equation*}
R  u(x) \,: = \, - u(x) \quad \text{ and } \quad  S u(x) \,:=\,  u (x_+ + x_- -x).
\end{equation*}
Furthermore, the measures $  \mathcal{W}^{-1,1}_{\eps, (x_{-},x_+)}$ and $  \mu^{-1,1}_{\eps,(x_-,x_+)}$ are invariant under the \emph{point reflection} $ u \mapsto RSu$.

\subsection{Markov properties}
\label{ss:GM}

We first  present a two-sided version of the  Markov property for the measures $  \mathcal{W}^{u_{-},u_+}_{\eps, (x_{-},x_+)}$ and $  \mu^{u_-,u_+}_{\eps,(x_-,x_+)}$, which states that for any fixed points $x_- \leq \hat{x}_- < \hat{x}_+ \leq x_+$  and for $u$ distributed according to  to $  \mathcal{W}^{u_{-},u_+}_{\eps, (x_{-},x_+)}$ (or $  \mu^{u_-,u_+}_{\eps,(x_-,x_+)}$), the conditional distribution of $(u(x), x \in [\hat{x}_-,\hat{x}_+])$, given all the information about $u(x)$ for $x \in [x_-,x_+] \setminus (\hat{x}_-,\hat{x}_+)$, is $  \mathcal{W}^{u(\hat{x}_-)  ,u(\hat{x}_+)}_{\eps, (\hat x_{-},\hat x_+)}$ (or  $\mu^{u(\hat{x}_-)  ,u(\hat{x}_+)}_{\eps,(\hat x_-,\hat x_+)} $).  Then in Lemma~\ref{p:Markov}, we give the \emph{strong} Markov property, which states that the same statement holds true when the deterministic points $\hat{x}_{\pm}$  are replaced by left and right stopping points $\chi_\pm$. The proofs of these statements are quite standard. For completeness, we have included them in
Subsection \ref{ss:62}.

In the case of the measures $  \mathcal{W}^{u_{-},u_+}_{\eps, (x_{-},x_+)}$, the Markov property can be stated in the following way. For $\hat{x}_- < \hat{x}_+$, we define the piecewise linearization $u_{\hat{x}_-}^{\hat{x}_+}$ of $u$ between $\hat{x}_-$ and $\hat{x}_+$ as

\begin{equation}
u_{\hat{x}_-}^{ \hat{x}_+}(x) \, = \,
\begin{cases}
 h^{u(\hat{x}_-), u(\hat{x}_+)}_{(\hat{x}_-, \hat{x}_+)}(x) \qquad &\text{if } x \in (\hat x_-, \hat x_+) \\
u(x) \qquad &\text{else. }\label{linear}
\end{cases}
\end{equation}
Recall the definition~\eqref{e:Defh} of $\hx$. Then the following holds.

 \begin{lemma}\label{le:Markov1}
Suppose $x_- \leq \hat{x}_- < \hat{x}_+ \leq x_+$ are fixed, non-random points. Then under $  \mathcal{W}^{u_{-},u_+}_{\eps, (x_{-},x_+)}$ the random functions $u - u_{\hat{x}_-}^{\hat{x}_+}$ and $ u_{\hat{x}_-}^{\hat{x}_+}$ are independent. Furthermore, $u - u_{\hat{x}_-}^{\hat{x}_+}$ is zero outside of $(\hat{x}_-, \hat{x}_+)$ and is distributed according to $  \mathcal{W}^{0,0}_{\eps, (\hat{x}_{-},\hat{x}_+)}$ between the two points.
\end{lemma}
Due to the lack of spatial homogeneity, the corresponding property for the measures $  \mu^{u_-,u_+}_{\eps,(x_-,x_+)}$ has to be stated in a different way. For $I \subseteq [x_-,  x_+]$, we denote by $\F_{I}$ the sigma-algebra  generated by $u(x)$ for  $x \in I$, completed with respect to $  \mathcal{W}^{u_{-},u_+}_{\eps, (x_{-},x_+)}$.

We  also introduce the following notation that extends the measures  to paths on a larger domain by prescribing the values outside of an interval. Suppose that $[\hat{x}_-,\hat{x}_+] \subseteq [x_-,x_+]$ and that  ${\bf u} \in C([x_-,x_+])$ is a fixed path. We say that $u$ is distributed according to $  \mathcal{W}^{{\bf u}}_{\eps, (\hat x_{-},\hat x_+)}$, resp.  $  \mu^{{\bf u}}_{\eps,(\hat x_-,\hat x_+)}$,  if it  almost surely coincides with ${\bf u}$ outside of the interval $[\hat{x}_-,\hat{x}_+]$ and is distributed according to $  \mathcal{W}^{{\bf u}(\hat{x}_-) ,{\bf u}(\hat{x}_+) }_{\eps, (\hat x_{-},\hat x_+)}$, resp. $\mu^{{\bf u}(\hat{x}_-) ,{\bf u}(\hat{x}_+) }_{\eps, (\hat x_{-},\hat x_+)} $, on  $[\hat{x}_-,\hat{x}_+]$.

 Then the Markov property takes the following form.

\begin{lemma}\label{le:Markov1b}
Suppose $x_- \leq \hat{x}_- < \hat{x}_+ \leq x_+$ are fixed, non-random points. Then for any bounded measurable test function $\Phi \colon C([x_-,x_+] ) \to \R$, we get the following identity:
\begin{equation}\label{e:Markovmu}
 \mathbb{E}_{(x_{-},x_{+})}^{\mu_\eps,u_{-},u_{+}} \big( \Phi  \big|  \F_{[x_-,\hat{x}_-]}  \vee  \F_{[\hat{x}_+,x_+]  }  \big) \, = \,    \,  \mathbb{E}_{(\hat{x}_{-},\hat{x}_{+})}^{\mu_\eps,{\bf u}} \big( \Phi  \big) .
\end{equation}
\end{lemma}
Here $\F_{[x_-,\hat{x}_-]}  \vee  \F_{[\hat{x}_+,x_+]  }$ denotes the smallest sigma-algebra that contains all sets in  $\F_{[x_-,\hat{x}_-]}$ and $\F_{[\hat{x}_+,x_+]}$.

We will typically use~\eqref{e:Markovmu} in the following way: For given points $x_- \leq x_1 \leq x_2 \leq \ldots  \leq x_{2n} \leq x_+$ and given events $\mathcal{A}_i \in \F_{[x_{2i-1}, x_{2i}]}$, we can write
\begin{align}
 \mathbb{E}_{(x_{-},x_{+})}^{\mu_\eps,u_{-},u_{+}} & \Big( \mathbf{1}_{\Aa_1} \ldots \mathbf{1}_{\Aa_n} \Big) \,\label{e:Markomub}\\
  = \int_{\R^{2n}}&  \nu_{x_1, \ldots, x_{2n}}  (du_1, \ldots, du_{2n})  \, \mathbb{E}_{(x_{1},x_{2})}^{\mu_\eps,u_{1},u_{2}}  \Big( \mathbf{1}_{\Aa_1} \Big) \ldots \mathbb{E}_{(x_{2n-1},x_{2n})}^{\mu_\eps,u_{2n-1},u_{2n}}  \Big( \mathbf{1}_{\Aa_n} \Big) \notag.
\end{align}
Here $\nu_{x_1, \ldots, x_{2n}} $ denotes the distribution of the random vector $(u(x_1), \ldots, u(x_{2n}))$ under $  \mu^{-1,1}_{\eps,(x_-,x_+)}$. Formula~\eqref{e:Markomub} follows directly by applying~\eqref{e:Markovmu} $n$ times.

To state the strong Markov property, we additionally need the notion of left and right stopping points. These are defined analogously to stopping times for Markov processes. A random variable $\chi_-$ taking values in $[x_-, x_+]$ will be called a \emph{left stopping point} if for all $x \in [x_-,x_+]$ the event $\{ \chi_- \leq x \}$ is contained in $\F_{[x_-,x]}$. In the same way a random variable $\chi_+$ is called a \emph{right stopping point} if for all $x$ the event   $\{ \chi_+ \geq x \}$ is contained in  $\F_{[x,x_+]}$.  In all of our applications the stopping points $\chi_{\pm}$ are going to be left or rightmost hitting points of a closed set. It is easy to check that these random points are indeed left and right stopping points as defined above.

For given left and right stopping points $\chi_{\pm}$, we define the sigma-algebra $\F_{[x_+, \chi_-]}$ of events that occur left of $\chi_-$ and the sigma-algebra $\F_{[ \chi_+, x_+]}$ of events that happen to the right of $\chi_+$ by
\begin{align*}
\F_{[x_-,\chi_-]} \, :=\, & \big\{ \Aa \in \F_{[x_-,x_+]} \colon \forall x \quad \Aa \cap \{ \chi_- \leq x  \} \in \F_{[x_-,x]}  \big\},\\
\F_{[\chi_+, x_+]} \, :=\, & \big\{ \Aa \in \F_{[x_-,x_+]} \colon \forall x \quad  \Aa \cap \{ \chi_+ \geq x  \} \in \F_{[x,x_+]}  \big\}.
\end{align*}
The strong Markov property can be stated in an analogous way to~\eqref{e:Markovmu}.

\begin{lemma}\label{p:Markov}
Suppose $\chi_-$ and $\chi_+$ are left and right stopping points with $\chi_- < \chi_+$ almost surely. Suppose that $\Phi \colon C([x_-,x_+] ) \to \R$ is measurable and bounded. Then for any $u_\pm \in \R$,  we get the following identities
\begin{equation}\label{e:strMarkovW}
 \mathbb{E}_{(x_{-},x_{+})}^{\mathcal{W}_\eps,u_{-},u_{+}} \big( \Phi \big|   \F_{[x_-,\chi_-]} \vee \F_{[\chi_+, x_+]}   \big) \, = \,  \mathbb{E}_{(\chi_{-},\chi_{+})}^{\mathcal{W}_\eps,{\bf u}} \big( \Phi  \big)
\end{equation}
and
\begin{equation}\label{e:strMarkovmu}
 \mathbb{E}_{(x_{-},x_{+})}^{\mu_\eps,u_{-},u_{+}} \big( \Phi \big|   \F_{[x_-,\chi_-]} \vee \F_{[\chi_+, x_+]}   \big) \, = \,  \mathbb{E}_{(\chi_{-},\chi_{+})}^{\mu_\eps,{\bf u}} \big( \Phi  \big).
\end{equation}
\end{lemma}

The strong Markov property is a crucial ingredient in the proofs of  both Theorem \ref{t:layers} and Theorem \ref{t:uniform}. Let us illustrate how it is used in the proof of  Theorem \ref{t:layers}. Let $\chi_-$ be the leftmost hitting of zero to the right of a  given point $x_-$ and $\chi_+$ the  rightmost hitting of zero to the left of a given point $x_+$. The values $u(\chi_\pm)$ in the formulas~\eqref{e:strMarkovW} and~\eqref{e:strMarkovmu} are almost surely $0$. Then we can use the invariance of $  \mathcal{W}^{0,0}_{\eps, (\chi_{-},\chi_+)}$ and $  \mu^{0,0}_{\eps,(\chi_-,\chi_+)}$ under vertical reflection $R$ to conclude that the whole right-hand side of~\eqref{e:strMarkovW} and~\eqref{e:strMarkovmu} is invariant under   vertical reflection on $[\chi_-, \chi_+]$. In Section \ref{s:layers}, we will use this observation to reduce the problem of calculating the probability of transition layers to computing the probability of wasted excursions (see Definition  \ref{def:wasted}).

\subsection{Large deviations}
\label{ss:LD}

Large deviation estimates for the  measures $  \mu^{u_{-},u_+}_{\eps, (x_{-},x_+)}$ constitute an important ingredient for our argument. Large deviation bounds for Gaussian measures with a small variance, e.g., for $  \mathcal{W}^{u_{-},u_+}_{\eps, (x_{-},x_+)}$, are well-known (see e.g. \cite[Sec. 4.9]{Bog}). They can be extended to the measures $   \mu^{u_-,u_+}_{\eps,(x_-,x_+)}$ with an ``exponential tilting" argument (see e.g. \cite{Var84}, or~\cite[p.34]{dH00} ) in a standard way. Let $\mathcal{A}^{{\rm bc}}$ represent the set of continuous paths $u$ on $[x_-,x_+]$ that satisfy $u(x_\pm)=u_\pm$. The estimates then state that for every \emph{closed} set $\mathcal{A}\subseteq\mathcal{A}^{\rm bc}$ and every $\gamma>0$, there exists $\eps_0>0$ such that, for $\eps \leq \eps_0$, we have
\begin{equation}\label{e:LDclassical1}
 \mu^{u_-,u_+}_{\eps,(x_-,x_+)}(\Aa) \leq \exp  \Big(-\frac{1}{\eps} \big(  \DE(\Aa) - \gamma  \big)  \Big).
\end{equation}
Similarly, for every \emph{open} set $\Aa\subseteq\mathcal{A}^{\rm bc}$ and $\gamma>0$ there exists $\eps_0>0$ such that, for $\eps \leq \eps_0$, we have
\begin{equation}\label{e:LDclassical2}
 \mu^{u_-,u_+}_{\eps,(x_-,x_+)}(\Aa) \geq \exp  \Big(-\frac{1}{\eps} \big(  \DE(\Aa) + \gamma  \big)  \Big) .
\end{equation}
Here the energy difference $\DE(\Aa)$ is defined as
\begin{equation}
\DE (\Aa) := \inf_{u \in \Aa} E(u)  - \inf_{u \in \mathcal{A}^{\rm bc}} E(u).\label{endiff}
\end{equation}
Here and in the sequel, all topological notions like open and closed refer to the uniform topology, i.e., the topology generated by
\begin{equation}
\| u  \|_\infty := \sup_{x \in [x_-,x_+]} |u(x)|.
\end{equation}
Although we will not make use of it here, we remark that the bounds~\eqref{e:LDclassical1} and~\eqref{e:LDclassical2} are also true for different choices of topology. The Gaussian large deviation bounds hold for any separable Banach space that supports the Gaussian measure, and the  ``exponential tilting"  works as soon as the exponential density is continuous.

A priori, the choice of $\eps_0$ depends not only on $\gamma$ but also on the interval length $\ell:=  x_+ - x_-$, the boundary data $u_\pm$, and even the set $\Aa$ itself. As pointed out in Subsection \ref{ss:methods}, however, our argument requires integrating probabilities for different boundary conditions. Therefore, we need to know that we can choose the same $\eps_0$ for these different boundary conditions simultaneously. Moreover, in Lemma~\ref{l:smallu} we will need uniform estimates for measures with different potentials. Hence, we require uniform large deviation estimates, which is the content of the following two propositions.
They deliver local uniformity with respect to $\ell, u_\pm, \Aa$, and even with respect to the potential function $V$. To state the result, it is convenient to introduce the notation
\begin{equation}\label{e:Iu}
\Ixu := I(\hx) =    \frac{1}{2}\frac{(u_+ - u_-)^2}{x_+ - x_-}
\end{equation}
for the minimal Gaussian energy with the given boundary conditions. We will also  write
\begin{equation*}
B( \Aa, \delta)  = \big\{ u \colon \exists v \in \Aa, \, \| v -u  \|_\infty \leq \delta  \big\}
\end{equation*}
for the $\delta$ neighborhood of a set $\Aa$.

\begin{proposition}[Large deviation upper bound]\label{pr:LD1}
Fix constants $1 < M,R<\infty$ and $0 < \ell_-  < \ell_+ < \infty$.
For any $x_\pm\in\R$ with $x_+-x_-\in [\ell_-, \ell_+]$ and any $u_{\pm} \in [-M,M]$, let $\Aa$ be a measurable subset of $C([x_-,x_+])$ consisting of paths $u$ that satisfy the boundary conditions $u(x_{\pm}) = u_{\pm}$. Additionally, assume that
\begin{equation}\label{e:3.13}
\inf_{u \in \Aa} E(u) -  \Ixu \leq R.
\end{equation}

Then for any $\delta,\gamma >0 $  there exists an $\eps_0>0$ such that for all $\eps \leq \eps_0$ we have
\begin{equation}\label{e:LDUB1}
 \mu^{u_-,u_+}_{\eps,(x_-,x_+)}(\Aa) \leq \exp  \Big(-\frac{1}{\eps} \big(  \DE\big(B( \Aa,\delta)  \big) - \gamma  \big)  \Big),
\end{equation}
where $\DE$ is defined in~\eqref{endiff}. This $\eps_0$ depends on $M,R, \ell_\pm, \delta,$ and $\gamma$ but not on the particular choice of $x_\pm, u_\pm$. It only depends on the set $\Aa$ through the choice of $R$ in condition~\eqref{e:3.13}. Furthermore, $\eps_0$ depends on $V$ only through the local Lipschitz norm
\begin{equation*}
\sup_{|v| \leq M +\sqrt{2^{-1} (\ell_+ R +1)}+ 1} |V'(v) | .
\end{equation*}
\end{proposition}

In particular, the same bounds hold for the  same $\eps_0$ if $V$ varies over a set of potentials with  uniformly bounded local $C^1$-norm. This uniformity of~\eqref{e:LDUB1} with respect to $V$ will be used in Subsection \ref{ss:unilem}. There, it will be applied to the family  $\big\{ 4^k V( 2^{-k}(u-1) +1   ) \colon k \in \N \big\}$ of rescaled versions of $V$.

We also get the corresponding lower bounds without a condition on the minimal energy of $E(u)$ for $u \in \Aa$.

\begin{proposition}[Large deviation lower bound]\label{pr:LD2}
Fix constants $M$ and $0 < \ell_-  < \ell_+ < \infty$. Suppose that $\ell= x_+-x_- \in [\ell_-, \ell_+]$ and  $u_{\pm} \in [-M,M]$. Assume that there exists an energy minimizer
\begin{equation*}
u_\ast = \underset{u \in \Aa}{\rm argmin} \, E(u)
\end{equation*}
satisfying $u_\ast \in [-M,M]$. Then, for any $\gamma>0$ and $\delta>0$ small enough, there exists $\eps_0 >0$ such that for all $\eps \leq \eps_0$ there holds
\begin{equation}\label{e:LDLB1}
 \mu^{u_-,u_+}_{\eps,(x_-,x_+)} \big(B(\Aa,\delta)\big) \geq \exp  \Big(-\frac{1}{\eps} \big(  \DE \big( \Aa  \big) + \gamma  \big)  \Big),
\end{equation}
where $\DE$ is defined in~\eqref{endiff}. As above, $\eps_0$ depends on $M, \ell_\pm, \delta,$ and $\gamma$, but not on the particular choice of $x_\pm, u_\pm$ or the set $\Aa$.
\end{proposition}
 The same remark about the uniform dependence on $V$ holds for the lower bounds.
\begin{remark}\label{r:exminapprox}
The existence of energy minimizers $u_\ast$ in $\Aa$ satisfying $u_\ast \in [-M,M]$  is not necessary and it can be replaced by an approximation. Actually, we will show the Proposition under the slightly weaker assumption that for every $\gamma>0$ there exists a profile $u_\gamma \in \Aa$ with $uu_\gamma(x) \in [-M,M]$ for all $x \in [x_-,x_+]$ and such that
\begin{equation}\label{e:coneta}
E(u_\gamma) \leq \inf_{u \in \Aa} E(u) + \gamma.
\end{equation}
\end{remark}

The proofs of these Propositions are essentially a careful copy of the classical proofs and can be found in Subsection~\ref{ss:63}. Let us remark here that we do not expect the bounds~\eqref{e:LDclassical1} and~\eqref{e:LDclassical2} to hold uniformly for all open or closed sets. In fact, the argument for the classical statements makes use of qualitative properties such as existence of coverings by finitely many open sets. One sums over this finite number and uses the fact that, for $\eps$ small enough, only the largest summand matters. For different open or closed sets, this finite number will in general be different, and  the choice of $\eps_0$ would also be different . We can resolve this issue by taking the $\delta$ neighborhood of $\Aa$ in the bounds ~\eqref{e:LDUB1} and~\eqref{e:LDLB1} as a uniform version of the topological assumptions on $\Aa$.
%

\section{Proofs of Theorem~\ref{t:layers}: Domination by single transition layer of minimal energy}
\label{s:layers}

In this section we prove Theorem~\ref{t:layers}. 
This theorem estimates the exponentially small probability of having more than one layer (with the correct entropic effect and exponential factor).  Hence, the most likely functions are those with only one transition layer.

As outlined in Subsection~\ref{ss:methods}, at the heart of the method is the idea of  decomposing the invariant measure into conditional measures and the corresponding marginals, so that we can reduce to estimating the probability of transition layers on order-one subintervals.  When the boundary data of the subinterval falls within a compact set $[-M,M]$, large deviation theory will allow us to estimate probabilities in the spirit of
\begin{align*}
\int_{-M}^M \int_{-M}^M   \,   \nu(du_-, du_+) \;\mu_{\eps,(-2\ell,2\ell)}^{u_-,u_+}  \big(\text{there is a transition layer in }(-\ell,\ell)\big).
\end{align*}
On the other hand, the probability that $|u(\pm 2\ell)|\geq M$ is uniformly small.  Before turning to the proofs of the main theorems, we introduce this fact about the decay of the one-point distribution.
\begin{lemma}\label{le:onept}
There exist $M_1<\infty$, $C_2<\infty$ (depending only on $V$) such that the following holds.  For any $M\geq M_1$, there exists $\eps_0>0$ such that for all $\eps\leq \eps_0$ and any $x_0$ in $(-\Le,\Le)$ there holds
\begin{align}
\mu^{-1,1}_{\eps,(-\Le,\Le)}\Big(|u(x_0)|\geq M\Big)\,\leq \,\exp\left(-\frac{M}{\eps\,C_2}\right).\label{onept}
\end{align}
\end{lemma}
The proof of the lemma is given in Subsection~\ref{ss:onept}. With this preliminary estimate in hand, we turn now to the proof of Theorem~\ref{t:layers}. We consider separately the upper and lower bounds.

\bigskip

\begin{proof}[Proof of Theorem~\ref{t:layers}]
Fix $\gamma>0$. Fix a corresponding $\delta>0$ sufficiently small. Let $\ell$ and $M$ be large constants to be specified later.  To begin with, let $\ell$ be large enough so that~\eqref{e:lazy} and~\eqref{fdy} hold for the given $\delta$.  We will divide the system into  $2N_\eps$ intervals with
\begin{align}
N_\eps:=\left\lfloor\frac{\Le}{\ell}\right\rfloor,\label{ne}
\end{align}
labelling the endpoints:
\begin{align}
x_{\pm k}:=\begin{cases}\label{xk}
\pm k\,\ell,& k=0,\ldots, (N_\eps -1),\\
\pm \Le, & k=\pm N_\eps.
\end{cases}
\end{align}
We will work with this grid for the rest of this paper.

Then we consider the (overlapping) intervals
\begin{align}
I_k := [x_{k-1}, x_{k+1}],\qquad\text{for}\quad k=-(N_\eps-1),-(N_\eps-2),\ldots,(N_\eps-1).\label{ik}
\end{align}

Notice that $x_{\pm N_\eps}$ is separated from $x_{\pm(N_\eps-1)}$ by
up to length $2\ell$, while the rest of the points are separated by length $\ell$.  Since our energetic estimates will all hold uniformly for subsystems that are sufficiently large, and our large deviation bounds will all hold uniformly for subsystems whose length vary within a compact set, it will not matter that the boundary points may be up to $2\ell$ away from the neighboring points and we will ignore this issue for the rest of the proof.

\begin{center}
{\bf Upper bound.}
\end{center}
Here we will prove the upper bound, i.e., that
\begin{align*}
\mu_{\eps,(-\Le,\Le)}^{-1,1} \big( & \text{ $u$ has $(2n+1)$ transition layers  } \big) \\
&\lesssim (\Le)^{2n}\,\exp\left(-\frac{ 2n c_0-\gamma}{\eps}\right).
\end{align*}
As explained in Section~\ref{s:detac}, for the upper bound we will work with $\delta^-$~transition layers, and it will be sufficient to show that for any sufficiently small $\gamma>0$ and some sufficiently small $\delta>0$,  there is an $\eps_0>0$ such that for all $\eps\leq \eps_0$ we have
\begin{align}
\mu_{\eps,(-\Le,\Le)}^{-1,1} \big(&  \text{ $u$ has $(2n+1)$ $\delta^-$~transition layers  } \big)\notag\\
 &\lesssim (\Le)^{2n}\,\exp\left(-\frac{ 2n c_0-\gamma}{\eps}\right).\label{jn20.3}
\end{align}
Since the probability of transition layers is less than the probability of $\delta^-$ transition layers, the proof of the upper bound follows immediately.

The subtle part of the proof will be estimating the probability of a transition layer on a subsystem.  Recall from Subsection~\ref{ss:methods} that we cannot get the expected cost $c_0$ by estimating the probability
\begin{align*}
\mu_{\eps,(-2\ell,2\ell)}^{u_-,u_+}  \big(\text{ $u$ has a $\delta^-$~transition layer in }(-\ell,\ell)\big)
\end{align*}
because of the nontrivial dependence of this probability on the boundary conditions $u_\pm$.  To avoid this problem, we will use reflection operators to transform $\delta^-$~transition layers into wasted $\delta^-$ excursions (see Definition~\ref{def:wasted} and the accompanying discussion).

With this scheme in mind, let us now begin our estimates.

\medskip

\underline{Step 1}. Fix $\gamma>0$. Let $\delta>0$ be a small constant and $M<\infty$ be a large constant to be chosen below. Our first step will be to
 decompose the set of functions in which we are interested.  Namely, we notice that the set of continuous paths $u: [-\Le, \Le] \to \R$ satisfying the boundary conditions $u(\pm \Le) = \pm 1$ and exhibiting at least $(2n+1)$ $\delta^-$~transition layers is contained in the union of the following three sets:

\bul The set of paths that exhibit an atypically large value at one of the $x_k$:
\begin{align}
\mathcal{A}_1 := \Big\{ u  \colon \, \exists k \in \big\{ -(N_\eps -1) , \ldots , (N_\eps-1)  \big\} \colon |u(x_k)| \geq M \     \Big\}. \label{e:defA1}
\end{align}

\bul The complementary set intersected with the set of paths that are bounded away from $\pm 1$ on all of  $[x_k,x_{k+1}]$ for some $k$:
\begin{align}
\mathcal{A}_2 := \complement\mathcal{A}_1\cap \Big\{ u  \colon \,   \exists k \in \big\{ - & (N_\eps-1), \ldots , (N_\eps-2) \big\}: \label{e:defA2}\\
&  u\in [-1+\delta, 1 - \delta] \;\text{ on all of }  [x_k, x_{k+1}]    \Big\}. \notag
\end{align}

\bul The complement of $\mathcal{A}_1$ intersected with the set of paths performing $(2n+1)$ $\delta^-$ transitions, each of which is completely contained in (at least) one of the overlapping intervals  $I_k$.  We denote this set
\begin{align*}
\mathcal{A}_3 := \complement\mathcal{A}_1\cap \Big\{&   u \colon \text{there exist $2n$ integers} \\
&-N_\eps+1\leq k_1 \leq k_2 \leq \cdots \leq k_{2n}\leq N_\eps-1\text{ such that}\\
& \text{in each interval $I_{k_i}$   there is a $\delta^-$  layer}\Big\}.\\
\end{align*}
Note that there might be more than one layer in a single interval; the $2n$-tuple $(k_1, \ldots, k_{2n})$ allows for a possible higher multiplicity. There may also be \emph{more than} $2n$ layers; the statement is that there are \emph{at least} $2n$ layers.

Above, we have made use of the boundary conditions.  Indeed, for $\Aa_1$, we have omitted the points $x_{\pm N_\eps}$ since $u(\pm \Le)=\pm 1$.  For $\Aa_2$ we have omitted the boxes at the boundary since the boundary conditions make it impossible that $u(x) \in [-1+\delta, 1 - \delta] $ for all $x $ in the box.  For $\Aa_3$ we have recalled that the boundary conditions force there to be at least one transition. Even though $u$ has $2n+1$ layers, we can  expect an additional cost only for the $2n$ ``extra'' layers and hence only keep track of $2n$ layers.


Because the set of interest is contained within the above-mentioned sets, it suffices to bound
\begin{align}\label{e:decomposition}
 \mu_{\eps,(-\Le,\Le)}^{-1,1} \big(\mathcal{A}_1 \big)   +   \mu_{\eps,(-\Le,\Le)}^{-1,1} \big(\mathcal{A}_2  \big)  +  \mu_{\eps,(-\Le,\Le)}^{-1,1} \big(\mathcal{A}_3   \big).
\end{align}

 \medskip

 \underline{Step 2}.
 We first give a bound on the probability of $\mathcal{A}_1$. This bound follows directly from Lemma \ref{le:onept}. In fact, we get
 \begin{eqnarray}
 \mu_{\eps,(-\Le,\Le)}^{-1,1} \big(\mathcal{A}_1 \big) \,  & \leq&  \sum_{ k= -(N_\eps-1) }^{ N_\eps-1 }  \mu_{\eps,(-\Le,\Le)}^{-1,1} \big( |u(x_k)| \geq M  \big)\notag \\
 & \overset{\eqref{onept} }{\leq}&   \bigg( 2 \frac{\Le}{\ell} -1 \bigg)     \exp \bigg( - \frac{M}{\eps C_2} \bigg)\notag\\
 &\leq& \Le\exp \bigg( - \frac{M}{\eps C_2} \bigg).\label{sq}
 \end{eqnarray}
In particular, we can choose $M$ large enough so that $M/C_2\geq 2nc_0$ and
\begin{eqnarray*}
 \mu_{\eps,(-\Le,\Le)}^{-1,1} \big(\mathcal{A}_1 \big) \,
 &\leq& \Le\exp \bigg( - \frac{2nc_0}{\eps} \bigg).
 \end{eqnarray*}
Hence, the probability of $\Aa_1$ is of higher order with respect to the right-hand side of~\eqref{jn20.3}.

We remark that it is here where $M$ (and therefore also $\eps_0$) acquires a dependence on $n$.

\medskip

\underline{Step 3}.
To bound the second probability in~\eqref{e:decomposition}, we write
\begin{align}
\lefteqn{ \mu_{\eps,(-\Le,\Le)}^{-1,1}  \big(\mathcal{A}_2  \big)}  \notag\\
 \leq  & \sum_{ k= -(N_\eps-1) }^{ N_\eps -2}  \mu_{\eps,(-\Le,\Le)}^{-1,1} \Big(   u  \in
 [-1+\delta, 1 - \delta ] \text{ on all of }  [ x_k, x_{k+1} ] \notag\\
 & \qquad \qquad \qquad \qquad \qquad \text{ and } u(x_{k-1}), u(x_{k+2}) \in [-M,M]  \Big).\label{jn20.1}
\end{align}
Using the Markov property~\eqref{e:Markomub}, we can write for any $k$
\begin{align}
\lefteqn{\mu_{\eps,(-\Le,\Le)}^{-1,1} \Big( u \in [-1+\delta, 1 - \delta ]\text{ on all of }[ x_k, x_{k+1} ]  }\notag\\
 & \qquad \qquad \qquad \qquad \qquad \text{ and } u(x_{k-1}), u(x_{k+2}) \in [-M,M]  \Big) \notag\\
&=\int_{-M}^M \int_{-M}^M   \, \nu_{k-1, k+2}(du_-, du_+) \notag\\
& \qquad\qquad\times  \mu^{u_-,u_+}_{\eps,(x_{k-1},x_{k+2})} \Big(  u \in [-1+\delta, 1 - \delta ]\text{ on all of } [ x_k, x_{k+1} ]  \Big),\label{jn20}
\end{align}
where $\nu_{k-1, k+2}$ denotes the marginal distribution of the pair $(u(x_{k-1}),u(x_{k+2}))$.
We now want to invoke the large deviation bound~\eqref{e:LDUB1} and the energy bound from Lemma \ref{le:lazy} for the measures
$  \mu^{u_-,u_+}_{\eps,(x_{k-1},x_{k+2})}$. To this end, we observe that a $\delta/2$ ball around functions contained in $[-1+\delta,1-\delta]$ consists of functions contained in $[-1+\delta/2,1-\delta/2]$. Redefining $C_1$ by up to a factor of $8$ to account for the parameter $\delta/2$ and interval length (here $\ell$ rather than $2\ell$), we have that,  for any $\gamma >0$ and $\delta>0$, there exists an $\eps_0>0$ such that, for all $\eps \leq \eps_0$ and all $u_-, u_+ \in [-M,M]$, there holds
\begin{align}
\mu^{u_-,u_+}_{\eps,(x_{k-1},x_{k+2})} &\Big(  u(x)  \in [-1+\delta, 1 - \delta ] \text{ on all of } [x_k,x_{k+1}]\Big)  \notag\\
&\leq \exp \Big( -   \frac{1}{\eps} \Big(  \frac{2 \delta^2\ell }{C_1} - \gamma \Big)     \Big) .\label{jn20.2}
\end{align}
(Here we have used the fact that $\Le\gg 1$, so that we can choose $\ell>\ell_*$ to satisfy Lemma~\ref{le:lazy}.)
Letting $\gamma=1$ and choosing $\ell$ so that $\delta^2\ell\geq C_1$, the combination of~\eqref{jn20.1},~\eqref{jn20}, and~\eqref{jn20.2} gives
\begin{equation}
\mu_{\eps,(-\Le,\Le)}^{-1,1} \big(\mathcal{A}_2  \big) \, \leq \, \Le \exp\Big( - \frac{ \delta^2\ell }{C_1\,\eps}  \Big) ,  \label{e:A2bou}
\end{equation}
where we have trivially bounded the integral of $\nu$ by $1$.
In particular, for $\ell$ large but order-one (and depending on $n$, $\delta$), we have that the probability of $\mathcal{A}_2 $ is also of higher order with respect to the right-hand side of~\eqref{jn20.3}.

\medskip

\underline{Step 4}. Finally, we arrive at the subtler part, in which we will need the reflection operators. To begin with, let $\bar{k}=(k_1,\ldots,k_{2n})$ and write
\begin{align}
\mu_{\eps,(-\Le,\Le)}^{-1,1}  \big(\mathcal{A}_3 \big)
 \leq     \sum_{  \bar{k}\in\mathcal{I} }
\mu_{\eps,(-\Le,\Le)}^{-1,1} \big( \mathcal{A}_3^{\bar{k}}  \big), \label{e:transitions1}
\end{align}
where $\mathcal{I}$ is the set of nondecreasing $2n$-tuples, i.e.,
\begin{align*}
\mathcal{I} \, := \, \Big\{  & \bar{k} =  (k_1, k_2, \ldots, k_{2n})   \in \big\{ \!  -(N_\eps-1) , \ldots, (N_\eps-1)  \big\}^{2n} \\
&\text{with}\, k_{i-1}\leq k_i    \Big\},
\end{align*}
and
\begin{align}
\mathcal{A}_3^{\bar{k}} \,: = \,& \complement{\mathcal{A}_1} \cap\Big\{ \text{in each  $I_{k_i}$  there is a $\delta^-$~layer} \Big\}.\label{e:defa3}
\end{align}
The right-hand side of \eqref{e:defa3} is slightly ambiguous if several indices coincide or in the case of overlapping intervals, i.e. if $k_{i+1} = k_i+1$ for some $i$. If $j$ subsequent indices coincide, the right-hand side of  \eqref{e:defa3} has to be interpreted as saying that there are at least  $j$ $\delta^-$~transitions in the corresponding interval. In the case of overlapping intervals, for instance if $k_{i+1}=k_i+1$, the right-hand side of \eqref{e:defa3} should be interpreted to mean that there are at least two transitions in the interval $[(k_{i}-1)\ell, (k_{i}+2)\ell]$ and, moreover, one is fully contained in $[(k_{i}-1)\ell, (k_{i}+1)\ell]$ and one is fully contained in $[k_{i}\ell, (k_{i}+2)\ell]$.

The index set satisfies
\begin{align}
\big| \mathcal{I}\big| \lesssim N_\eps^{2n}\lesssim (\Le)^{2n}.\label{9pm}
\end{align}
(Recall our convention for the use of the symbol $\ls$ introduced in Notation~\ref{N:Notation}.) Hence, to complete the proof of~\eqref{jn20.3}, it suffices to show that for fixed  $\bar{k}\in\mathcal{I}$, we have
\begin{align}
\mu_{\eps,(-\Le,\Le)}^{-1, 1}\Big( \mathcal{A}_3^{\bar{k}}\Big)\lesssim\exp\left(-\frac{2nc_0-\gamma}{\eps}\right).\label{502pm}
\end{align}

As explained above, the main step is to reduce the problem of estimating the probability of $\delta^-$ layers to estimating the probability of wasted $\delta^-$ excursions. This will be achieved through suitable \emph{reflections}.

Let us at first assume that the $I_{k}$ are well-separated in the sense that
\begin{align*}
k_i\geq k_{i-1}+4\;\;\text{for all } i.
\end{align*}
Let us also assume that we are away from the boundary, i.e., that
\begin{align*}
k_1\geq -N_\eps+2,\quad k_{2n}\leq N_\eps-2.
\end{align*}
We will consider the possibilities of (a) intervals that overlap or are nearby, (b) intervals that are the same ($k_i=k_{i+1}$), and (c) boundary intervals at the end of Step 5.

We start by defining $n$ left stopping points $\chi_1, \ldots, \chi_n$ in the following manner. For $i =1, \ldots, n$ we set
\begin{align}
\chi_{i}:= \inf \big\{ &y \geq x_{k_{i}-1}  \colon u(y) =0\;\text{and }  |u(x)|=1-\delta \notag \\
&\text{ for some } x\in (x_{k_{2i-1}-1 },y)   \big\}.\label{e:defstp}
\end{align}
Here we set $\chi_{i} = \Le$ if the corresponding set is empty. It is easy to see that these random points are all left stopping points.  In a similar fashion, for $i =n+1, \ldots, 2n$ we set
\begin{align}
\chi_{i}:= \sup \big\{ &y \leq x_{k_{i}+1} \colon u(y) =0\;\text{and }  |u(x)|=1-\delta \notag \\
&\text{ for some } x\in (y,x_{k_{i}+1} )   \big\}.\label{e:defstp1}
\end{align}
Here we set $\chi_i  =-\Le$ if the corresponding set is empty.  Then $\chi_i$ is a right stopping point for all $i=n+1, \ldots 2n$. For any $u$ in $\A_3^{\bar{k}}$, all the left and right stopping points $\chi_i$  are contained in the corresponding intervals $I_{k_i}$ and, furthermore, we have
\begin{equation}
\chi_1 < \chi_2 < \ldots < \chi_n < \chi_{n+1} < \ldots < \chi_{2n}.\label{dubbub}
\end{equation}
Finally, note that as soon as $\chi_i \neq \pm \Le$, we have that $u(\chi_{i}) =0$.

For any left stopping point $\chi_l  \in \{ \chi_1, \ldots \chi_n\}$ and any right stopping point $\chi_r \in \{ \chi_{n+1}, \ldots , \chi_{2n}\}$, we now define the reflection operator  $R_{\chi_{l}}^{\chi_{r}}$. If $\chi_{l} < \chi_{r}$ (which is the case for any $u \in \Aa_{3}^{\bar{k}}$ as remarked above), we set
\begin{equation*}
\Rrl u(x) :=
\begin{cases}
- u(x)  \qquad &\text{for } x \in [\chi_l , \chi_r],\\
 u(x)  \qquad &\text{for } x \notin [\chi_{l} , \chi_r].
\end{cases}
\end{equation*}
 If $\chi_{l} \geq \chi_r$ we set $\Rrl u := u$. We clearly have $\Rrl \Rrl = \Id$; hence, $\Rrl$ is injective and onto. In order to show that $\Rrl$ preserves $\mu_{\eps,(-\Le,\Le)}^{-1, 1}$, we observe that for any measurable and bounded test function $\Phi\colon C([-\Le,\Le]) \to \R$, we have
 \begin{eqnarray*}
  &&\mathbb{E}_{(-\Le,\Le)}^{\mu_\eps,-1,1} \big( \Phi \circ \Rrl \big) \\
  &=& \mathbb{E}_{(-\Le,\Le)}^{\mu_\eps,-1,1} \big( \mathbf{1}_{ \{\chi_l < \chi_r \} }\Phi \circ \Rrl  \big) +  \mathbb{E}_{(-\Le,\Le)}^{\mu_\eps,-1,1} \big(\mathbf{1}_{ \{\chi_l \geq \chi_r \} }  \Phi \circ \Rrl \big) \\
  &=&  \mathbb{E}_{(-\Le,\Le)}^{\mu_\eps,-1,1}\Big(\, \mathbf{1}_{ \{\chi_l < \chi_r \} } \,  \mathbb{E}_{(-\Le,\Le)}^{\mu_\eps,-1,1}\big( \Phi \circ \Rrl  \big| \F_{[-\Le,\chi_l]} \vee \F_{[\chi_r, \Le]}   \big)\Big)  \\
  &&\qquad +  \mathbb{E}_{(-\Le,\Le)}^{\mu_\eps,-1,1} \big( \mathbf{1}_{ \{\chi_l \geq \chi_r \} }\,\Phi \big)\\
  &\overset{\eqref{e:strMarkovmu}}{=}&  \mathbb{E}_{(-\Le,\Le)}^{\mu_\eps,-1,1}\Big( \mathbf{1}_{ \{\chi_l < \chi_r \}} \, \mathbb{E}_{(\chi_l,\chi_r)}^{\mu_\eps,{ \bf u}}\big( \Phi \circ \Rrl    \big) \,  \Big)  \\
  &&\qquad +  \mathbb{E}_{(-\Le,\Le)}^{\mu_\eps,-1,1} \big(\mathbf{1}_{ \{\chi_l \geq \chi_r \} }\,\Phi \big).
 \end{eqnarray*}
Now we can use the fact that on the set $\{\chi_l < \chi_r \}$ we have almost surely that $u(\chi_l)= u(\chi_r)=0$ and  the invariance of the measure $\mu_{\eps,(\chi_l, \chi_r)}^{0,0}$ under the reflection $R\colon u \mapsto -u$. Note that the latter property relies on the symmetry of the double-well potential $V$. We get
\begin{eqnarray}
  && \hspace{-30pt}\mathbb{E}_{(-\Le,\Le)}^{\mu_\eps,-1,1}\Big(\, \mathbf{1}_{ \{\chi_l < \chi_r \} } \,  \mathbb{E}_{(\chi_l,\chi_r)}^{\mu_\eps,{\bf u}}\big( \Phi \circ \Rrl  \big) \, \Big)   +  \mathbb{E}_{(-\Le,\Le)}^{\mu_\eps,-1,1} \big(  \mathbf{1}_{ \{\chi_l \geq \chi_r \} }\,\Phi \big)\notag\\
 &=&  \mathbb{E}_{(-\Le,\Le)}^{\mu_\eps,-1,1}\Big(\mathbf{1}_{ \{\chi_l < \chi_r \} } \, \mathbb{E}_{(\chi_l,\chi_r)}^{\mu_\eps,{\bf u}}\big( \Phi     \big) \Big) +   \mathbb{E}_{(-\Le,\Le)}^{\mu_\eps,-1,1} \big( \mathbf{1}_{ \{\chi_l \geq \chi_r \} }\,\Phi \big)\notag\\
  &\overset{\eqref{e:strMarkovmu}}{=}& \mathbb{E}_{(-\Le,\Le)}^{\mu_\eps,-1,1}\big( \Phi \big)  \label{e:rinv}.
 \end{eqnarray}

Now we are finally ready to define the reflection operator as the composition
\begin{align}
\mathsf{R}:=R_{\chi_{1}}^{\chi_{2n}}\circ \cdots
 \circ  R_{\chi_{n-1}}^{\chi_{n+2}} \circ  R_{\chi_{n}}^{\chi_{n+1}}.\label{r}
\end{align}
We have again that $\mathsf{R}^2= \Id$. For any profile $u \in \Aa_{3}^{\bar{k}}$, the operator $\mathsf{R}$ acts in the following way:  In intervals of the form $(\chi_{i}, \chi_{i+1})$ for $i$ odd, $u$ is replaced by $-u$, and on the rest of the system, $u$ is left invariant. The action of the operator $\mathsf{R}$ on a typical  path in $\Aa_3$ is illustrated in Figure \ref{fig:4}.

\begin{figure}
\centering
\begin{tikzpicture}[xscale=0.025, yscale=1.2, >=stealth]

\draw[black, style=dotted] (1,0) -- (500,0) node[anchor=north east]{\small $x$};
\foreach \x in {20,38,56,74,...,488}
\draw[black, thin, style = densely dotted] (\x, -.05) -- (\x,0.05);

\draw[black] plot file{./Graphs/G3P2.txt};
\draw[black, style = densely dashed] plot file{./Graphs/G3P1.txt};
\draw[black, style = densely dashed] plot file{./Graphs/G3P3.txt};

\draw[->] (88,.6) --  node[fill=white]{$\mathsf{R}$} (88,-.6) ;
\draw[->] (307.5,.6) --  node[fill=white]{$\mathsf{R}$} (307.5,-.6) ;

\draw (51,0) node[anchor= south east]{$\chi_1$};
\draw (125,0) node[anchor= south west]{$\chi_2$};
\draw (265,0) node[anchor= south east]{$\chi_3$};
\draw (350,0) node[anchor= south west]{$\chi_4$};
\end{tikzpicture}

\caption{A typical path in $\Aa_3$. The reflection operator $\mathsf{R}$ turns the up and down transitions in the intervals $I_{k_i}$ into wasted excursions in the same intervals.}
\label{fig:4}
\end{figure}
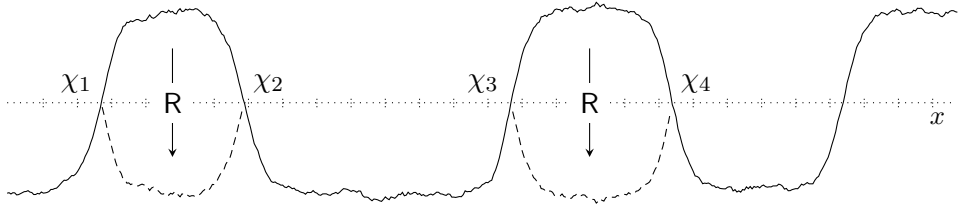


Finally, define the reflection of a set $\mathcal{A}$ as
\begin{align*}
\mathsf{R} \mathcal{A}=\{v: v= \mathsf{R}\, u\;\text{for some}\;u\in\mathcal{A} \}.
\end{align*}
As a composition of measure-preserving transformations, the operator  $\mathsf{R}$ preserves $\mu^{-1,1}_{\eps,(-\Le,\Le)}$ as well. Hence, we have in particular that
\begin{align}
\mu^{-1,1}_{\eps,(-\Le,\Le)} \big( \mathcal{A}_3^{\bar{k}}  \big)= \mu^{-1,1}_{\eps,(-\Le,\Le)} \Big(\mathsf{R}  \mathcal{A}_3^{\bar{k}}  \Big).\label{501pm}
\end{align}
This is useful because for $u \in \Aa^{\bar{k}}_3$ the profile $\mathsf{R}u$ has a wasted $\delta^-$ excursion on each interval $I_{k_i}$ (as is easy to check).
In other words, we note that
$\mathsf{R}\mathcal{A}_3^{\bar{k}}$ is a (proper) subset of the functions with wasted $\delta^-$ excursions in the given intervals.

\underline{Step 5.} It remains to bound the probability of the sets $\mathsf{R} \Aa_{3}^{\bar{k}}$. Again, let us at first assume that the $I_{k}$ are well-separated and away from the boundary in the sense described above. We consider the more general case at the end of this step.

Using the Markov property again, we have
\begin{align}
&\!\!\mu^{-1,1}_{\eps,(-\Le,\Le)} \Big(\mathsf{R} \mathcal{A}_3^{\bar{k}}  \Big)\notag\\
&\leq\mu^{-1,1}_{\eps,(-\Le,\Le)} \big(u\in \complement \mathcal{A}_1  \text{ and in each  $I_{k_i}$ there is a wasted $\delta^-$ excursion}\big) \notag\\
&= \,  \int_{-M}^M  \cdots  \int_{-M}^M  \,   \nu_{k_1-2, k_1 +2,   k_2-2,  \ldots , k_{2n}+2 } \big( du_{k_1 -2} , du_{k_1+2}, du_{k_2-2}, \ldots , du_{k_{2n} +2} \big)   \notag\\
& \qquad\prod_{\substack{i=1  }}^{2n}   \,  \mu_{\eps,  (x_{k_i -2}, x_{k_i +2 })}^{u_{k_i -2}, u_{k_i +2} } \Big(  \text{there is a wasted $\delta^-$ excursion in $I_{k_i}$}  \Big), \label{e:MarkovExc}
\end{align}
where $\nu_{k_1-2, k_1 +2,   k_2-2,\ldots, k_{2n}+2}$ denotes the distribution of the $4n$-dimensional marginal $u\big(x_{k_1-2} \big) , u\big(x_{k_1 +2} \big), u\big(x_{k_2 -2} \big) , \ldots, u\big(x_{k_{2n} +2}\big)$.

Now we would like to apply the large deviation bound~\eqref{e:LDUB1} and the energy bound from Lemma \ref{l:cl}. We observe that a $\delta$ ball around paths with a wasted $\delta^-$ excursion is equal to the set of paths with a wasted $(2\delta)^-$ excursion. As a result,
we get that for any $\gamma >0$ and $\delta>0$ there exists an $\eps_0>0$ such that for all $\eps\leq \eps_0$ and for all boundary data contained in $[-M,M]$, the probability of a wasted $\delta^-$ excursion is bounded by
\begin{align}
\mu_{\eps,  (x_{k_i -2}, x_{k_i +2 })}^{u_{k_i -2}, u_{k_i +2} } &\Big(  \text{there is a wasted $\delta^-$ excursion in $I_{k_i}$}  \Big)
\notag\\
&\leq\exp\Big( - \frac{1}{\eps}  \big( c_0 - 2C \delta- \gamma  \big) \Big).\label{5pm}
\end{align}
Choosing $\delta$ sufficiently small with respect to $\gamma$ and estimating the integral of $\nu$ by $1$ as usual, we have from the combination of~\eqref{501pm}, \eqref{e:MarkovExc}, and~\eqref{5pm} that~\eqref{502pm} holds (up to a redefinition of $\gamma$). Thus, finally,~\eqref{e:transitions1},~\eqref{9pm}, and~\eqref{502pm} imply
\begin{align*}
\mu_{\eps,(-\Le,\Le)}^{-1,1}  \big(\mathcal{A}_3 \big)\lesssim (\Le)^{2n}\exp\left(-\frac{2nc_0-\gamma}{\eps}\right),
\end{align*}
which concludes the proof of the upper bound in the well-separated case.

It remains to consider the three special cases:
(a) intervals that overlap or are nearby, (b)  intervals that are the same ($k_i=k_{i+1}$), (c) intervals that are boundary intervals.

\medskip

{\bf Case (a)} If two or more intervals  overlap (i.e., if $k_i=k_{i-1}+1$) or are nearby (i.e., if $k_{i-1}+2\leq k_i\leq k_{i-1}+3$), then we lump them together into a single, larger interval and proceed as in (b), below. The size of the largest possible interval formed in this way is $(4+3(2n-1))\ell$.  Our energy estimates require only that the interval length be sufficiently large and our large deviation estimates are uniform as long as the interval length falls within a compact set. (Here we rely on the fact that $n$ is order-one with respect to $\eps$.)

\medskip

{\bf Case (b)} If a multi-index $\bar{k}$ has repeated indices so that there is more than one  $\delta^-$~transition layer in a single interval, then we will use large deviation estimates for the event of having \emph{more than one wasted $\delta^-$ excursion} in a single interval.

Assume that we have $k_j=k_{j+1} = \ldots k_{j+m}$ for some $j <2n$ and some $1\leq m \leq 2n$. Furthermore, assume that the set of $m+1$ indices is maximal in the sense that either $j=1$ or $k_{j-1} \leq k_j-4$ and similarly that either $j+m=2n$ or $k_{j+m+1} \geq k_{j+m}+4$. In this case, we define the $m+1$ stopping points $\chi_j, \ldots \chi_{j+m}$ in the following way.

Consider any index $i\in\{j,\ldots,j+m\}$ that satisfies $i\leq n$. For $i=j$, we define $\chi_j$ as in~\eqref{e:defstp}. On the other hand, for $i>j$, we define
\begin{align*}
\chi_{i}:= \inf \big\{ y \geq x_{k_{j}-1}  \colon \text{$u(y)=0$ and there are $(i-j)$  $\delta^-$ layers in }  (x_{k_{j}-1 },y)   \big\}.
\end{align*}
As usual, we define $\chi_{i} = \Le$ if the set above is empty.

Now consider any index $i\in\{j,\ldots,j+m\}$ that satisfies $i> n$.
For $i=j+m$, we define $\chi_{j+m}$  as in \eqref{e:defstp1}. On the other hand, for $i<j+m$, we define
\begin{align*}
\chi_{i}:= \sup \big\{ &y \leq x_{k_{j}+1}  \colon u(y)=0 \notag \\
&\text{ and there are $(m-(i-j))$  $\delta^-$ layers in }  (y, x_{k_{j}+1 })   \big\}.
\end{align*}
Again, we take the usual definition $\chi_{i}= -\Le$ if the set above is empty.

As above these random points $\chi_i$ are left stopping points for $i\leq n$ and right stopping points for $i \geq n+1$. Furthermore, we still have that~\eqref{dubbub} holds for all $u \in \Aa^{\bar{k}}_3$.
The measure preserving reflection operator $\mathsf{R}$ can be defined as above in \eqref{r}, and $\mathsf{R}$ maps each $u \in \Aa_3^{\bar{k}}$ to a path that has $m+1$ wasted $\delta^-$ excursions in $I_{k_j}$. (Specifically, we mean $m+1$ wasted $\delta^-$ excursions on intervals $(x_-^i,x_+^i)\subset I_{k_j}$ for $i\in\{j,\ldots,j+m\}$ that are mutually disjoint except for possibly the endpoints.)

We leave it to the reader to verify that a generalization of Lemma~\ref{l:cl} is:

\begin{lemma}\label{l:gen}
There exists $C<\infty$ with the following property. Fix  $\delta >0$.  For any system sizes $\ell_1, \,\ell_2 <\infty$ sufficiently large and  boundary conditions $u_{\pm}\in\R$, set
\begin{align*}
\mathcal{A}^{\rm bc}&:=\{u\in C([-\ell_1-\ell_2,\ell_1+\ell_2])\colon u(-\ell_1-\ell_2)=u_{-}, u(\ell_1+\ell_2)=u_+\},\\
\mathcal{A}_0^{\rm bc}&:=\{u\in \mathcal{A}^{\rm bc}\colon u\,\text{has }m\text{ disjoint wasted $\delta^-$ excursions in } (-\ell_1,\ell_1)\}.
\end{align*}
Define the optimal cost
\begin{align*}
c_\ell:=\frac{1}{m}\left(\inf_{\mathcal{A}_0^{\rm bc}}E_{(-\ell_1-\ell_2,\ell_1+\ell_2)}(u)
-\inf_{\mathcal{A}^{\rm bc}}E_{(-\ell_1-\ell_2,\ell_1+\ell_2)}(u)\right).
\end{align*}

Then uniformly with respect to the boundary values $u_\pm$, one has
\begin{align*}
c_\ell-\,c_0\leq\,C\,\delta.
\end{align*}
\end{lemma}

\medskip

{\bf Case (c)}  Suppose for instance that there is a transition layer in $(x_{-N_\eps},x_{-N_\eps+1})$.  Then we know the boundary value $u(x_{-N_\eps})=u(-\Le)=-1$, while the boundary value $u(x_{N_\eps-2})$ at the other end of the subinterval is unknown.  This is easily handled by a suitable ``one-sided'' generalization of Lemma~\ref{l:cl}, which is easy to prove.

\medskip

Using the facts from above, the proof of the upper bound is completed by decomposing $ \mathcal{A}_3^{\bar{k}}$ into the various cases and recovering the correct (and identical) bounds in each case.

\medskip

\begin{center}
{\bf Lower bound.}
\end{center}
We turn now to the matching lower bound, i.e., that
\begin{align*}
\mu_{\eps,(-\Le,\Le)}^{-1,1} \big( & \text{ $u$ has $(2n+1)$ transition layers  } \big) \\
&\gtrsim (\Le)^{2n}\,\exp\left(-\frac{ 2n c_0+\gamma}{\eps}\right).
\end{align*}

As explained in Subsection~\ref{s:detac}, for the lower bound we will work with $\delta^+$ transition layers (cf. Definition~\ref{def:layerB}). Because of the boundary conditions and the definition of $\delta^+$ layers, it
will be sufficient to show that, for some $\delta\in(0,1/2)$, we have
\begin{align}
\lefteqn{\mu_{\eps,(-\Le,\Le)}^{-1,1} \big(  \text{ $u$ has $(2n)$ $\delta^+$ transition layers  } \big)}\notag\\
 &\gtrsim (\Le)^{2n}\,\exp\left(-\frac{ 2n c_0+\gamma}{\eps}\right).\label{jn23.3}
\end{align}
Indeed, in analogy with the upper bound, the probability of $\delta^+$ layers is bounded above by the probability of transition layers, and because of the boundary conditions  there must be an odd number of transitions.

\medskip

\underline{Step 1}.
Once again, we will use the gridpoints $x_k$ defined in~\eqref{xk}. Our first step is to get some control on the values of $u$ at the gridpoints. The following lemma, used below, is established via techniques similar to those used for the upper bound.
\begin{lemma}\label{l:negneg}
For any $M<\infty$ sufficiently large, there exists $\ell_*<\infty$ and $\eps_0>0$ such that, for $\ell\geq\ell_*$ and $\eps\leq \eps_0$, we have for any $\Le$ satisfying~\eqref{lbd} that
\begin{align}
\mu_{\eps,(-\Le,\Le)}^{-1,1} \Big( u\in\complement \mathcal{A}_1\colon  u(x)\leq 0\text{ for all }x\in[-\Le,-2\ell] \Big)\geq \frac{1}{3}.\label{bigenuf}
\end{align}
Recall the definition of $\mathcal{A}_1$ in~\eqref{e:defA1}.
\end{lemma}
The proof is similar to the proof of the upper bound, and is deferred to Subsection~\ref{ss:negneg}. The main idea is that while the boundary conditions force there to be a transition layer, with high probability, there is \emph{only one transition layer}.  Moreover, by symmetry, this layer is as likely to appear on $[0,\Le]$ as it is on $[-\Le,0]$ (hence neither probability can be more than $1/2$).  On the other hand, for $u$ to hit zero away from the transition layer is energetically unlikely, by arguments similar to those used for the upper bound.

\medskip

\underline{Step 2}.
With Lemma~\ref{l:negneg} in hand, we turn to the basic set-up for the lower bound.
In this case, we will not want to use overlapping subintervals. We will also not work with the full system, but only with intervals on the left-hand side. Specifically, we will work with
$$I_k=[x_{k-1},x_{k+1}]\quad\text{for}\;\;k\in\{-(N_\eps-4),-(N_\eps-8),\ldots,
-4\}=:E.$$
We have assumed without loss of generality that $4$ divides $N_\eps$.  (If not, then $N_\eps=4j+r$ for some $j\in\N$ and $r\in\{1,2,3\}$. Replace $N_\eps$  by $N_\eps-r$ throughout.)

We remark that, as usual, for an event falling in the interval $I_k$, we will condition on the boundary values on a larger interval. Specifically, we will use a Markov decomposition in which we condition on the boundary values of the enlarged interval
$$\tilde{I}_{k}:=[x_{k-2},x_{k+2}].$$
Notice that for all $k\in E$, the enlarged intervals $\tilde{I}_k$ are  nonintersecting. For future reference, let us denote the set of boundary indices
\begin{align*}
 E_b:=\{-(N_\eps-2),-(N_\eps-6),\ldots,-2\}.
\end{align*}

The rough idea is to consider sets of functions having $2n$ layers with a layer in one of the intervals $I_k$ for $2n$ distinct values of $k\in E$. Unfortunately, because we work with functions $u$ that have \emph{at least} $2n+1$ transitions rather than \emph{exactly} $2n+1$ transitions, a given function $u$ may have more than $2n+1$ layers and belong to more than one of the sets we have just described. Hence we cannot translate the probability of the union into the sum of the probabilities. In order to work around this, we will work with more restrictive sets.

Analogous to the set $\Aa_1$ defined in~\eqref{e:defA1} above, we define the following set.  Rather than keeping track of all the boundary values, it will be convenient to track only the boundary values for the extended intervals described above. That is, we consider
\begin{align*}
\tilde{\mathcal{A}}_1 := \Big\{ u  \colon \,  |u(x_k)| \geq M \text{ for some }k\in E_b\     \Big\}.
\end{align*}

We now introduce a set that is analogous to the set $\Aa_3^{\bar k}$ above (but more restrictive, for the reason we have explained). For ease of notation, we do not introduce a new label. Let $\bar{k}=(k_1,\ldots,k_{2n})$ and consider the set
\begin{align*}
\mathcal{A}_3^{\bar{k}} : =  &\complement{\tilde{\mathcal{A}}_1}\cap \Big\{ \text{in each  $I_{k_i}$ with $i$ odd there exists a $\delta^+$ up layer} \\
&   \text{and in each $I_{k_i}$ with $i$ even there exists a $\delta^+$ down layer and}\\
& \text{for $k\in E\setminus\{k_i,\;1\leq i\leq 2n\}$ $u$ does not have a $\delta^+$ layer in $I_k$ }\Big\}.
\end{align*}
Clearly, we have the following inclusion of sets of paths:
\begin{align}
\big\{ \text{ $u$ has $(2n)$  $\delta^+$ transition layers } \big\} \supseteq \bigcup_{\bar{k} \in \mathcal{I} } \mathcal{A}_3^{\bar{k}}  ,\label{incly2}
\end{align}
where  $\mathcal{I}$ is the following set of \emph{well-separated indices on the negative $x$-axis}:
\begin{align*}
\mathcal{I} \, := \, \Big\{  \bar{k} =  (k_1, k_2, \ldots, k_{2n})   \in E^{2n} \colon \text{for all } i, \;  k_{i-1} < k_{i}      \Big\}.
\end{align*}
Moreover, the sets $\mathcal{A}_3^{\bar{k}}$ for $\bar{k}\in\mathcal{I}$ are disjoint, so  that~\eqref{incly2} implies
\begin{align}
\mu^{-1,1}_{\eps,(-\Le,\Le)}   \big( \text{ $u$ has $(2n)$  $\delta^+$~transition layers } \big) \geq \sum_{\bar{k} \in \mathcal{I} } \mu^{-1,1}_{\eps,(-\Le,\Le)} \big( \mathcal{A}_{3}^{\bar{k}} \big).\label{parak}
\end{align}

The set on the right-hand side of~\eqref{incly2} is certainly smaller than the set on the left-hand side, but the bound will be good enough on the level of scaling since
\begin{align}
|\mathcal{I}|\gtrsim N_\eps^{2n}\gtrsim (\Le)^{2n}.\label{ibig}
\end{align}

\medskip

\underline{Step 3}.
Given~\eqref{parak} and~\eqref{ibig}, we will be done if we can establish that for any $\gamma>0$ and for $\eps>0$ sufficiently small, we have
\begin{align}
\mu^{-1,1}_{\eps,(-\Le,\Le)} \big( \mathcal{A}_3^{\bar{k}} \big)  \gtrsim\exp\Big( - \frac{2nc_0+\gamma}{\eps} \Big).\label{parak2}
\end{align}
%
%

To this end, fix any multi-index $\bar{k} \in \mathcal{I}$. We will now bound the probability of $\mathcal{A}^{\bar{k}}_3$ using reflections, as we did for the upper bound.  Indeed, let
\begin{align*}
\chi_{2i-1}:=\inf&\Big\{y\in I_{k_{2i-1}}\colon u(y)=0,\;\;u(x)=-1-\delta\\
&\qquad\qquad\text{ for some }x\in(x_{k_{(2i-1)}-1},y)\Big\},\\
\chi_{2i}:=\sup&\Big\{y\in I_{k_{2i}}\colon u(y)=0,\;\;u(x)=-1-\delta\\
&\qquad\qquad\text{ for some }x\in (y,x_{k_{2i}+1})\Big\}.
\end{align*}
Then we define the reflection operator $\mathsf{R}$ as
\begin{equation*}
\mathsf{R}= R^{\chi_{2n}}_{\chi_{2n-1}} \circ \cdots  R^{\chi_{2}}_{\chi_{1}} .
\end{equation*}
By the same argument as above in~\eqref{e:rinv} it can be seen that this operator preserves the measure $\mu^{-1,1}_{\eps,(-\Le,\Le)}$. Notice that $\mathsf{R}$ creates $\delta^+$ wasted excursions in the intervals $I_{k_i}$ and cannot create layers in any interval $I_k$  for $k\in E\setminus\{k_i,\;1\leq i\leq 2n\}$. We recover
\begin{align}
&\mu^{-1,1}_{\eps,(-\Le,\Le)} \Big(  \mathcal{A}_3^{\bar{k}} \Big)\notag\\
&=\mu^{-1,1}_{\eps,(-\Le,\Le)} \Big(\mathsf{R} \mathcal{A}_3^{\bar{k}} \Big)\notag\\
&= \,  \int_{-M}^M  \cdots  \int_{-M}^M  \,   \nu_{-(N_\eps-2), -(N_\eps-6),\ldots , -2} \big( du_{-(N_\eps-2)} , du_{-(N_\eps-6)}, \ldots , du_{-2} \big)   \notag\\
& \qquad\prod_{\substack{i=1  }}^{2n}   \,  \mu_{\eps,  (x_{k_i -2}, x_{k_i +2 })}^{u_{k_i -2}, u_{k_i +2} } \Big(  \text{ there is a wasted $\delta^+$ excursion in $I_{k_i}$}  \Big)\notag\\
&\qquad
\prod_{\substack{k\in E\setminus \{k_i, 1\leq i\leq 2n\}}}   \,  \mu_{\eps,  (x_{k -2}, x_{k +2 })}^{u_{k -2}, u_{k +2} } \Big(  \text{ there is no  $\delta^+$ layer in $I_{k}$}  \Big)\notag\\
 &\geq \,  \int_{-M}^0 \cdots  \int_{-M}^0  \,   \nu_{-(N_\eps-2), -(N_\eps-6),\ldots , -2} \big( du_{-(N_\eps-2)} , du_{-(N_\eps-6)}, \ldots , du_{-2} \big)   \notag\\
& \qquad\prod_{\substack{i=1  }}^{2n}   \,  \mu_{\eps,  (x_{k_i -2}, x_{k_i +2 })}^{u_{k_i -2}, u_{k_i +2} } \Big(  \text{ there is a wasted $\delta^+$ excursion in $I_{k_i}$}  \Big)\notag\\
 &\qquad
\prod_{\substack{k\in E\setminus \{k_i, 1\leq i\leq 2n\}}}   \,  \mu_{\eps,  (x_{k -2}, x_{k +2 })}^{u_{k -2}, u_{k +2} } \Big(  \text{ there is no  $\delta^+$ layer in $I_{k}$}  \Big).\label{newtag}
\end{align}
As usual, $\nu$ denotes the distribution of boundary values, here at the boundary points of each extended interval $\tilde{I}_k$ for $k\in E$. Note that the second equality follows from the definition of wasted $\delta^+$ excursions. The definition of wasted $\delta^-$ excursions is different and led to an inequality in the analogous estimate, cf.~\eqref{e:MarkovExc}.

We remark that we do not actually need to condition on the boundary values for every $k\in E$---it would be enough to consider the intervals $\tilde{I}_k$ for $k\in \bar{k}$ and the complementary intervals---but doing it this way keeps notation simple and because of~\eqref{lbd}, it does not affect our bound by more than an exponentially small amount.

We now turn to the lower large deviation bound~\eqref{e:LDLB1}  and the energy bound from Lemma \ref{l:cll} (where we use that the boundary values are in $[-M,0]$). We  recall that the set $\mathcal{A}_{\delta,pre}^{\rm bc}$ from~\eqref{preset} was defined precisely so that
$$B(\mathcal{A}_{\delta,pre}^{\rm bc},\delta)=\{u\colon \text{$u$ has a wasted $\delta^+$ excursion in $[-\ell,\ell]$}\}.$$
Therefore, applying the large deviation estimate to~\eqref{newtag},  we conclude that  for any $\gamma >0$ and $\delta>0$ small enough, there exists an $\eps_0>0$ such that for any $\bar{k} \in \mathcal{I}$ and any $\eps \leq \eps_0$, we have
\begin{align}
 \mu^{-1,1}_{\eps,(-\Le,\Le)} \big( \mathcal{A}_3^{\bar{k}}  \big)  \geq \exp\left( - \frac{2nc_0+\gamma}{\eps}  \right) \, \mu^{-1,1}_{\eps,(-\Le,\Le)} \big(\mathcal{A}_4^{\bar k}\big),\label{j28}
\end{align}
where
\begin{align*}
\mathcal{A}_4^{\bar k} := &\Big\{ u \colon \,\text{$u(x_k)\in[-M,0]$ for all $k\in E_b$ and}\\
&\qquad\text{$u$ has no  layer in $I_{k}$ for any $k\in E\setminus\{k_i, 1\leq i\leq 2n\}$}\Big\}.
\end{align*}

At the same time, for any $\bar{k}\in\mathcal{I}$  we have
\begin{align*}
 \mathcal{A}_4^{\bar k} \supseteq \{u\in \complement\mathcal{A}_1\colon u(x)\leq 0 \text{ for all }x\in[-L_\eps,-2\ell]\},
\end{align*}
where $\mathcal{A}_1$ includes all the gridpoints, as defined in~\eqref{e:defA1}. Hence by the estimate \eqref{bigenuf} from Lemma~\ref{l:negneg},
the lower bound~\eqref{j28} improves to
\begin{align}
\mu^{-1,1}_{\eps,(-\Le,\Le)} \big( \mathcal{A}_3^{\bar{k}} \big)  \geq \frac{1}{3}\exp\Big( - \frac{2nc_0+\gamma}{\eps} \Big),\label{impy}
\end{align}
which establishes~\eqref{parak2} and completes the proof of the lower bound.

\end{proof}

\bigskip

\section{Proof of Theorem~\ref{t:uniform}: The uniform distribution of the layer location}
\label{s:Uniform}


As pointed out in Subsection \ref{ss:methods}, the proof of Theorem \ref{t:uniform} relies on the construction of a measure-preserving operator $\Rx$. This operator maps paths that exhibit a transition near $y$ to paths that exhibit a transition near $z$. It is constructed by performing a point reflection  between hitting points of $\pm 1$ near $y$ and $z$.

The main difficulty of the proof is to show that these hitting points exist with very high probability on the set of paths that perform a transition near $y$. The argument for this is provided in the following two lemmas.

The first lemma states, roughly speaking, that in the ``bulk,'' fluctuations around $\pm 1$ are of order $\eps^{1/2}$.  The system needs $O(| \log \eps|)$ space to relax to this scale. For simplicity, we state the lemma for paths that stay close to $1$. By symmetry, the analogous statement holds near $-1$.

\begin{lemma}\label{l:smallu}
There exists $C<\infty$ with the following property. For every $\ell_0<\infty$ sufficiently large, there exists $\eps_0'>0$ such that the following holds. For every $\eps$ and $\eps_0$ with $\eps \leq \eps_0 \leq \eps_0'$, there exists $K_\eps\in\N$ with
$$K_\eps\sim\log\left(\sqrt{\frac{\eps_0}{\eps}}\right)$$
such that for
\begin{equation*}
\ell_\eps:=(2K_\eps+1)\ell_0
\end{equation*}
and all $u_{\pm}\in[1/2,3/2]$, we have
\begin{align*}
&\mu_{\eps,(-\ell_\eps,\ell_\eps)}^{u_{-},u_{+}} \bigg(\sup_{x \in [- \ell_0, \ell_0]} |u(x)-1|\geq \sqrt{\frac{\eps}{\eps_0}} \;\bigg|\\
&\qquad\qquad\qquad |u(\pm(2k-1)\ell_0)-1|\leq \frac{1}{2},\,k=1,2,\ldots, K_\eps \bigg)\\
& \leq 4   \,\exp\left(-\frac{1}{C\eps_0}\right).
\end{align*}
\end{lemma}

We present the proof in Subsection~\ref{ss:unilem}.  Next we need a lemma that says that with positive probability, the path \emph{actually hits} $\pm 1$. Again, we state the result for hitting points of $+1$. By symmetry, the analogous statement holds for hitting points of $-1$.
\begin{lemma}
\label{l:hittingzero}
For any  $ \ell_0< \infty$  sufficiently large, there exist $\eps_0>0$ and $\lambda \in (0,1)$ such that the following holds. For  any $u_{\pm} \in [1/2, 3/2] $, any $\eps \leq \eps_0$,  and  $K_\eps$, $\ell_\eps$ as in Lemma \ref{l:smallu}, we get
\begin{align*}
\mu_{\eps,(-\ell_\eps,\ell_\eps)}^{u_{-},u_{+}}  &\Big( \exists x \in [-\ell_0, \ell_0] \text{ such that } u(x) = 1  \\
 & \,\Big|  |u(\pm(2k-1)\ell_0)-1|\leq \frac{1}{2},\,k=1,2,\ldots, K_\eps \Big)  \geq 1-\lambda.
\end{align*}
\end{lemma}
The proof of this lemma, also presented in Subsection \ref{ss:unilem}, follows as a corollary to the previous result by a rescaling argument.\\

\begin{proof}[Proof of Theorem~\ref{t:uniform}]
We will show  that for some $\delta \in (0,1/2)$ and any $\alpha>0$, we have for $\eps$ sufficiently small that
\begin{align}
1-\alpha\leq \frac{\Le}{\de} \; \mu^{-1,1}_{\eps,(-\Le,\Le)}\big(&\text{at least one $\delta^-$ up layer of length $\leq 2\ell$}\notag\\
 &\text{ in }[y-d_\eps,y+d_\eps] \big)\leq 1+\alpha.\label{unif2}
\end{align}
At the end of the proof, it will not be hard to improve from a $\delta^-$ up layer of length less than or equal to $2\ell$ to a full up transition layer.

\begin{notation}
 For brevity, we will often say ``a transition layer $\leq 2\ell$'' as shorthand for ``a transition layer of length less than or equal to $2\ell$.''
\end{notation}

For $\eps$ small enough we consider intervals of type  $\Je=[y -d_\eps,y+d_\eps ]\subseteq [\Le,\Le]$. The main step of our argument consists of proving that the probabilities of transitions in these intervals $\Je$ for different values of $y$ are roughly the same. Hence fix two points $y$, $z$ such that $\Je,J_{z,\eps}\subseteq [-\Le,\Le]$. Without loss of generality, assume that $y\leq z$.

As above in the proof of Theorem~\ref{t:layers}, let $\ell$ and $M$ be a large constants to be fixed later, and let $N_\eps$ and $x_{\pm k}$ be as defined in~\eqref{ne} and~\eqref{xk}. Moreover, consider the overlapping intervals $I_k=[x_{k-1},x_{k+1}]$ as in~\eqref{ik}. Finally, define as in~\eqref{e:defA1} the ``bad set'' $\mathcal{A}_1$ of functions that have boundary values larger than M in magnitude.
In~\eqref{sq} above, we have already established that there is a universal constant $C_2<\infty$ such that
\begin{equation*}
 \mu^{-1,1}_{\eps,(-\Le,\Le)}\big(\Aa_1 \big) \leq \Le\,\exp \bigg(-  \frac{M}{C_2\,\eps} \bigg).
\end{equation*}
Hence, for the system sizes $\Le$ that we consider, the probability of $\Aa_1$  can be made arbitrarily small by choosing $M$ large.

We now define the set of functions
\begin{align}
\mathcal{J}_{y,\eps} := &\big\{ u \in \complement \Aa_1  \colon \text{ $u$ has a $\delta^-$ up layer $\leq 2\ell$ in $J_{y,\eps}$}   \big\}.\label{e:ast}
\end{align}
The set $\mathcal{J}_{z,\eps}$ is defined analogously.
In Steps 1--3 below, we will establish that the probabilities of the $\mathcal{J}_{y,\eps}$ and $\mathcal{J}_{z,\eps}$ are roughly the same. The bounds that we obtain will be uniform with respect to $y$ and $z$. Finally, in Step 4 we will show how this implies~\eqref{unif2}, and in Step 5 we will improve to the statement of Theorem~\ref{t:uniform}.

\medskip

\underline{Step 1.} The first step consists of proving that on the set $\mathcal{J}_{y,\eps}$,  with high probability, the profile $u$ is  close to $-1$ on a sufficiently large interval $J_{y,-}^\eps$ just to the left of $J_{y,\eps}$ and close to $+1$ on a sufficiently large interval $J_{z,+}^\eps$ just to the right of $J_{z,\eps}$.

The length $h_\eps$  of each of these auxiliary intervals $\mathcal{J}^\eps_{y,-}$ and $\mathcal{J}^\eps_{z,+}$ will be chosen below such that
\begin{equation*}
|\log\eps| \ll h_\eps \ll d_\eps.
\end{equation*}

We first fix the ``inner" boundary points of $J_{y,-}^\eps$ and $J_{z,+}^\eps$: In units of $\ell$, we set
\begin{align*}
k_{y,+}^\eps:=& \max \big\{ k \colon x_k \leq y - d_\eps  \big\} -2,\\
k_{z,-}^\eps:=& \min \big\{ k \colon x_k \geq z + d_\eps  \big\} +2 .
\end{align*}
Let $K_\eps$ be as in the statement of  Lemma \ref{l:smallu}. The idea is to make the probability of hitting $\pm 1$ on the auxiliary intervals large by concatenating many subintervals of length $K_
\eps \ell$ and applying Lemma~\ref{l:hittingzero} on each subinterval. With this end in mind, we fix integers  $\bar{K}_\eps$ such that

\begin{equation}\label{e:Defhe}
\bar{K}_\eps \gg 1 \qquad \text{ and } \qquad h_\eps := \ell (2 K_\eps+1)  \bar{K}_\eps \ll d_\eps.
\end{equation}
Then we set
\begin{align*}
k_{y,-}^\eps:= k_{y,+}^\eps   - (2 K_\eps +1) \bar{K_\eps} \quad \text{and} \quad
k_{z,+}^\eps:=   k_{z,-}^\eps  + (2  K_\eps+1) \bar{K_\eps},
\end{align*}
and finally
\begin{equation*}
\Hm:= [k_{y,-}^\eps\ell,k_{y,+}^\eps\ell]  \qquad \text{ and } \qquad  \Hp:= [k_{z,-}^\eps\ell,k_{z,+}^\eps\ell] .
\end{equation*}
(See Figure \ref{fig:3} for an illustration of $J_{y,\eps}$ and $J_{y,-}^\eps$.)

\begin{figure}
\centering
\begin{tikzpicture}[xscale=0.0225, yscale=1.8]

\draw[black,thin, style = densely dotted] (1,0)  -- (500,0) node[anchor=north east]{\small $x$};
\foreach \x in {20,38,56,74,...,488}
\draw[black, thin, style = densely dotted] (\x, -.05) -- (\x,0.05);

\draw[black, ultra thick] (205,0) -- (445,0);
\draw[black, ultra thick] (205,-.05) -- (205,0.05);
\draw[black, ultra thick] (445,-.05) -- (445,0.05);

\draw[color=white, fill=white]  (300,-.05) rectangle (350,.05);
\foreach \x in {314,325,336}
\filldraw[black] (\x,0) ellipse (32pt and .4pt);

\draw [decorate,decoration={brace,amplitude=5.5pt,mirror,raise=6pt},yshift=-7pt]
(205,-.05) -- (445,-.05) node [fill=white,midway,xshift=0pt, yshift =-23pt] {
$2 d_\eps $};

\draw[black, ultra thick] (72,0) -- (164,0);
\draw[black, ultra thick] (72,-.05) -- (72,0.05);
\draw[black, ultra thick] (164,-.05) -- (164,0.05);

\draw[color=white, fill=white]  (98,-.05) rectangle (138,.05);
\foreach \x in {110,118,126}
\filldraw[black] (\x,0) ellipse (24pt and .3pt);

\draw [decorate,decoration={brace,amplitude=5.5pt,mirror,raise=6pt},yshift=-7pt]
(72,-.05) -- (164,-.05) node [fill=white,midway,xshift=0pt, yshift =-23pt] {
$ h_\eps $};

\draw (325,0.1) node[above] {$J_{y,\varepsilon}$};
\draw (118,0.1) node[above] {$J_{y,-}^{\varepsilon}$};


\draw(170,-0.05) node[fill=white, below]{$x_{ k_{y,+}^\varepsilon} $};
\draw(82,-0.05) node[fill=white, below]{$x_{k_{y,-}^\varepsilon}$};

\end{tikzpicture}
\caption{The interval $J_{y,\eps}$ and the auxiliary interval $J_{y,-}^\eps$ to its left.}
\label{fig:3}
\end{figure}
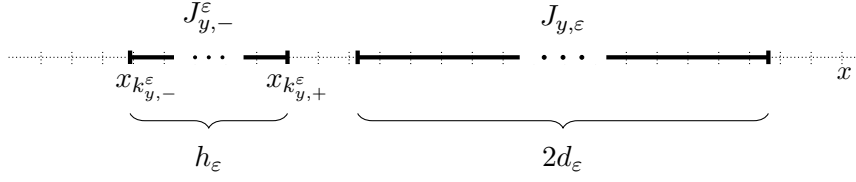

We also define the following sets of indices
\begin{equation*}
\mathcal{I}^{\eps}_{-}:= \big\{ k \colon \kxepsm\leq k\leq \kxepsp\big\},\qquad\mathcal{I}^{\eps}_+:=\big\{k\colon \kzepsm\leq k\leq \kzepsp\big\}.
\end{equation*}
For later use in~\eqref{e:n678} in the proof of Lemma \ref{l:rightshape}, we will make the additional growth assumption
\begin{equation}\label{e:Defhe2}
|\mathcal{I}_{\pm}^\eps| =   (2 K_\eps+1)  \bar{K}_\eps+1 \leq \exp(c_1/4\eps),
\end{equation}
where $c_1>0$ is defined in~\eqref{c1}, below. This is not a strong condition; we will typically think of $h_\eps$ as being much smaller.

Finally, we define another set of ``unlikely" paths, paths that have extra $\delta^-$ layers to the left of $J_{y,\eps}$ or to the right of $J_{z,\eps}$:
\begin{align}
\Aa_{y,3}^-&:=\big\{ u \in \Be\colon  \text{there exists $x \leq  (k_{y,+}^\eps+1) \,\ell$ with $u(x) \geq 1-\delta$} \big\},\notag\\
\Aa_{y,3}^+&:=\big\{ u \in \Be\colon  \text{there exists $x \geq (k_{z,-}^\eps-1) \,\ell  $ with $u(x) \leq -1+\delta$} \big\},\notag\\
 \Aa_{y,3}&:= \Aa_{y,3}^-\cup \Aa_{y,3}^+. \label{Apm}
\end{align}

We now introduce two lemmas. The proofs of both lemmas are given in Subsection \ref{ss:unilem}.
The first lemma is an extension of the upper bound in Theorem \ref{t:layers} and states roughly speaking that \emph{conditioned} on having a transition in a given interval, the probability of extra layers somewhere else is small.
 \begin{lemma}\label{l:reflection}
Let $Y$ be a subinterval of $[-\Le,\Le]$ and let $x_{-}= k_-\ell$ and $x_+= k_+\ell$ be two gridpoints (cf.~\eqref{xk}) to the left and to the right of $Y$ respectively with distance $\geq \ell$ from $Y$.
We denote by $\JY$ and $\Ayth$ the sets
\begin{align*}
\JY:=& \big\{ u \in \complement \Aa_1 \colon \, \text{$u$ has a }\delta^- \text{ up layer in $Y$}  \big\}, \\
\Aa_{Y,3}:=& \big\{ u \in \JY \text { and $u$ has another $\delta^-$ layer outside of $[x_-,x_+]$} \big\} .
\end{align*}
Fix any $\gamma>0$ and any $M <\infty$ sufficiently large. For any $\delta>0$ sufficiently small and $\ell< \infty$ sufficiently large, there exists $\eps_0 >0$ such that, for all $\eps \leq \eps_0$,  we have
\begin{align}
&\mu^{-1,1}_{\eps,(-\Le,\Le)} \big(\Aa_{Y,3}   \big)  \ls  \Le \exp \Big( - \frac{c_0-\gamma}{\eps}  \Big)  \mu^{-1,1}_{\eps,(-\Le,\Le)}\big(  \JY \big). \label{e:rightshapea}
\end{align}
\end{lemma}
We now apply Lemma \ref{l:reflection} for $Y = J_{\eps,y}$ and for $x_-= (k_{y,+}^\eps +1) \ell $ and $x_+ = (k_{z,-}^\eps -1) \ell $. Because of the boundary conditions, the absence of layers to the left of $x_-$ and the right of $x_+$ implies in particular that  $u\leq 1-\delta$ to the left of $x_-$ and that $u\geq -1+\delta$ to the right of $x_+$. Hence we deduce that
\begin{align}
 \mu^{-1,1}_{\eps,(-\Le,\Le)} &\big( \Aa_{y,3}  \big)  \ls  \Le \exp \Big( - \frac{c_0-\gamma}{\eps}  \Big)  \mu^{-1,1}_{\eps,(-\Le,\Le)}\big( \mathcal{J}_{y,\eps} \big). \label{e:rightshape2}
\end{align}

The second lemma establishes that, on the other hand, when there are no extra layers, there is only a small probability of making an excursion from $-1$ at some gridpoint in $J_{y,-}^\eps$ (respectively, an excursion from $1$ at some gridpoint in $J_{z,+}^\eps$). The result from the second lemma is exactly the necessary ingredient that we need in Step 2 in order to invoke Lemma~\ref{l:hittingzero}.
\begin{lemma}\label{l:rightshape}
Fix any  $M <\infty$ sufficiently large. For any $\delta>0$ sufficiently small and $\ell< \infty$ sufficiently large, there exists $\eps_0 >0$ such that for all $\eps \leq \eps_0$  we have
\begin{align}
 \mu^{-1,1}_{\eps,(-\Le,\Le)} &\big( u \in \mathcal{J}_{y,\eps} \cap \complement \Aa_{y,3}^-\colon |u(x_k) + 1| \geq \frac{1}{2} \text{ for some  $k \in \mathcal{I}^{\eps}_-$ }  \big) \notag\\
 & \leq  \exp\bigg(  -\frac{c_1}{2\eps} \bigg)  \mu^{-1,1}_{\eps,(-\Le,\Le)}\big( \mathcal{J}_{y,\eps} \big), \label{e:rightshape}
\end{align}
and, similarly,
\begin{align}
 \mu^{-1,1}_{\eps,(-\Le,\Le)} &\big( u \in \mathcal{J}_{y,\eps} \cap \complement \Aa_{y,\eps}^+\colon |u(x_k) - 1| \geq \frac{1}{2} \text{ for some  $k \in \mathcal{I}^{\eps}_+$ }  \big) \notag\\
 & \leq \exp\bigg(  -\frac{c_1}{2 \eps} \bigg)  \mu^{-1,1}_{\eps,(-\Le,\Le)}\big( \mathcal{J}_{y,\eps} \big), \label{e:leftshape}
\end{align}
where $c_1$ is defined in~\eqref{c1}, below.
\end{lemma}

\underline{Step 2.} The second step of the proof consists of showing that paths in the set $\Be$ have hitting points of $-1$ in $\Hm$ and hitting points of $+1$ in $\Hp$ with large probability. This is captured by the following lemma, which is also proved in Subsection~\ref{ss:unilem}.

\begin{lemma}\label{l:hittingz} There exists $C<\infty$ with the following property.
Fix any $\gamma>0$ and any $M <\infty$ sufficiently large. For any $\delta>0$ sufficiently small and $\ell< \infty$ sufficiently large, there exists $\eps_0 >0$  and $\lambda>0$ such that, for all $\eps \leq \eps_0$,  we have
\begin{align}
 \mu^{-1,1}_{\eps,(-\Le,\Le)}& \Big( u \in \Be \colon \text{no hitting of $-1$ in $\Hm$} \Big)\notag \\
 &  \leq  \frac{1}{2}\,E(\eps)  \, \mu^{-1,1}_{\eps,(-\Le,\Le)}\big( \Be \big) ,\label{e:hz11}\\
  \mu^{-1,1}_{\eps,(-\Le,\Le)}  &\Big( u \in \Be \colon \text{no hitting  of $+1$ in $\Hp$} \Big) \notag\\
  & \leq  \frac{1}{2}\,E(\eps)  \,  \mu^{-1,1}_{\eps,(-\Le,\Le)}\big( \Be  \big)  \label{e:hz29} ,
\end{align}
where the error term satisfies
\begin{align}\label{e1}
E(\eps):= C\left(\lambda^{\bar{K}_\eps}+\Le\exp\bigg(-\frac{c_0-\gamma}{\eps}\bigg)+\exp\bigg(-\frac{c_1}{2\eps}\bigg)\right),
\end{align}
and $c_1$ is defined in~\eqref{c1}, below.
\end{lemma}

\underline{Step 3.} Now we are ready to define the reflection operator $\Rx$. First, we define the following left and right stopping points
\begin{align*}
\chi_-  \,:= \, & \inf\big\{x  \in \Hm \colon u(x) = -1    \big\} , \\
\chi_+\, : = \, & \sup \big\{x \in \Hp \colon u(x) = 1 \big\}.
\end{align*}
Here we use the convention that $\chi_-= \Le$ if there is no hitting point of $-1$ in $\Hm$ and similarly $\chi_+=-\Le$ if there is no hitting point of $1$ in $\Hp$. We use these hitting points to define the reflection operator
\begin{align}
\Rx u(x):=
\begin{cases}
u(x)  & \text{for } x \leq \chi_-,\\
-u(\chi_- + \chi_+ - x) & \text{for } \chi_- < x < \chi_+,\\
u(x)  & \text{for } x \geq \chi_+,\\
\end{cases}
\end{align}
if $\chi_- \leq \chi_+$. We set  $\Rx$ to be the identity otherwise. In other words the operator $\Rx$ performs a point reflection of the graphs of $u$ between the left and right stopping points $\chi_{\pm}$. As in Step 4 of the proof of the upper bound in Theorem~\ref{t:layers}, one argues that the strong Markov property~\eqref{e:strMarkovmu} implies that $\Rx$ leaves the measure  $  \mu^{-1,1}_{\eps,(-\Le,\Le)}$ invariant. The action of the reflection operator is illustrated in Figure \ref{fig:5}.

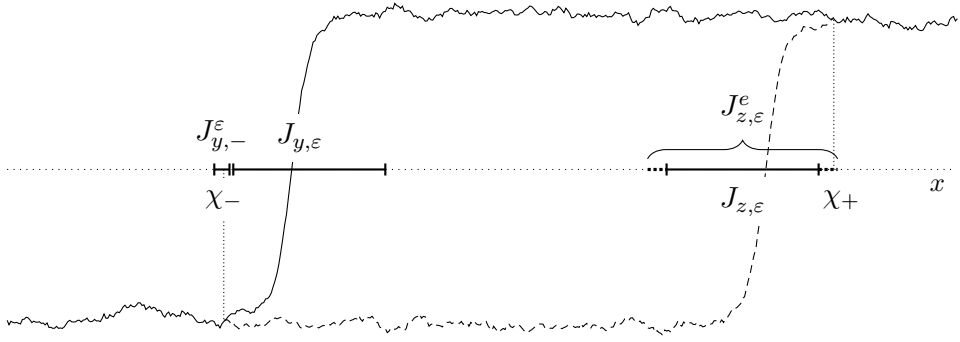
\begin{figure}
\centering

\begin{tikzpicture}[xscale=0.025, yscale=2, >=stealth]

\draw[black, style=dotted] (1,0) -- (500,0) node[anchor=north east]{\small $x$};

\draw[black] plot file{./Graphs/P5G1.txt};
\draw[black, style = densely dashed] plot file {./Graphs/P5G2.txt};

\draw[black, thick] (120,0) -- (200,0);
\draw[black,  thick] (120,-.04) -- (120,0.04);
\draw[black,  thick] (200,-.04) -- (200,0.04);
\draw(155,0.05) node[above, fill=white] {$J_{y,\varepsilon}$};
\draw[black, thick] (110,0) -- (118,0);
\draw[black,  thick] (110,-.04) -- (110,0.04);
\draw[black,  thick] (118,-.04) -- (118,0.04);
\draw (114,0.05) node[above] {$J_{y,-}^{\varepsilon}$};

\draw[black, thick] (348,0) -- (428,0);
\draw[black,  thick] (348,-.04) -- (348,0.04);
\draw[black,  thick] (428,-.04) -- (428,0.04);
\draw (388,-0.05) node[below,fill=white] {$J_{z,\varepsilon}$};
\draw[densely dotted, very thick] (428,0) -- (438,0);
\draw[densely dotted, very thick] (338,0) -- (348,0);


\draw [decorate,decoration={brace,amplitude=7pt},xshift=-0pt,yshift= 2pt]
(338,0.0) -- (438,0) node [midway,yshift= 18 pt]
{ $J_{z,\varepsilon}^e$};


%
\draw[style= densely dotted] (115,0) -- (115,-1);
\draw[style=densely dotted] (436,0) --  (436,1);
\draw (440,-0.2) node[ fill=white] {$\chi_+$} ;
\draw (115,-0.2) node[ fill=white] {$\chi_-$} ;

\end{tikzpicture}
\caption{The reflection operator $\mathsf{R}_{y,z}$ performs a point reflection of the path between the left and right hitting points $\chi_{\pm}$. In this way the $\delta^-$ transition in $J_{y,\eps}$ is mapped into $J_{z,\eps}^e$.}
\label{fig:5}
\end{figure}

Assume that  $u\in \Be$ is a path that admits a hitting point of $-1$ in $\Hm$ and a hitting point of $+1$ in $\Hp$. Recall that if $u \in \Be$, then  $u$ has a $\delta^-$ up transition layer of length $\leq 2\ell$ in $J_{y,\eps}$. Under $\Rx$ the $\delta^-$ up transition layer is mapped from $J_{y,\eps}$ to near $J_{z,\eps}$ and we would like to conclude that the reflected path is contained within $\mathcal{J}_{z,\eps}$.

Unfortunately, the layer does not necessarily fall within $J_{z,\eps}$. What is true is that
there is a $\delta^-$ up layer of length less than  $2\ell$ in the extended interval
\begin{equation}\label{e:jee}
J_{z,\eps}^e:= [z - d_\eps - 3\ell - h_\eps, z + d_\eps + 3\ell + h_\eps] .
\end{equation}
(Recall that $h_\eps$, the length of the auxiliary intervals, was defined above in~\eqref{e:Defhe}.)


Let us denote by $\mathcal{J}_{z,\eps}^e$ the set of functions with a $\delta^-$ up transition layer of length less than $2\ell$ in $J_{z,\eps}^e$:
\begin{equation*}
\mathcal{J}_{y,\eps}^e := \big\{  u \colon \text{ $u$ has a $\delta^-$ up layer $\leq 2\ell$ in $J_{y,\eps}^e$}   \big\}.
\end{equation*}

In Step 2, we had established that
\begin{align}
\mu^{-1,1}_{\eps,(-\Le,\Le)}& \Big( u \in \Be \colon \text{no hitting of $-1$ in $\Hm$ or no hitting of $+1$ in $\Hp$} \Big) \notag \\
 &  \leq  E(\eps)\, \mu^{-1,1}_{\eps,(-\Le,\Le)}\big( \Be \big). \notag
\end{align}
Hence, as $\Rx$ leaves $\mu^{-1,1}_{\eps,(-\Le,\Le)}$ invariant, we can conclude that
\begin{equation} \label{e:ResStt}
\mu^{-1,1}_{\eps,(-\Le,\Le)} \big( \mathcal{J}_{z,\eps}^e \big) \geq   \big(1 -   E(\eps) \big)  \, \mu^{-1,1}_{\eps,(-\Le,\Le)}\big( \Be \big).
\end{equation}
An analogous construction to turn transitions in $J_{z,\eps}$ into transitions near $\Je$ can be performed to obtain the same bound with $J_{y,\eps}$ and $J_{z,\eps}$ interchanged.

\medskip
\underline{Step 4.} In this step, we establish the bound \eqref{unif2}. For notational convenience we will establish the bound in the case of the center interval $[-d_\eps,d_\eps]$, but our argument does not depend on this. More precisely, what we show is that for some $\delta>0$ and any $\alpha>0$, there exists an $\eps_0>0$ such that, for $\eps\leq\eps_0$, we have
\begin{align}
\bigg| \frac{\Le}{\de} \; \mu^{-1,1}_{\eps,(-\Le,\Le)}\big(&u\in\complement{A}_1\colon \text{at least one $\delta^-$ up layer }\notag\\
& \text{ of length $\leq 2\ell$}\text{ in }[-d_\eps,d_\eps] \big)\;-1\;\bigg|\leq\alpha.\label{step5}
\end{align}

The main ingredient will be the estimate  \eqref{e:ResStt}. We will also use make use of Lemma \ref{l:reflection} but except for that, the argument is completely elementary and only consists of choosing the right intervals and sets of paths.

We first split  up the system into smaller blocks. Actually, it will useful to define two different partitions $\{ J_{k,\eps}, k = -M_\eps, \ldots, M_\eps-1\}$ and $\{ J_{m,\eps}^e, m = -\tilde{M}_\eps, \ldots, \tilde{M}_\eps-1\}$ of $[-\Le,\Le]$. The lengths  of the intervals $J_{k,\eps}$ will be chosen small relative to $d_\eps$ but still large relative to $|\log \eps|$. These intervals  will be overlapping and play the role of $J_{y,\eps}$ when we  apply \eqref{e:ResStt}. The  intervals $J_{m,\eps}^e$ will be slightly larger than than the intervals $J_{k,\eps}$ and will be of distance $2 \ell$ away from each other. They will be used as $J_{z,\eps}^e$ when applying \eqref{e:ResStt}.

We fix integers  $M_\eps$ and $k_\eps$ such that
\begin{align}
&|\log \eps| \ll M_\eps^{-1} \Le \ll \de, \notag\\
\text{and }\; & k_\eps M_\eps^{-1} \Le \leq \de <(k_\eps+1) M_\eps^{-1 }\Le.\label{sqar}
\end{align}
Then we set $\tilde{d}_\eps:=\Le/M_\eps$ and define the overlapping intervals
\begin{equation*}
J_{k,\eps}:= [k  \tilde{d}_\eps - 2 \ell , (k+1) \tilde{d}_\eps +2\ell  ],\quad k = - ( M_\eps- 1) , \ldots, M_\eps-2.
\end{equation*}
The boundary intervals are defined as
\begin{align*}
J_{-M_\eps ,\eps}:=& [-\Le ,-(M_\eps-1) \tilde{d}_\eps +2\ell  ] \quad \text{and}\\
 J_{M_\eps-1 ,\eps}:=& [(M_\eps-1) \tilde{d}_\eps -2\ell, \Le  ].
\end{align*}
As above in \eqref{e:ast}, we then define the associated sets of paths as
\begin{align}
\mathcal{J}_{k,\eps} := &\big\{ u \in \complement \Aa_1  \colon \text{ $u$ has $\delta^-$ up  layer of length $\leq 2\ell$ in $J_{k,\eps}$}   \big\}.
\end{align}

In order to define the slightly longer intervals,  in analogy to the parameters $h_\eps$ and $\bar K$ from Steps 1--3, we choose parameters $\tilde{h}_\eps$ and $\tilde{\bar K}_\eps$ such that
\begin{equation}\label{e:hba}
\tilde{\bar K}_\eps \gg 1 \quad \text{ and } \quad \tilde h_\eps := \ell (2 K_\eps+1)  \tilde{\bar K}_\eps \ll  \min\left\{\tilde{d},\exp\bigg(\frac{c_1}{4\eps}\bigg)\right\}.
\end{equation}
These  parameters then define the error term $E(\eps)$ (see~\eqref{e1}, above).
Then we  define the integers $\tilde{M}_\eps$ and $m_\eps$ such that
\begin{align}
 M_\eps \bigg( \frac{1}{1 + (\tilde{h}_\eps M_\eps)L^{-1}}   \bigg) -1 \;\leq\; & \tilde{M}_\eps <  M_\eps \bigg( \frac{1}{1 + (\tilde{h}_\eps M_\eps)L^{-1}}   \bigg) \notag \\
\text{and }\quad  m_\eps \tilde{M}_\eps^{-1} \Le \;\leq\; &\de <(m_\eps+1) \tilde{M}_\eps^{-1 }\Le. \label{e:mee}
\end{align}
As above,  we define the intervals
\begin{equation*}
J_{m,\eps}^e:= [m   \tilde{M}_\eps^{-1} \Le + 2 \ell , (m+1) \tilde{M}_\eps^{-1} \Le - 2\ell  ],\quad m = - \tilde{M}_\eps  , \ldots, \tilde{M}_\eps-1.
\end{equation*}
Each of these intervals $J_{m,\eps}^e$ is of length
\begin{equation}\label{e:mtil}
\frac{\Le}{\tilde{M}_\eps} - 4 \ell = \frac{\Le}{M_\eps} + \tilde{h}_\eps - 4\ell,
\end{equation}
and in particular these intervals are long enough to use them as $J_{z,\eps}^e$ in \eqref{e:ResStt}. Actually, when comparing \eqref{e:mtil} to \eqref{e:jee}, one notices a discrepancy in the length of $10\ell$ but this can easily be treated by making $\tilde{h}_\eps$ a bit larger.

 We define the associated sets of paths
\begin{align*}
\mathcal{J}_{m,\eps}^e := &\big\{ u  \colon \text{ $u$ has a $\delta^-$ up layer $\leq 2\ell$ in $J_{m,\eps}^e$}   \big\},\\
\mathcal{J}_{m,\eps}^{e,\ast} := &\big\{ u  \in \mathcal{J}_{m,\eps}^e \colon \text{ $u$ has no $\delta^-$ up layer in any $J_{n,\eps}^e$ for any $n \neq m$}   \big\}.
 \end{align*}

 After these preliminary definitions, we are now ready to proceed to the proof of  \eqref{step5}.

 As mentioned above, the intervals $J_{k,\eps}$ are overlapping. In particular, every $\delta^{-}$ layer $\leq 2\ell$ in $[-d_\eps,d_\eps]$ must be contained in at least one of the $J_{k,\eps}$. This implies that
\begin{align}\label{e:nfs}
 \mu^{-1,1}_{\eps,(-\Le,\Le)}& \big( u \in \complement\Aa_1\colon \text{$\delta^-$ up layer $\leq 2\ell$ in $[-d_\eps,d_\eps]$}  \big)\notag\\
 &\leq \sum_{k = -(k_\eps +1)} ^{k_\eps}  \mu^{-1,1}_{\eps,(-\Le,\Le)} \big(  \mathcal{J}_{k,\eps} \big).
\end{align}
In the same way we see that \emph{every} possible path on $[-\Le,\Le]$ must be either
\begin{itemize}
\item  in one of the unlikely sets $\Aa_1$ or $\Aa_2$ defined above in~\eqref{e:defA1} and~\eqref{e:defA2}
\item  or in at least one of the sets $\mathcal{J}_{k,\eps}$.
\end{itemize}
 This implies that
\begin{equation}
\sum_{k=-M_\eps}^{M_\eps-1}  \mu^{-1,1}_{\eps,(-\Le,\Le)} \big( \mathcal{J}_{k,\eps}  \big)  \geq 1 -   \mu^{-1,1}_{\eps,(-\Le,\Le)}\big( \mathcal{A}_1 \cup \Aa_2 \big) ,
\end{equation}
and hence we have
\begin{equation}\label{e:mindeins}
\max_{k} \mu^{-1,1}_{\eps,(-\Le,\Le)} \big( \mathcal{J}_{k,\eps}  \big) \geq \frac{1 -   \mu^{-1,1}_{\eps,(-\Le,\Le)}\big( \mathcal{A}_1 \cup \Aa_2 \big) }{2 M_\eps}.
\end{equation}
On the other hand, applying \eqref{e:ResStt} gives that for any $k$ and $m$
\begin{equation}\label{e:nfs2}
\mu^{-1,1}_{\eps,(-\Le,\Le)} \big(  \mathcal{J}_{k,\eps} \big) \leq \big(1 -  E(\eps) \big)^{-1} \mu^{-1,1}_{\eps,(-\Le,\Le)} \big(  \mathcal{J}_{m,\eps}^e \big).
\end{equation}
Then, applying Lemma \ref{l:reflection} with $Y=J_{m,\eps}^e$, we have for every $m$ that
\begin{equation}\label{e:JdotJ}
 \big(1 - E(\eps) \big) \mu^{-1,1}_{\eps,(-\Le,\Le)} \big(\mathcal{J}^e_{m,\eps} \big)  \leq\mu^{-1,1}_{\eps,(-\Le,\Le)} \big(\mathcal{J}^{e,*}_{m,\eps} \big).
\end{equation}
 Finally, the sets $\mathcal{J}_{m,\eps}^{e,\ast}$ are all disjoint and in particular, we have
 \begin{equation*}
\sum_{m=-\tilde{M}_\eps}^{\tilde{M}_\eps-1}  \mu^{-1,1}_{\eps,(-\Le,\Le)} \big( \mathcal{J}_{m,\eps}^{e,\ast} \big) \leq 1,
 \end{equation*}
which implies that
\begin{equation}\label{e:ee}
\min_{m}  \mu^{-1,1}_{\eps,(-\Le,\Le)} \big( \mathcal{J}_{m,\eps}^{e,\ast} \big) \leq  \frac{1}{2\tilde{M}_\eps}  .
 \end{equation}
We now collect ingredients to deduce the upper bound
\begin{eqnarray}
&&\hspace{-30pt}\mu^{-1,1}_{\eps,(-\Le,\Le)} \big(u\in\complement{\mathcal{A}_1}\colon \text{$u$ has a $\delta^-$ up layer $\leq 2\ell$ in $[-\de,\de]$ }\big)\notag\\
&\overset{\eqref{e:nfs}}\leq& \sum_{k = -(k_\eps+1)}^{k_\eps} \mu^{-1,1}_{\eps,(-\Le,\Le)}  \big(  \mathcal{J}_{k,\eps} \big)  \notag\\
& \leq& (2 k_\eps +2) \,  \max_k\,  \mu^{-1,1}_{\eps,(-\Le,\Le)}  \big(  \mathcal{J}_{k,\eps} \big)\notag\\
& \overset{\eqref{e:nfs2}}{\leq}&  (1 - E(\eps) )^{-1} (2 k_\eps +2)  \min_m \,  \mu^{-1,1}_{\eps,(-\Le,\Le)}  \big(  \mathcal{J}_{m,\eps}^e \big)\notag \\
& \overset{\eqref{e:JdotJ}}{\leq}&  (1 - E(\eps) )^{-2}(2 k_\eps +2)  \min_m \,  \mu^{-1,1}_{\eps,(-\Le,\Le)}  \big(  \mathcal{J}_{m,\eps}^{e,*} \big)\notag \\
& \overset{\eqref{e:ee}}{\leq}&  (1 - E(\eps) )^{-2} \;\frac{ ( k_\eps +1)}{ \tilde{M}_\eps}.\label{upb}
\end{eqnarray}

The proof of the lower bound now follows along similar lines:
\begin{eqnarray}
&&\mu^{-1,1}_{\eps,(-\Le,\Le)} \big(u\in\complement\mathcal{A}_1\colon \text{$u$ has a $\delta^{-}$ up layer $\leq 2\ell$ in $[-\de,\de]$ }\big) \notag\\
&& \qquad \qquad+ \mu^{-1,1}_{\eps,(-\Le,\Le)}\big( \Aa_1\big)  \notag\\
&\overset{\eqref{e:mee}}\geq &\sum_{m = -m_\eps}^{m_\eps-1} \mu^{-1,1}_{\eps,(-\Le,\Le)}  \big(  \mathcal{J}^{e,*}_{m,\eps} \big) \notag\\
& \geq& 2m_\eps \min_{m}  \mu^{-1,1}_{\eps,(-\Le,\Le)}  \big(  \mathcal{J}^{e,\ast}_{m,\eps} \big)\notag\\
&\overset{\eqref{e:JdotJ}}{ \geq}&  \big(1 - E(\eps) \big) \, 2m_\eps \min_{m}  \mu^{-1,1}_{\eps,(-\Le,\Le)}  \big(  \mathcal{J}^{e}_{m,\eps} \big)\notag\\
& \overset{\eqref{e:nfs2}}{\geq}& \big(1 - E(\eps) \big)^2  \, 2 m_\eps   \max_{k}  \mu^{-1,1}_{\eps,(-\Le,\Le)}  \big(  \mathcal{J}_{k,\eps} \big)  \notag \\
&\overset{\eqref{e:mindeins}}{ \geq}& \big(1 - E(\eps) \big)^2 \big( 1- \mu^{-1,1}_{\eps,(-\Le,\Le)} ( \Aa_1 \cup \Aa_2  ) \big)  \frac{ m_\eps}{ M_\eps} .\label{lowb}
\end{eqnarray}
Now, from the assumptions~\eqref{sqar} on $k_\eps$ and $M_\eps$  as well as the assumptions~ \eqref{e:mee} on $m_\eps$ and $\tilde{M}_\eps$,  we have that
\begin{equation}
1\lessapprox\frac{  m_\eps}{ M_\eps}  \frac{\Le}{\de}  \leq 1 \qquad \text{and} \qquad   1\leq\frac{ k_\eps+1}{ \tilde{M}_\eps}  \frac{\Le}{\de} \lessapprox 1.\label{smallfrc}
\end{equation}
Moreover, if we choose for instance $M\geq 4\,C_2\,c_0$ in the bound~\eqref{sq} on $\Aa_1$, we recover
\begin{equation}
\Le\,\mu^{-1,1}_{\eps,(-\Le,\Le)} \big( \Aa_1 \big)  \ll 1\ll d_\eps.\label{nom}
\end{equation}
Combining~\eqref{upb},~\eqref{lowb},~\eqref{e:A2bou},~\eqref{smallfrc}, and~\eqref{nom} establishes~\eqref{step5}, as desired.

\underline{Step 5.}
It remains to remove the restriction on the length of the layer and improve from a $\delta^-$ up layer to a full up layer.

The upper bound is immediate, since
\begin{eqnarray*}
\lefteqn{\mu^{-1,1}_{\eps,(-\Le,\Le)} \big(u\in\complement\mathcal{A}_1\text{ and there exists an up layer in }[y-d_\eps,y+d_\eps]\big)}\\
&\leq&\mu^{-1,1}_{\eps,(-\Le,\Le)} \big(u\in\complement\mathcal{A}_1\text{ and }\delta^- \text{ up layer in }[y-d_\eps,y+d_\eps]\big)\\
&\leq& \mu^{-1,1}_{\eps,(-\Le,\Le)} \big(u\in\complement\mathcal{A}_1\text{ and }\delta^- \text{ up layer }\leq 2\ell\text{ in}\\
&&\qquad\qquad\qquad[y-d_\eps,y+d_\eps]\big)\\
&&+
\mu^{-1,1}_{\eps,(-\Le,\Le)} \big(u\in\complement\mathcal{A}_1\text{ and }u\in [-1+\delta, 1-\delta]\text{ on}\\
 &&\qquad\qquad\qquad\text{ all of }[x_k,x_{k+1}]\text{ for some }k\big)\\
&\overset{\eqref{unif2},\eqref{e:A2bou}}{\leq}& (1+\alpha)\frac{d_\eps}{\Le}+ \Le\exp \Big( -  \frac{ \delta^2\ell }{C_1\eps}  \Big)\\
&\leq &(1+2\alpha)\frac{d_\eps}{\Le},
\end{eqnarray*}
for $\ell$ large enough with respect to $1/\delta^2$.

For the lower bound, on the other hand, we use Step 2 once more. To this end, we will consider layers falling strictly interior to $J_{y,\eps}$ on the subset $J_{y,\eps}^\odot:=[y-d_\eps+h_\eps+3\ell,y+d_\eps-h_\eps-3\ell]$. Then, according to Step 2, there is a high probability of hitting $\pm 1$ on $J_{y,\eps}\setminus J_{y,\eps}^\odot$. More precisely, notice that we can estimate
\begin{align}
\lefteqn{\mu^{-1,1}_{\eps,(-\Le,\Le)} \big(u\in\complement\mathcal{A}_1\text{ and there exists an up layer in }[y-d_\eps,y+d_\eps]\big)}\notag\\
&\geq \mu^{-1,1}_{\eps,(-\Le,\Le)} \big(u\in\complement\mathcal{A}_1\text{ and there exists a $\delta^-$ up layer $\leq 2\ell$ in }J_{y,\eps}^\odot\notag\\
&\qquad\qquad\qquad\text{ and $u$ hits $-1$ in $(y-d_\eps,y-d_\eps+h_\eps+3\ell)$}\notag\\
&\qquad\qquad\qquad\text{ and $u$ hits $+1$ in $(y+d_\eps-h_\eps-3\ell,y+d_\eps)$}\big)\notag\\
&\geq\big(1 - E(\eps) \big)\,\mu^{-1,1}_{\eps,(-\Le,\Le)}\big( u\in\complement\mathcal{A}_1\text{ and }\notag\\
&\qquad\qquad\qquad\text{ there exists a $\delta^-$ up layer $\leq 2\ell$ in }J_{y,\eps}^\odot \big),\label{piks}
\end{align}
where in the last line, we have applied Lemma~\ref{l:hittingz}. On the other hand the probability on the last line can be estimated
\begin{align}
\lefteqn{\mu^{-1,1}_{\eps,(-\Le,\Le)}\big( u\in\complement\mathcal{A}_1\text{ and  $\delta^-$ up layer $\leq 2\ell$ in }J_{y,\eps}^\odot \big)}\notag\\
&\geq \mu^{-1,1}_{\eps,(-\Le,\Le)}\big( u\in\complement\mathcal{A}_1\text{ and $\delta^-$ up layer $\leq 2\ell$ in }J_{y,\eps} \big)\notag\\
&\qquad- \mu^{-1,1}_{\eps,(-\Le,\Le)}\big( u\in\complement\mathcal{A}_1\text{ and $\delta^-$ up layer $\leq 2\ell$}\notag\\
  &\qquad\qquad\qquad\qquad\qquad\text{in }(y-d_\eps,y-d_\eps+h_\eps+5\ell) \big)\notag\\
&\qquad- \mu^{-1,1}_{\eps,(-\Le,\Le)}\big( u\in\complement\mathcal{A}_1\text{ and $\delta^-$ up layer $\leq 2\ell$}\notag\\
 &\qquad\qquad\qquad\qquad\qquad\text{ in }(y+d_\eps-h_\eps-5\ell,y+d_\eps) \big).\label{piks2}
\end{align}
Applying the bound~\eqref{unif2} to each term in~\eqref{piks2} and substituting into~\eqref{piks} completes the lower bound.

Finally, recalling the bound~\eqref{sq} on the probability of $\mathcal{A}_1$ completes the proof of Theorem~\ref{t:uniform}.

\end{proof}

\section{Proofs of the Lemmas}
\label{s:lemmas}
\subsection{Proofs of preliminary energy lemmas}\label{ss:enlem}

The energy lemmas rely on \emph{upper bounds} and \emph{lower bounds} for the energy over various sets. The upper bounds are derived based on constructions. (The minimum value of the energy is necessarily less than or equal to the value that we can achieve with any given construction.) The lower bound, on the other hand, describes the best possible value for \emph{any function} and is based on the so-called Modica-Mortola trick discussed in Section~\ref{s:detac}. Before we begin, we make a remark about our constructions.

\begin{remark}\label{r:constr}
In addition to giving us an ODE for the energy minimizer on $\R$, equation~\eqref{absode} serves as the backbone for the \emph{constructions} that are used to establish upper bounds for energy minimization problems on finite systems. For instance, suppose we want to minimize the energy on $(-\ell,\ell)$ subject to $u(\pm\ell)=\pm 1$. For $\ell$ large, we can build a construction that almost achieves the cost $c_0$. Specifically, consider the centered solution of~\eqref{ode} on $(-\ell+a,\ell-a)$ for $a=1/\ell$. Linearly interpolate from its value at $-\ell+a$ to $-1$ at $-\ell$, and symmetrically at the other end. Because of the exponential convergence of the minimizer to $\pm 1$ (cf., Lemma~\ref{l:expdecmin}), the energy on $(-\ell,-\ell+a)$ and $(\ell-a,\ell)$ is $o(1)$ as $\ell\uparrow\infty$. Similarly, if we minimize the energy over functions satisfying $u(\pm\ell)=\pm M$ for $M$ large, we can build a piecewise-defined construction that goes from $-M$ at $-\ell$ to a neighborhood of $-1$ at $-\ell/2$,
goes from a neighborhood of $-1$ at $-\ell/2+a$ to a neighborhood of $1$ at $\ell/2-a$, and goes from a neighborhood of $1$ at $\ell/2$ to $M$ at $\ell$, with linear interpolation near $\pm\ell/2$ to make the function continuous. The cost of such a construction is
\begin{align*}
\int_{-M}^{-1}\sqrt{2V(u)}\,du +\int_{-1}^1 \sqrt{2V(u)}\,du +\int_1^M \sqrt{2V(u)}\,du +o(1)_{\ell\uparrow\infty},
\end{align*}
where we write the integrals separately to emphasize the additivity of the energy over the three subintervals described above. Because according to~\eqref{absode} we can get a good bound using increasing or decreasing functions, the analogous bounds hold for $u(\pm \ell)=\mp M$, $u(\pm \ell)=M$, et cetera.

If $M$ is very large, the constant $\ell_*$ in the energy lemmas may also need to be very large in order to make the $o(1)$ term small. The idea in all of the following proofs is to make this term small enough so that it can be absorbed into a $\delta$-dependent term, so the ordering of the constants is important: We fix $M$ (large) and $\delta$ (small) and then choose $\ell_*$ large enough so that the term(s) that are $o(1)$ with respect to $\ell$ can be absorbed.
\end{remark}

In what follows, it will be convenient to introduce the notation:
\begin{align*}
\varphi_{-1}(u)=\bigg| \int_{-1}^u \sqrt{2V(s)}\,ds\bigg|,\qquad\qquad \varphi_{+1}(u)=\bigg |\int_u^1\sqrt{2V(s)}\,dx\bigg|.
\end{align*}

\begin{proof}[Proof of Lemma~\ref{le:lazy}]
We will establish~\eqref{e:lazy} via an upper bound on the energy over $\mathcal{A}^{\rm bc}$ and a lower bound on the energy over $\mathcal{A}^{\rm bc}_0$. Because of the extra condition in $\A_0^{\rm bc}$, the energy on $(-\ell,\ell)$ is large (of order $\delta^2\ell$), and we do not have to be as careful about the boundary conditions as usual. A rough bound will suffice.

\underline{Step 1}. As explained in Remark~\ref{r:constr}, the upper bound relies on a construction. Given any $u_-\in[-M,M]$, we can use the solution of~\eqref{absode} to connect to a neighborhood of $1$ or $-1$, and similarly for $u_+$. If the optimal connection for $u_-$ is to $-1$ and the optimal connection for $u_+$ is to $+1$, then in order to build a continuous construction, we incur the additional cost
$\phim(1)=c_0$, where we have used the notation introduced above and recalled the value of $c_0$ from~\eqref{c0}. (If the optimal connection for $u_-$ and $u_+$ is to the same value, then the construction does not incur this extra cost, but the upper bound is still valid.) Putting together these three pieces of the construction and the small correction terms for continuity (see Remark~\ref{r:constr}), we can express the upper bound derived in this way as:
\begin{align}
\inf_{u\in\A^{\rm bc}} E_{(-2\ell,2\ell)}(u)
&\leq \min\{\phim(u_-),\phip(u_-)\}\notag\\
&+\min\{\phim(u_+),\phip(u_+)\}+c_0+o(1)_{\ell\uparrow\infty}.\label{23up}
\end{align}
Note that Assumption~\ref{ass:V} allows that $o(1)_{\ell\uparrow\infty}$ may depend on $M$: If $u_-$ is very large, the (near) optimal connection from $u_-$ to $1$ requires a lot of space. This explains why $\ell_*$ in the statement of the lemma depends on $M$.

\underline{Step 2}. Now we turn to the lower bound over $\A_0^{\rm bc}$. On the one hand, on $(-\ell,\ell)$, the condition in $\A_0^{\rm bc}$ implies that the integral of $V$ over $(-\ell,\ell)$ cannot be too small. Using the quadratic behavior of $V$ near $\pm 1$ (see Assumption~\ref{ass:V}), we have for $\delta$ small enough
\begin{align}
E_{(-\ell,\ell)}(u)\geq \int_{-\ell}^{\ell} V(u)\,du\geq \frac{V''(1)\,\ell\,\delta^2}{2}.\label{23ll}
\end{align}

To integrate over the rest of the interval, we recall the trick of Modica and Mortola that was explained in Section~\ref{s:detac}. Consider first $(-2\ell,-\ell)$. We divide into two cases: $|u_-|> 1$ and the complement.

If $u\in\A_0^{\rm bc}$ and $|u_-|> 1$, then there is a point $x_-\in(-2\ell,-\ell)$ such that $|u(x_-)|=1$. In this case, the Modica-Mortola trick on $(-2\ell,x_-)$ gives
\begin{align}
E_{(-2\ell,-\ell)}(u)\geq E_{(-2\ell,x_-)}(u)&\geq \min\{\phim(u_-),\phip(u_-)\}.\label{23um}
\end{align}
On the other hand if $|u_-|\leq 1$, then for $\ell$ large enough, we have
\begin{align}
\min\{\phim(u_-),\phip(u_-)\}\leq \frac{V''(1)\,\ell\,\delta^2}{8}.\label{23umb}
\end{align}

If $|u_-|>1$, then adding the contributions from~\eqref{23ll} and~\eqref{23um} and subtracting the contribution from~\eqref{23up} gives
\begin{align*}
\inf_{u\in\A_0^{\rm bc}}&E_{(-2\ell,2\ell)}(u)-\inf_{u\in\A^{\rm bc}}E_{(-2\ell,2\ell)}(u)\\
&\geq \frac{V''(1)\,\ell\,\delta^2}{2}-\min\{\phim(u_+),\phip(u_+)\}-c_0+o(1)_{\ell\uparrow\infty}.
\end{align*}
On the other hand if $|u_-|\leq 1$, then the contributions from~\eqref{23ll} and~\eqref{23up} together with the bound from~\eqref{23umb} imply
\begin{align*}
\inf_{u\in\A_0^{\rm bc}}&E_{(-2\ell,2\ell)}(u)-\inf_{u\in\A^{\rm bc}}E_{(-2\ell,2\ell)}(u)\\
&\geq \frac{3V''(1)\,\ell\,\delta^2}{8}-\min\{\phim(u_+),\phip(u_+)\}-c_0+o(1)_{\ell\uparrow\infty}.
\end{align*}
Since this is a weaker bound, it holds in either case.

Repeating the identical argument on $(\ell,2\ell)$ and in addition absorbing $c_0$ by $V''(1)\ell\delta^2/8$ gives
\begin{align*}
\inf_{u\in\A_0^{\rm bc}}&E_{(-2\ell,2\ell)}(u)-\inf_{u\in\A^{\rm bc}}E_{(-2\ell,2\ell)}(u)\\
&\geq \frac{V''(1)\,\ell\,\delta^2}{8}+o(1)_{\ell\uparrow\infty},
\end{align*}
which completes the proof of Lemma~\ref{le:lazy}.
\end{proof}

\begin{proof}[Proof of Lemma~\ref{l:cl}]
We rewrite the set  $\Aa_0^{\rm bc}$ as
\begin{equation*}
\Aa_0^{\rm bc} = \Aa_- \cup \Aa_+,
\end{equation*}
where the $\Aa_{\pm}$ are the sets of paths that perform a wasted excursion starting from a neighborhood of $\pm 1$. We will prove the bound on the energy difference for $\Aa_+$. The corresponding bound for $\Aa_-$ follows in the same way.

As usual, our task is to produce appropriate upper and lower bounds.

\underline{Step 1}. The upper bound on $\inf_{\mathcal{A}^{\rm bc}}E_{(-2\ell,2\ell)}(u)$ is by construction.  Consider the function $\bar u$ that minimizes $E_{(-2\ell,2\ell)}$ subject to
$$u(\pm 2\ell)=u_{\pm },\quad u(0)=1,$$
and notice that
\begin{align}
E_{(-2\ell,2\ell)}(\bar u)&=\inf\left\{\int_{-2\ell}^0\frac{1}{2}(\partial_x u)^2+V(u)\,dx\colon
u(0)=1, u(-2\ell)=u_-\right\} \notag\\
&\quad+\inf\left\{\int_0^{2\ell}\frac{1}{2}(\partial_x u)^2+V(u)\,dx\colon
u(0)=1, u(2\ell)=u_+\right\}\notag\\
&= \phip(u_-)+\phip(u_+)+o(1)_{\ell\uparrow\infty},\notag
\end{align}
uniformly for $u_\pm\in[-M,M]$. (This can be established by building a construction by hand, as we have explained in Remark~\ref{r:constr} and the proof of Lemma~\ref{le:lazy}.)
Hence, since $\bar u \in \mathcal{A}^{\rm bc}$, we have the (not necessarily tight) upper bound
\begin{eqnarray}
\inf_{\mathcal{A}^{\rm bc}}E_{(-2\ell,2\ell)}(u)\leq E_{(-2\ell,2\ell)}(\bar u){=}\phip(u_-)+\phip(u_+)+o(1)_{\ell\uparrow\infty}
.\label{M.e2}
\end{eqnarray}

\underline{Step 2}. We now turn to the lower bound on $\inf_{\Aa_+}E_{(-2\ell,2\ell)}(u)$. Recall the points $x_\pm$ that follow from the definition of $\A_+$ and Definition~\ref{def:wasted}. Because of the properties of the potential, we may without loss of generality assume that $u(x_\pm)=1-\delta$ and $u(x_0)=\delta$.

We now use the Modica-Mortola trick on $(-2\ell,x_-)\cup(x_+,2\ell)$ to recover
\begin{eqnarray}
&&\hspace{-40pt}E_{(-2\ell,x_-)}(u)+E_{(x_+,2\ell)}(u)\notag\\
&\geq& \phip(u_-)+\phip(u_+)-C\delta^2\notag\\
&\overset{\eqref{M.e2}}\geq & \inf_{u\in\A^{\rm bc}}E_{(-2\ell,2\ell)}(u)-C\delta^2-o(1)_{\ell\uparrow\infty}.\label{M.e11}
\end{eqnarray}
On the other hand, applying the Modica-Mortola trick on $(x_-,x_0)\cup(x_0,x_+)$ gives
\begin{align}
E_{(x_-,x_0)}(u)+E_{(x_0,x_+)}(u)\geq 2\int_\delta^{1-\delta}\sqrt{2V(u)}\,du\overset{\eqref{c0}}{=}c_0-C\delta.
\label{M.e12}
\end{align}
Combining~\eqref{M.e11} and~\eqref{M.e12} completes the proof of Lemma~\ref{l:cl}.
\end{proof}

\begin{proof}[Proof of Lemma~\ref{l:cll}]
\underline{Step 1}. For the upper bound over $\A_{\delta,pre}^{\rm bc}$, we use the function $\bar{u}$ that minimizes the energy subject to
\begin{align*}
u(\pm 2\ell)=u_{\pm},\quad u(\pm \ell)=-1-2\delta, \quad u(0)=\delta.
\end{align*}
As in the proof of Lemma~\ref{l:cl}, we observe that $\bar u\in \A_{\delta,pre}^{\rm bc}$ and hence the construction gives an upper bound
\begin{eqnarray}
\inf_{\A_{\delta,pre}^{\rm bc}}E_{(-2\ell,2\ell)}(u)&\leq&\phim(u_-)+\phim(u_+)+2\phip(0)+C\delta+o(1)_{\ell\uparrow\infty}\notag\\
&\overset{\eqref{c0}}=&\phim(u_-)+\phim(u_+)+c_0+C\delta+o(1)_{\ell\uparrow\infty}.\label{28up}
\end{eqnarray}

\underline{Step 2}. For the lower bound over $\A^{\rm bc}$, we observe that for any $u\in\A^{\rm bc}$, either there is a point $x_-\in(-2\ell,0)$ and a point $x_+\in(0,2\ell)$ such that $u(x_\pm)$ is in a $\delta$ neighborhood of $1$ or $-1$, or else the energy (by the same argument as in the proof of Lemma~\ref{le:lazy}) is bounded below by $\delta^2\ell V''(1)/2$ for $\delta$ small enough. We can choose $\ell$ so large that this is greater than $\phim(u_-)+\phim(u_+)$ and hence dominates the boundary terms in~\eqref{28up}. On the other hand, if the points $x_\pm$ exist, then by the usual trick of Modica and Mortola, we recover
\begin{align}
&\hspace{-20pt}\inf_{\A^{\rm bc}}E_{(-2\ell,2\ell)}(u)\notag\\
&\geq \min\{\phim(u_-),\phip(u_-)\}+\min\{\phim(u_+),\phip(u_+)\} - C\delta\notag\\
&= \phim(u_-)+\phim(u_+) - C\delta,\label{28low}
\end{align}
where the second line follows by virtue of the boundary conditions $u_\pm\in[-M,0]$ and the symmetry of the potential.

The combination of~\eqref{28up} and~\eqref{28low} completes the proof of Lemma~\ref{l:cll}.
\end{proof}

\subsection{Proof of the strong Markov property}\label{ss:62}
\begin{proof}[Proof of Lemma \ref{le:Markov1}]
By subtracting $h^{u_-,u_+}_{(x_-,x_+)}$, we  can reduce the problem to the case of zero boundary conditions. Under $\,  \mathcal{W}^{0,0}_{\eps, (x_{-},x_+)}$, $u-u_{\hat{x}_-}^{\hat{x}_+} $ and $ u_{\hat{x}_-}^{\hat{x}_+}$ are jointly Gaussian and centered, because they are both linear images of $u$. So it is sufficient to calculate their covariances. Using~\eqref{e:Cov}, it is easy to see that, for all $x_1,x_2 \in [x_-, x_+]$, one has
\begin{equation*}
 \mathbb{E}_{(x_{-},x_{+})}^{\mathcal{W}_\eps,0,0} \Big((u-u_{\hat{x}_-}^{\hat{x}_+})  \big(x_1\big) \, u_{\hat{x}_-}^{\hat{x}_+} \big( x_2 \big) \Big) \, = \, 0,
\end{equation*}
and for $x_1,x_2 \in [\hat{x}_-,\hat{x}_+]$, one has
\begin{align*}
 \mathbb{E}_{(x_-,x_{+})}^{\mathcal{W}_\eps,0,0} &\Big( (u-u_{\hat{x}_-}^{\hat{x}_+}) \big(x_1\big) \, (u-u_{\hat{x}_-}^{\hat{x}_+})  \big( x_2 \big)  \Big) \, \\
 &= \,   \frac{\eps}{\hat{x}_+ - \hat{x}_-}  \Big(  (x_1 - \hat{x}_-)(\hat{x}_+ - x_2)  \wedge (x_2  - \hat{x}_-)(\hat{x}_+ -x_1) \Big).
\end{align*}
This shows the claim.
\end{proof}

\begin{proof}[Proof of Lemma \ref{le:Markov1b}]
We start by observing that the statement of  Lemma \ref{le:Markov1} implies that
\begin{equation}\label{e:pM1}
 \mathbb{E}_{(x_{-},x_{+})}^{\mathcal{W}_\eps,u_{-},u_{+}} \big( \Phi  \big|  \F_{[x_-,\hat{x}_-]}  \vee  \F_{[\hat{x}_+,x_+]  }  \big) \, = \,    \,  \mathbb{E}_{(\hat{x}_{-},\hat{x}_{+})}^{\mathcal{W}_\eps,{\bf u}} \big( \Phi  \big) .
\end{equation}

In order to prove the desired  statement~\eqref{e:Markovmu}, observe that the density of $  \mu^{u_-,u_+}_{\eps,(x_-,x_+)}$ with respect to $  \mathcal{W}^{u_{-},u_+}_{\eps, (x_{-},x_+)}$ can be written as
\begin{align*}
\exp\Big( - \frac{1}{\eps}\int_{x_-}^{x_+} V( u) \, dx \Big) \, = \, \Psi_- \, \Psi_\odot \, \Psi_+,
\end{align*}
where %
\begin{align*}
\Psi_- :=  &\exp\Big(     - \frac{1}{\eps}\int_{x_-}^{\hat{x}_-} V( u) \, dx \Big), \qquad \Psi_+ :=  \exp\Big( - \frac{1}{\eps} \int_{\hat{x}_+}^{x_+} V(u) \, dx \Big), \\
\text{and}\;\; \Psi_\odot := & \exp\Big( - \frac{1}{\eps} \int_{\hat{x}_-}^{\hat{x}_+} V(u) \, dx \Big)
\end{align*}
are measurable with respect to $\F_{[x_-, \hat{x}_-]}$,  $\F_{[\hat{x}_-, \hat{x}_+]}$, and  $\F_{[\hat{x}_+, x_+]}$. Suppose that test functions  $\Xi_-$ and $\Xi_+$ are measurable with respect to $\F_{[x_-, \hat{x}_-]}$ and  $\F_{[\hat{x}_+, x_+]}$. Then we get
\begin{eqnarray}
 &\mathbb{E}_{(x_{-},x_{+})}^{\mu_\eps,u_{-},u_{+}} \Big( & \Phi  \, \Xi_- \,  \Xi_+ \Big) \label{e:6.5}\\
  &\, =  &  \frac{1}{  \mathcal{Z}^{u_-,u_+}_{\eps, (x_-,x_+)}}   \mathbb{E}_{(x_{-},x_{+})}^{\mathcal{W}_\eps,u_{-},u_{+}} \Big( \Xi_- \,  \Psi_- \, \Phi \,   \Psi_\odot \Xi_+  \Psi_+ \Big) \notag\\
 &\overset{\eqref{e:pM1}}=  & \frac{1}{  \mathcal{Z}^{u_-,u_+}_{\eps, (x_-,x_+)}}   \mathbb{E}_{(x_{-},x_{+})}^{\mathcal{W}_\eps,u_-,u_+} \Big( \Xi_- \,  \Psi_-   \mathbb{E}_{(\hat{x}_{-},\hat{x}_{+})}^{\mathcal{W}_\eps,{\bf u}} \big( \, \Phi \,   \Psi_\odot  \big) \Xi_+  \Psi_+ \Big) \notag\\
 & =  & \frac{1}{  \mathcal{Z}^{u_-,u_+}_{\eps, (x_-,x_+)}}   \mathbb{E}_{(x_{-},x_{+})}^{\mathcal{W}_\eps,u_-,u_+} \Big( \Xi_- \,  \Psi_-   \mathbb{E}_{(\hat{x}_{-},\hat{x}_{+})}^{\mathcal{W}_\eps,{\bf u}} \big( \Psi_{\odot}  \big)   \mathbb{E}_{(\hat{x}_{-},\hat{x}_{+})}^{\mu_\eps,{\bf u}} \big( \, \Phi   \big) \Xi_+  \Psi_+ \Big) \notag \\
 &\overset{\eqref{e:pM1}}=  &    \mathbb{E}_{(x_{-},x_{+})}^{\mu_\eps,u_-,u_+} \Big( \Xi_- \,      \mathbb{E}_{(\hat{x}_{-},\hat{x}_{+})}^{\mu_\eps,{\bf u}} \big( \, \Phi   \big)\, \Xi_+ \Big).\notag
\end{eqnarray}
This finishes the proof of Lemma \ref{le:Markov1b}.
\end{proof}

We are now ready to give a proof of the strong Markov property.

\begin{proof}[Proof of Lemma~\ref{p:Markov}:]
We treat only the Gaussian case~\eqref{e:strMarkovW}. Equation~\eqref{e:strMarkovmu} then follows as in the proof of Lemma \ref{le:Markov1b}.

We start by proving~\eqref{e:strMarkovW} in the case in which  $\chi_-$ and $\chi_+$ are left and right stopping points that attain values in a finite set  $\big\{ \chi^1, \ldots, \chi^N \big\}$. Then we can write

\begin{eqnarray*}
&&\mathbb{E}_{(x_{-},x_{+})}^{\mathcal{W}_\eps,u_{-},u_{+}} \big( \Phi \big|   \F_{[x_-,\chi_-]} \vee \F_{[\chi_+, x_+]}   \Big) \\
&\, = \,&  \sum_{n=1}^N \sum_{m=1}^N \mathbb{E}_{(x_{-},x_{+})}^{\mathcal{W}_\eps,u_{-},u_{+}} \Big( \Phi\,  \mathbf{1}_{\{\chi_- = \chi^n \}} \mathbf{1}_{\{\chi_+= \chi^m \}}  \big|   \F_{[x_-,\chi_-]} \vee \F_{[\chi_+, x_+]} \Big)\\
&\, = \,&  \sum_{n=1}^N \sum_{m=1}^N    \mathbf{1}_{\{\chi_- = \chi^n \}} \mathbf{1}_{\{\chi_+= \chi^m \}}  \mathbb{E}_{(x_{-},x_{+})}^{\mathcal{W}_\eps,u_{-},u_{+}} \Big( \Phi\, \big|   \F_{[x_-,\chi^n]} \vee \F_{[\chi^m, x_
+]} \Big)\\
&\,\overset{\eqref{e:pM1}}{ =}&   \sum_{n=1}^N \sum_{m=1}^N    \mathbf{1}_{\{\chi_- = \chi^n \}} \mathbf{1}_{\{\chi_+= \chi^m \}}  \,  \mathbb{E}_{(\chi^n,\chi^m)}^{\mathcal{W}_\eps,{\bf u}} \big( \Phi\, \big)\\
&\,= \,&    \mathbb{E}_{(\chi_-,\chi_+)}^{\mathcal{W}_\eps,{\bf u}} \big( \Phi\, \big).
\end{eqnarray*}
In the second equality, we have used the fact that the $\chi_{\pm}$ are left and right stopping points.

In order to see the general case, we approximate the stopping points by
\begin{align*}
\chi^N_- \, := \,& \inf \big\{  x \,= \, i \, 2^{-N} \colon \, i  \in \Z , \, x \geq \chi_- \big\},\\
\chi^N_+ \, := \,& \sup \big\{   x \, = \,  i  \, 2^{-N}\colon i  \in \Z , \, x \leq \chi_+ \big\}.
\end{align*}
Then $\chi^N_-$ and $\chi^N_+$ are stopping points taking values in a finite set and, in particular,~\eqref{e:strMarkovW} holds for them. We have
\begin{equation*}
\chi^N_-  \downarrow \chi_- \qquad \text{and} \qquad  \chi^N_+  \uparrow \chi_+ \qquad \text{as }N \uparrow \infty.
\end{equation*}
Now, in order to conclude that~\eqref{e:strMarkovW} also holds for $\chi_{\pm}$, we first observe that  for any continuous, bounded $\Phi \colon C([x_-,x_+]) \to \R$,  we have for every path $u$ that
\begin{eqnarray*}
&& \mathbb{E}_{(\chi_{-},\chi_{+})}^{\mathcal{W}_\eps,{\bf u}} \big( \Phi  \big) \notag\\
& \,=\, & \lim_{N \to \infty } \mathbb{E}_{(\chi^N_{-},\chi^N_{+})}^{\mathcal{W}_\eps,{\bf u}} \big( \Phi  \big) \\
&\, =\, & \lim_{N \to \infty } \mathbb{E}_{(x_{-},x_{+})}^{\mathcal{W}_\eps,u_{-},u_{+}} \big( \Phi \big|   \F_{[x_-,\chi^N_-]} \vee \F_{[\chi^N_+, x_+]}   \big)\\
&\, =\, &  \mathbb{E}_{(x_{-},x_{+})}^{\mathcal{W}_\eps,u_{-},u_{+}} \big( \Phi \big|   \F_{[x_-,\chi_-]} \vee \F_{[\chi_+, x_+]}   \big).
\end{eqnarray*}
In the first step, we have used  that, due to the continuity of $u$,  the measures $\mathcal{W}^{u}_{\eps, (\chi^N_{-},\chi^N_+)}$ converge weakly to $\mathcal{W}^{u}_{\eps, (\chi_{-},\chi_+)}$, as can easily be confirmed. In order to see the last line, it suffices to check that the limit in the third line does indeed satisfy the characteristic properties of a conditional expectation.

This equality can then be extended to arbitrary test functions $\Phi$ with a standard monotone class argument (see e.g. \cite[Ch. 0, Thm 2.2]{RY}).

\end{proof}

\subsection{Proof of large deviation bounds}\label{ss:63}
The large deviation bounds~\eqref{e:LDUB1} and~\eqref{e:LDLB1}  are statements about the quotient of expectations of the form
\begin{equation*}
 \mathbb{E}_{(x_{-},x_{+})}^{\mathcal{W}_\eps,u_{-},u_{+}} \Big(\mathbf{1}_{\Aa} (u)\,  \exp \Big(- \frac{1}{\eps} \int_{x_-}^{x_+} V(u) \, dx \Big)\Big),
\end{equation*}
see~\eqref{m30.1}. Consequently, the results will follow as soon as we establish upper and lower bounds on these expectations. Throughout this subsection, $\Aa$ will always denote a set of continuous paths $u$ on $[x_-,x_+]$ that satisfy the boundary conditions $u(x_{\pm})= u_{\pm}$, and topological notions like open or closed will always refer to the topology of uniform convergence.  We will frequently use $I_{x_-,x_+}(u)$, the Gaussian energy of a path (defined in~\eqref{e:I}),  and $\Ixu$, the minimal Gaussian energy given the boundary conditions (defined in~\eqref{e:Iu}).

The upper bound for the Gaussian expectation can then be stated as follows.

\begin{lemma}[Upper bound]\label{le:LDUpBou}
Fix constants $M< \infty$, $0 < \ell_- < \ell_{+} < \infty$ and $R < \infty$. Suppose that $\ell=(x_+-x_-) \in [\ell_- , \ell_+]$ and  $u_{\pm} \in [-M,M]$. Then for any $\delta, \gamma >0$, there exists an $\eps_0>0$  such that for any measurable  set $\Aa$  satisfying
\begin{equation}\label{e:condA}
\inf_{u \in B(\Aa,\delta)}E(u) - \Ixu  \leq R
\end{equation}
and for any $\eps \leq \eps_0$, we have
\begin{align}
 \mathbb{E}_{(x_{-},x_{+})}^{\mathcal{W}_\eps,u_{-},u_{+}} \Big(\mathbf{1}_{\Aa}(u) \, &  \exp \Big(- \frac{1}{\eps}\int_{x_-}^{x_+} V(u) \, dx \Big)\Big) \notag\\
 & \leq \,    \exp  \Big(-\frac{1}{\eps} \big(  \inf_{u \in B(\Aa,\delta)} E(u) -\Ixu - \gamma   \big)  \Big) \label{e:LDUB} .
\end{align}
Here  $\eps_0$ depends on $M, \ell_{\pm},  \sup_{|v| \leq M +\sqrt{2^{-1} (\ell_+ R +1)}+ 1}  |V'(v)| , \delta,$ and $\gamma$ but not on the particular choice of $x_\pm, u_\pm$, and it depends on $\mathcal{A}$ only through condition~\eqref{e:condA}.
\end{lemma}

As usual in large deviation theory, the derivation of lower bounds for integrals is reduced to the case of a ball
\begin{equation*}
B(u_{\ast},\delta) := \big\{u \colon \|u - u_{\ast} \|_\infty \leq \delta    \big\}
\end{equation*}
around a suitably chosen profile $u_{\ast}$.
\begin{lemma}[Lower bound]\label{le:LDLoBou}
Fix constants $M$ and $\ell_{+} < \infty$. Suppose that $\ell= x_+-x_- \leq \ell_+$, $u_{\pm} \in [-M,M]$. Then for any profile  $u_{\ast}$ with
\begin{equation}\label{e:condu}
\sup_{x \in (x_-, x_+)}  |u_{\ast}(x)|  \leq M
\end{equation}
and any $\delta,\gamma>0$, there exists an $\eps_0>0$ such that for $\eps \leq \eps_0$
\begin{align}
 \mathbb{E}_{(x_{-},x_{+})}^{\mathcal{W}_\eps,u_{-},u_{+}} \Big(\mathbf{1}_{B(u_{\ast},\delta)}(u) \, &  \exp \Big(- \frac{1}{\eps} \int_{x_-}^{x_+} V(u) \, dx \Big) \notag\\
 & \geq \,    \exp  \Big(- \frac{1}{\eps} \big(  E(u_{\ast}) -I_{x_{\pm}}^{u_{\pm}} + \gamma   \big)  \Big) \label{e:LDLB} .
\end{align}
Here  $\eps_0$ depends on $\sup_{|v| \leq M+1}|V'(v)| , \ell_+, \delta,$ and $\gamma$ but not on the particular choice of $x_\pm, u_\pm$ and it depends on $u_{\ast}$ only through the condition~\eqref{e:condu}.
\end{lemma}

Now we give the proofs of Lemmas \ref{le:LDUpBou} and \ref{le:LDLoBou}. The proofs of Propositions \ref{pr:LD1} and \ref{pr:LD2} are given afterwards.

In order to prove the upper bound, we will invoke the known upper bound for Gaussian large deviations. In the current context, this can be stated as follows.

\def \Bogachev {\cite[Cor. 4.9.3 ]{Bog}}

\begin{proposition}[Gaussian large deviation, see e.g. \Bogachev]\label{pr:GLD} For every \emph{closed} set $\Aa$ and for any $\gamma > 0$, there exists an $\eps_0 >0$ such that for every $\eps \leq \eps_0$ we have
\begin{equation}\label{e:GLD}
 \mathcal{W}^{0,0}_{\eps, (0,1)} \big( \Aa \big) \leq \exp \Big(- \frac{1}{\eps} \big( \inf_{u \in \Aa} I_{0,1}(u) - \gamma \big) \Big) .
\end{equation}
\end{proposition}

The argument for  Lemma \ref{le:LDUpBou} is an adaptation of the proof of \cite[p. 34]{dH00}.

\begin{proof}[Proof of Lemma \ref{le:LDUpBou}]

\underline{Step 1.} We start by reducing the general problem to the case of homogeneous boundary conditions on $[0,1]$.  To this end, we introduce the following affine transformation. 
We define the transformation $T: u \mapsto \hat{u}$, where for a given path $u\colon[x_-,x_+] \to \R$ we denote by $\hat{u} \in C([0,1])$ the function
\begin{equation}\label{e:T}
\hat{u}(x) := u\big( x_-+\ell\,x\big) - h_{0,1}^{u_-,u_+}(x).
\end{equation}
Recall from~\eqref{e:Defh} that  $h_{0,1}^{u_-,u_+}(x) = xu_+ + (1-x)u_- $. It is clear that $T$ is a bijection between the set of continuous paths $u$ on $[x_-,x_+]$ with boundary conditions $u(x_{\pm}) = u_{\pm}$ and $C([0,1])$, the space of continuous paths on $[0,1]$ with homogeneous boundary conditions. Furthermore, if $u$ is distributed according to $  \mathcal{W}^{u_{-},u_+}_{\eps, (x_{-},x_+)}$, then $\hat{u}$ is distributed according to $  \mathcal{W}^{0,0}_{\ell \eps, (0,1)}$.  Note that the variance changes due to the rescaling by $\ell$.

The expectation that we want to bound can be expressed in terms of $\hat{u}$ as
\begin{align}
 \mathbb{E}_{(x_{-},x_{+})}^{\mathcal{W}_\eps,u_{-},u_{+}} & \Big(\mathbf{1}_{\Aa} (u)\,   \exp \Big(- \frac{1}{\eps} \int_{x_-}^{x_+} V( u) \, dx \Big)\Big) \notag \\
  = \,  \mathbb{E}_{(0,1)}^{\mathcal{W}_{\ell \eps},0,0}& \Big( \mathbf{1}_{\hat{\Aa}}(\hat{u}) \,   \exp \Big(- \frac{\ell}{\eps}  \int_{0}^{1} V\big( \hat{u} +h_{0,1}^{u_-,u_+} \big) \, dx \Big)\Big) \label{e:modint},
\end{align}
where $\hat{\Aa} := \big\{ Tu \colon u \in \Aa \big\}$. On the other hand, the condition~\eqref{e:condA} and the right-hand side of the desired bound~\eqref{e:LDUB} can also be expressed in terms of $\hat{u}$, as we will now do. We have for every $u$ that
\begin{equation*}
E(u) = E_\ell^{u_{\pm}}(\hat{u}) + I_{x_{\pm}}^{u_{\pm}},
\end{equation*}
where, for convenience, we have introduced the notation
\begin{equation*}
E_\ell^{u_{\pm}}(\hat{u}) := \int_0^1 \frac{1}{2\ell} \big(  \partial_x  \hat{u} \big)^2 + \ell  V\big( \hat{u} +h_{0,1}^{u_-,u_+}\big) \, dx.
\end{equation*}
(Note that we have not included $I_{x_{\pm}}^{u_{\pm}}$ in the definition of the rescaled energy $E_\ell^{u_{\pm}}$, because this way $E_\ell^{u_{\pm}}$ will appear as the natural rate functional.)
Condition~\eqref{e:condA} can now be expressed as
\begin{equation}\label{e:condA1}
\inf_{u \in B(\hat{\Aa},\delta)}E^{u_{\pm}}_\ell(\hat{u}) \, \leq \,   R,
\end{equation}
and for the right-hand side of~\eqref{e:LDUB}, we get
\begin{align*}
   \exp  \Big(-\frac{1}{\eps} &\big(  \inf_{u \in B(\Aa,\delta)} E(u) -\Ixu - \gamma   \big)  \Big) \\
   & =    \exp  \Big(-\frac{1}{\eps}\big(  \inf_{\hat{u} \in B(\hat{ \Aa},\delta)} E_{\ell}^{u_{\pm}}(\hat{u}) - \gamma   \big)  \Big).
\end{align*}

Relabelling $\hat{\Aa}$ as $\Aa$ and $\hat{u}$ as $u$, we  conclude that it suffices to show that for every set $\Aa \subseteq C([0,1])$ satisfying
\begin{equation}\label{e:condAA1}
\inf_{u \in B(\Aa,\delta)}E^{u_{\pm}}_\ell(u) \, \leq \,   R,
\end{equation}
we have for $\eps \leq \eps_0$ that
\begin{align}
\mathbb{E}_{(0,1)}^{\mathcal{W}_{\ell \eps},0,0}&  \Big( \mathbf{1}_{\Aa}(u) \,   \exp \Big(- \frac{\ell}{\eps}  \int_{0}^{1} V\big( u + h_{0,1}^{u_{\pm}}  \big) \, dx \Big)\Big) \notag\\
& \leq \exp \Big( - \frac{1}{\eps} \big(  \inf_{ u \in B(\Aa,\delta)} E_\ell^{u_{\pm}}(u) -\gamma \big) \Big) \label{e:wish}.
\end{align}
This bound will be established in  Steps 2-4.

\underline{Step 2.}The strategy to prove~\eqref{e:wish} consists of  decomposing $C([0,1])$ into a set of paths with high Gaussian energy and a finite number of small balls with lower Gaussian energy. One can use the Gaussian large deviation bound~\eqref{e:GLD} to bound the probability of the set of high Gaussian energy, which we will make to be a term of higher exponential order by choosing the Gaussian energy high enough. Then for the balls with lower Gaussian energy, the expectation over a given ball can be estimated by bounding an exponential factor by its supremum on that ball, and then bounding the Gaussian probability of the set using~\eqref{e:GLD} again. Finally, one has to sum over all the balls. As the total number of balls is finite and the bounds decay exponentially, the largest of the summands determines the behavior.

The main difference with respect to the classical argument in \cite{dH00} is that we choose a partition of $C([0,1])$ into sets that do not depend on $\Aa$. This is necessary to ensure that the number of balls is independent of $\Aa$. The price we have to pay is that on the right-hand side of~\eqref{e:wish} we take the infimum over the small neighborhood $B(\Aa,\delta)$ of $\Aa$ instead of taking it over $\Aa$ only, as in the classical argument.

Let us now give the details: First, fix a $\gamma<1$ and let
\begin{equation}\label{e:td}
\tilde{\delta} := \gamma\bigg( \ell_+ \sup_{|v| \leq \sqrt{2^{-1} (\ell_+ R +1)}+M+ 1} |V'(v) |  \bigg)^{-1} \wedge\delta \wedge \frac{1}{2}.
\end{equation}
The sublevel set
\begin{equation*}
\mathcal{K}_{\ell_+ R} \, := \,  \Big\{ u \colon I_{0,1} \leq \ell_+ R  \Big\}
\end{equation*}
is compact in $C([0,1])$, and  we can cover it by a finite number $N_{\tilde{\delta}, \ell_+ R}$ of open balls $B(u_k, \tilde{\delta})$ of radius $\tilde{\delta}$, where $u_k\in \mathcal{K}_{\ell_+ R}$ for each $k$. Note that $\Aa$ does not enter here, so both the profiles $u_k$ and the number $N_{\tilde{\delta}, \ell_+ R}$ depend only on $\gamma,\delta, \ell_+ R$, and $\sup_{|v| \leq \sqrt{2^{-1} (\ell_+ R +1)}+M+ 1} |V'(v) |$, not on the set $\Aa$ or the specific choice of $x_{\pm}, u_{\pm}$. Actually, it can be checked using the H\"older continuity of functions with bounded $H^1$-norm that this number grows like $\exp\big( C \big( R\ell_+ \tilde{\delta}^{-1} \big)^2  \big)$.

Using this covering and the positivity of $V$, we have for any set $\Aa$ that
\begin{align}
& \mathbb{E}_{(0,1)}^{\mathcal{W}_{\ell \eps},0,0} \Big( \mathbf{1}_{\Aa}(u) \,   \exp \Big(-\frac{\ell}{\eps}  \int_{0}^{1} V\big( u + h_{0,1}^{u_{\pm}} \big) \, dx \Big)\Big) \notag \\
  &  \leq \! \sum_{k=1}^{N_{\tilde{\delta}, \ell_+ R}}  \mathbb{E}_{(0,1)}^{\mathcal{W}_{\ell \eps},0,0} \Big( \mathbf{1}_{ B(u_k,   \tilde{\delta}) \cap \Aa}(u) \,   \exp \Big(- \frac{\ell}{\eps}  \int_{0}^{1} V\big( u + h_{0,1}^{u_{\pm}}  \big) \, dx \Big)\Big) \notag\\
  & \qquad \qquad +   \mathcal{W}^{0,0}_{\ell \eps, (0,1)}\Big( \Aa \setminus \cup_k B(u_k,  \tilde{\delta}) \Big) \label{e:decobou} .
\end{align}

\underline{Step 3.} The last term in~\eqref{e:decobou} can now easily be bounded:
\begin{align}
\mathcal{W}^{0,0}_{\ell \eps, (0,1)} \Big(  \Aa \setminus \cup_k B(u_k,  \tilde{\delta}) \Big) \leq
\mathcal{W}^{0,0}_{\ell \eps, (0,1)} \Big(  \complement \cup_k B(u_k,  \tilde{\delta}) \Big).
\label{e:618}
\end{align}
The set $\mathcal{B}:=\complement \cup_k B(u_k, \tilde{\delta})$ is closed and by definition $\inf_{u \in \Bb}I_{0,1}(u) \geq \ell_+ R$. Hence, the Gaussian large deviation bound~\eqref{e:GLD} implies that there exists an $\tilde{\eps}_0>0$ such that, for $ \eps \leq  \tilde{\eps}_0$, we have
\begin{align*}
\mathcal{W}^{0,0}_{ \eps, (0,1)} \big(  \Bb \big)
&\leq \exp \Big(- \frac{1}{\eps} \big( \ell_+ R - \gamma \big) \Big).
\end{align*}
Now we choose $\eps_0 = \tilde{\eps}_0 \ell_+^{-1}$. Then, for $\eps\leq \eps_0$, we can conclude that
\begin{eqnarray}
&&\mathcal{W}^{0,0}_{\ell \eps, (0,1)} \Big(  \Aa \setminus \cup_k B(u_k,  \tilde{\delta}) \Big)\notag\\
 &\leq& \exp \Big(- \frac{1}{\ell \eps}\big( \ell_+ R - \gamma \big) \Big) \notag\\
&\leq& \exp \Big(- \frac{1}{\eps} \big(  R - \frac{ \gamma}{\ell_-} \big) \Big)\notag\\
&\overset{\eqref{e:condA1}}{\leq}&  \exp  \Big(-\frac{1}{\eps} \big(  \inf_{u \in B(\Aa,\delta)} E_\ell^{u_{\pm}}(u) -   \frac{\gamma}{\ell_-}   \big)  \Big) \label{later} .
\end{eqnarray}

\underline{Step 4.} It remains to bound the sum on the right-hand side of~\eqref{e:decobou}. Since the number of summands $N_{\tilde{\delta}, \ell_+ R}$   remains constant as $\eps \downarrow 0$, the sum is dominated by the largest summand. Specifically, after fixing $\gamma$, $\delta$, $\ell_+R$ and $M$, we can choose $\eps_0>0$ sufficiently small so that $\eps\leq \eps_0$ implies
\begin{align}\label{ritt1}
 N_{\tilde{\delta}, \ell_+ R}=\exp\left(\frac{1}{\eps}\Big(\eps\log\big(N_{\tilde{\delta}, \ell_+ R}\big)\Big)\right)\leq\exp\left(\frac{\gamma}{\eps}\right).
\end{align}
Hence, up to an extra factor of $\gamma$, it is sufficient to obtain a good exponential bound on the largest summand on the right-hand side of~\eqref{e:decobou}.

If $B(u_k,\tilde{\delta}) \cap \Aa$ is empty, the largest summand is zero. Otherwise, we have
\begin{align}
&\mathbb{E}_{(0,1)}^{\mathcal{W}_{\ell \eps},0,0}  \Big( \mathbf{1}_{B(u_k,  \tilde{\delta}) \cap \Aa}(u) \,   \exp \Big(- \frac{\ell}{\eps}  \int_{0}^{1} V\big( u+ h_{0,1}^{u_{\pm}}  \big) \, dx \Big)\Big) \notag\\
 & \quad  \leq\, \!\!\!    \sup_{u \in B(u_k,   \tilde{\delta})} \exp \Big( -\frac{\ell}{\eps}  \int_{0}^{1} V\big( u + h_{0,1}^{u_{\pm}}  \big) \, dx \Big) \,   \mathcal{W}^{0,0}_{\ell \eps,  (0 ,1)} \big( \overline{ B(u_k,  \tilde{\delta})} \big)
 \label{e:620}.
\end{align}
Due to the lower semi-continuity of $I_{0,1}$, we can choose $\tilde{u}_k \in B(u_k, \tilde{\delta})$ so that \begin{align}
I_{0,1}(\tilde{u}_k) \leq  \inf_{ u \in \overline{ B(u_k,\tilde{\delta})} } I_{0,1}(u) + \gamma.\label{bso7}
 \end{align}
Then the first factor in~\eqref{e:620} can be bounded above by
\begin{equation}
\exp \Big( - \frac{1}{\eps} \Big( \ell \int_{0}^{1} V\big( \tilde{u}_k + h_{0,1}^{u_{\pm}} \,\big) dx  -2\tilde{\delta}  \ell_+\sup_{|v| \leq\| \tilde{u}_k\|_\infty +M +1} |V'(v)|   \Big)\Big), \label{e:granted1}
\end{equation}
and we need a bound on $||\tilde{u}_k||_\infty$. First we recall that $u_k\in\mathcal{K}_{\ell_+ R}$, which by definition gives $I_{0,1}(u_k)\leq\ell_+R$. Together with the definition of $\tilde{u}_k$, this gives
\begin{align*}
I_{0,1}(\tilde{u}_k)\overset{\eqref{bso7}}\leq \ell_+ R+ \gamma.
\end{align*}
Recalling the homogeneous boundary conditions, this implies that
\begin{align*}
||\tilde{u}_k||_\infty\leq \frac12 \int_0^1|\partial_x \tilde{u}_k|\,dx\leq \frac12\left(\int_0^1 |\partial_x \tilde{u}_k|^2\,dx\right)^{1/2}\leq \sqrt{2^{-1}(\ell_+ R+1)}.
\end{align*}
Hence, the definition~\eqref{e:td} of $\tilde{\delta}$ implies that the bound in~\eqref{e:granted1} improves to
\begin{align}
\exp \Big( - \frac{1}{\eps} \Big( \ell \int_{0}^{1} V\big( \tilde{u}_k + h_{0,1}^{u_{\pm}} \,\big) dx  -2 \gamma\Big)\Big).\label{mrmr}
\end{align}

On the other hand, the Gaussian large deviation bound~\eqref{e:GLD} and the definition~\eqref{bso7} of $\tilde{u}_k$ imply that for every $k$ there exists an $\eps_0>0$ such that for $\ell_+ \eps \leq \eps_0$ we have
\begin{equation}\label{Problem}
\mathcal{W}^{0,0}_{\ell \eps,  (0 ,1)} \big( \overline{B(u_k, \tilde{\delta} )} \big) \leq \exp \Big(- \frac{1}{\ell \eps} \big( I_{0,1}\big(\tilde{u}_k\big) - 2\gamma \big)\Big) .
\end{equation}
As there are only finitely many $u_k$ (the selection of which does not depend on $\Aa$), we can find an $\eps_0$ such that this bound holds for all $\tilde{u}_k$ simultaneously and such that~\eqref{later} holds as well.

Substituting~\eqref{mrmr} and~\eqref{Problem} into~\eqref{e:620} gives for each $k$ that
\begin{eqnarray}
&&\mathbb{E}_{(0,1)}^{\mathcal{W}_{\ell \eps},0,0}  \Big( \mathbf{1}_{B(u_k, \tilde{\delta}) \cap \Aa} \,   \exp \Big(- \frac{1}{\eps} \ell \int_{0}^{1} V\big( u + h_{0,1}^{u_{\pm}}  \big) \, dx \Big)\Big) \notag\\
 &   \leq& \exp\Big(- \frac{1}{\eps} \Big(  E^{u_{\pm}}_\ell( \tilde{u}_k) - \Big(2 +  \frac{2}{\ell_-}  \Big)\gamma  \Big) \Big) \notag\\
 &   \leq& \exp\Big(- \frac{1}{\eps}\Big( \inf_{u \in B(u_k,  \tilde{\delta})} E^{u_{\pm}}_\ell(u) -\Big(2 + \frac{2}{\ell_-} \Big)\gamma  \Big) \Big)\notag\\
 &\overset{\eqref{e:td}}\leq& \exp\Big(-\frac{1}{\eps}\Big( \inf_{u \in B(u_k, {\delta})} E^{u_{\pm}}_\ell(u) -\Big(2 + \frac{2}{\ell_-} \Big)\gamma  \Big) \Big)\notag.
\end{eqnarray}
After relabelling $\gamma$ (for instance by a factor of $6$), the above bound together with~\eqref{e:decobou},~\eqref{later}, and~\eqref{ritt1} finishes the proof of~\eqref{e:wish}.

\end{proof}

The proof of the lower bound~\eqref{e:LDLB} relies on the classical Cameron-Martin Theorem. In the current context it can be stated as follows.

\def\Hairer{\cite[Thm~3.41]{Ha09}}

\begin{theorem}[Cameron-Martin Thm. e.g.\Hairer ]\label{thm:CamMart}
For a fixed $f \in$ \\$C([x_-, x_+])$, define the shift map $T_f \colon C([x_-, x_+]) \to C([x_-, x_+])$ by $T_f(u) \, = u + f$. Then the image measure $T_f^* \mathcal{W}^{0,0}_{\eps, (x_{-},x_+)}$  is absolutely continuous with respect to $\mathcal{W}^{0,0}_{\eps, (x_{-},x_+)}$ if and only if $f \in H^{1}_0 (x_{-},x_+)$. In that case the Radon-Nykodym derivative is given by
\begin{align}
 \frac{d \,  T_f^* \mathcal{W}^{0,0}_{\eps, (x_{-},x_+)} }{d \,   \mathcal{W}^{0,0}_{\eps, (x_{-},x_+)}} (u) \, = \, & \exp \Big(  -  \frac{1}{ \eps}  I_{x_-,x_+}(f)  +\frac{1}{ \eps}    \int_{x_-}^{x_+} \, \partial_x f(x) \, d u(x) \Big)\Big)  .\label{radnyk}
\end{align}
 \end{theorem}

Here, as in the case of Brownian motion, the stochastic integral term \newline $  \frac{1}{\eps}\int_{x_-}^{x_+} \partial_x f(x) \, d u(x)$ can be defined as the limit of Riemann sums in $L^2\big(   \mathcal{W}^{0,0}_{\eps, (x_{-},x_+)} \big)$. In particular, it is a linear mapping in $u$ defined for all $u$ in a measurable subspace of $C([x_-,x_+])$ of full measure (See e.g. \cite[Sec. 3]{Ha09}).

Note that  ~\eqref{radnyk} can formally be derived by expanding the square in the non-rigorous expression~\eqref{e:Fey}.

\begin{proof}[Proof of Lemma \ref{le:LDLoBou} ]
We can assume that $u_{\ast} \in H^1$, because otherwise the bound is trivial. As in the proof of the upper bound, ~\eqref{e:LDLB} only gets stronger when we take a smaller $\delta$. Therefore, it is sufficient to show~\eqref{e:LDLB} with $\delta$ replaced by
\begin{equation}\label{e:td2}
\tilde{\delta}:=  \gamma\bigg( \sup_{|v| \leq M + 1} |V'(v) |\ell_+  \bigg)^{-1} \wedge \delta \wedge 1 .
\end{equation}

We begin by stating the simplistic bound
\begin{align}
& \mathbb{E}_{(x_{-},x_{+})}^{\mathcal{W}_\eps,u_{-},u_{+}}  \Big(\mathbf{1}_{B(u_{\ast},\tilde{\delta})}(u) \,  \exp \Big(- \frac{1}{\eps} \int_{x_-}^{x_+} V( u ) \, dx \Big)\Big) \notag \\
   & \qquad \geq \,    \exp  \Big(-\frac{1}{\eps} \sup_{u \in B(u_{\ast},\tilde{\delta}) } \int_{x_-}^{x_+} V( u ) \, dx \Big) \,   \mathcal{W}^{u_{-},u_+}_{\eps, (x_{-},x_+)} \big( B(u_{\ast},\tilde{\delta}) \big)\label{e:LoBo1}.
\end{align}
Due to the assumption~\eqref{e:condu}  on $u_{\ast}$ and the definition~\eqref{e:td2} of $\tilde{\delta}$, we get that
\begin{align*}
\sup_{u \in B(u_{\ast},\tilde \delta) } \int_{x_-}^{x_+} V(u) \, dx& \leq  \int_{x_-}^{x_+} V( u_{\ast} ) \, dx +\sup_{|v| \leq  M+1 } |V'(v)| \,  \tilde{\delta} \,  \ell_+\\
&\leq \, \int_{x_-}^{x_+} V( u_{\ast} ) \, dx+ \gamma.
\end{align*}
It only remains to derive a lower bound on $  \mathcal{W}^{u_{-},u_+}_{\eps, (x_{-},x_+)}\big(B(u_{\ast},\tilde{\delta}) \big)$ in terms of the Gaussian energy. To this end, we again transform $u_{\ast}$ to an interval of length one and shift it in a way that it satisfies homogenous boundary conditions, as in the proof of Lemma~\ref{le:LDUpBou}. To be more precise, we assume that $u$ is distributed according to $  \mathcal{W}^{u_{-},u_+}_{\eps, (x_{-},x_+)}$ and apply the affine transformation $T$ defined in~\eqref{e:T}. Then $Tu= \hat{u}$ is distributed according to $ \mathcal{W}^{0,0}_{ \ell \eps, (0,1)}$. Therefore, we have to bound the probability $  \mathcal{W}^{0,0}_{\ell \eps, (0,1)}\big(B(\hat{ u}_{\ast}, \tilde{\delta})   \big)$, where $\hat{u}_{\ast}:= Tu_\ast$. This can be obtained using the Cameron-Martin Theorem \ref{thm:CamMart} with $f:=\hat{u}_{\ast}$. According to~\eqref{radnyk}, we have
\begin{align*}
 \mathcal{W}^{0,0}_{\ell \eps, (0,1)}\big( B( \hat{u}_{\ast}, \tilde{\delta})   \big) =  &\exp \Big(  -\frac{1}{\ell \eps}  I_{0,1}\big(   \hat{u}_{\ast} \big)   \Big) \,  \\
 &\mathbb{E}_{(0,1)}^{\mathcal{W}_{\ell\eps},0,0} \Big( \mathbf{1}_{B(0,\tilde{\delta})} (\hat{u}) \exp \Big( \frac{1}{\ell \eps}   \int_{0}^{1} \, \partial_x \hat{u}_{\ast}(x) \, d \hat{u}(x) \Big)\Big).
\end{align*}
Now we will use the trick of sneaking in a cosh function. To this end, we remark that the map $\hat{u} \mapsto  \int_{0}^{1} \, \partial_x \hat{u}_{\ast}(x) \, d \hat{u}(x) $ is linear in $\hat{u}$. Also, the measure $\mathcal{W}^{0,0}_{\ell\eps,(0,1)}$ is invariant under
the mapping $\hat{u} \mapsto -\hat{u}$ and this mapping leaves the ball $B(0,\tilde{\delta})$ invariant. Hence, the last expectation is equal to
\begin{equation*}
\mathbb{E}_{(0,1)}^{\mathcal{W}_{\ell\eps},0,0} \Big( \mathbf{1}_{B(0,\tilde{\delta})}(\hat{u})  \exp \Big(-\frac{1}{\ell \eps}    \int_{0}^{1} \, \partial_x \hat{u}_{\ast}(x) \, d \hat{u}(x) \Big)\Big) .
\end{equation*}
Therefore, we can write
\begin{align*}
\mathbb{E}_{(0,1)}^{\mathcal{W}_{\ell\eps},0,0} & \Big( \mathbf{1}_{B(0,\tilde{\delta})}(\hat{u})  \exp \Big(   \frac{1}{\ell \eps}  \int_{0}^{1} \, \partial_x \hat{u}_{\ast}(x) \, d \hat{u}(x) \Big)\Big) \\
& = \,\frac{1}{2} \,   \mathbb{E}_{(0,1)}^{\mathcal{W}_{\ell\eps},0,0} \Big( \mathbf{1}_{B(0,\tilde{\delta})} (\hat{u}) \, \Big[ \exp \Big(  \frac{1}{\ell \eps}  \int_{0}^{1} \, \partial_x \hat{u}_{\ast}(x) \, d \hat{u}(x) \Big)\\
& \qquad \qquad \qquad  \qquad + \exp \Big(-\frac{1}{\ell \eps}    \int_{0}^{1} \, \partial_x \hat{u}_{\ast}(x) \, d \hat{u}(x) \Big) \Big]\Big) \\
&= \,\mathbb{E}_{(0,1)}^{\mathcal{W}_{\ell\eps},0,0}  \Big( \mathbf{1}_{B(0,\tilde{\delta})} (\hat{u}) \cosh \Big(  \frac{1}{\ell \eps}    \int_{0}^{1} \, \partial_x \hat{u}_{\ast}(x) \, d \hat{u}(x) \Big)\Big)\\
& \geq \,  \mathcal{W}^{0,0}_{\ell \eps, (0,1)}\big(B(0,\tilde{\delta}) \big).
\end{align*}

We claim that there exists an $\eps_0>0$ such that for all $\ell \leq \ell_+$ and all $\eps \leq \eps_0$ this probability is larger than $\exp\big(-\eps^{-1} \gamma\big)$. Actually, ~\eqref{e:GLD} even implies that for any $\tilde{\gamma}>0$ there exists $\tilde{\eps}_0>0$ such that, for $\ell \eps \leq \tilde{\eps}_0$, we have the stronger bound
\begin{equation*}
 \mathcal{W}^{0,0}_{\ell \eps, (0,1)}\big(\complement B(0,\tilde{\delta}) \big) \leq \exp \Big(- \frac{1}{\ell \eps}  \Big( \inf_{\hat{u} \in \complement B(0,\tilde{\delta}) }  I_{0,1}(\hat{u}) - \tilde{\gamma} \Big) \Big).
\end{equation*}

Note that this $\eps_0$ also depends on $  \sup_{|v| \leq  M+1 } |V'(v)| $ as we have potentially decreased $\delta$ in the first step. Then in order to conclude, it is sufficient to observe that
\begin{align*}
\frac{1}{\ell \eps}  I_{0,1}\big(   \hat{u}_{\ast} \big)  \, = \, \frac{1}{\eps} \big(  I_{x_-,x_+}\big(   u_{\ast} \big) - I_{x_{\pm}}^{u_{\pm}} \big).
\end{align*}

\end{proof}

Now the proofs of Propositions \ref{pr:LD1} and \ref{pr:LD2} are straightforward. We begin with the upper bound, Proposition  \ref{pr:LD1}.

\begin{proof}[Proof of Proposition \ref{pr:LD1}]
We want to derive a bound on
\begin{equation} \label{e:fraction}
 \mu^{u_-,u_+}_{\eps,(x_-,x_+)}\big( \Aa \big) \, = \,  \frac{  \mathbb{E}_{(x_{-},x_{+})}^{\mathcal{W}_\eps,u_{-},u_{+}} \Big( \mathbf{1}_{\Aa} (u)\,   \exp \Big(- \frac{1}{\eps} \int_{x_-}^{x_+} V(u) \, dx \Big)    \Big)}{\mathbb{E}_{(x_{-},x_{+})}^{\mathcal{W}_\eps,u_{-},u_{+}} \Big(   \exp \Big(- \frac{1}{\eps} \int_{x_-}^{x_+} V(u) \, dx \Big)    \Big)}.
\end{equation}
The assumptions on $\Aa$ in Proposition \ref{pr:LD1} are identical to those in Lemma \ref{le:LDUpBou}, so we can conclude from~\eqref{e:LDUB} that
\begin{align*}
  \mathbb{E}_{(x_{-},x_{+})}^{\mathcal{W}_\eps,u_{-},u_{+}}& \Big( \mathbf{1}_{\Aa} (u)\,   \exp \Big(- \frac{1}{\eps} \int_{x_-}^{x_+} V( u ) \, dx \Big)    \Big)\\
  & \leq \,    \exp  \Big(-\frac{1}{\eps} \big(  \inf_{u \in B(\Aa,\delta)} E(u) -\Ixu - \gamma   \big)  \Big)
\end{align*}
for $\eps \leq \eps_0$. Also this $\eps_0$ depends on $M,R,, \ell_+,\delta,$ and $\gamma$ but not on the particular choice of $x_{\pm}, u_{\pm}$. It only depends on $\Aa$ through the condition~\eqref{e:condA} and on $V$ through the local Lipschitz constant.

To get a lower bound on the denominator in~\eqref{e:fraction}, we observe that for every set of boundary conditions $u_{\pm}$, there exists at least one minimizer $u_\ast$ of $E$ given these boundary conditions. Furthermore,  this minimizer attains only values in $[-M,M]$. This is clear because replacing $u_\ast$ by $u_\ast\wedge M \vee (-M)$ only decreases the energy. Therefore, for any $\delta >0$, we get from~\eqref{e:LDLB}  that
\begin{align*}
\mathbb{E}_{(x_{-},x_{+})}^{\mathcal{W}_\eps,u_{-},u_{+}}&  \Big(   \exp \Big(- \frac{1}{\eps} \int_{x_-}^{x_+} V( u ) \, dx \Big)    \Big) \\
& \geq \, \mathbb{E}_{(x_{-},x_{+})}^{\mathcal{W}_\eps,u_{-},u_{+}}  \Big( \mathbf{1}_{B(u_\ast, \delta)}(u) \,   \exp \Big(- \frac{1}{\eps} \int_{x_-}^{x_+} V( u ) \, dx \Big)    \Big) \\
& \geq \,     \exp  \Big(-\frac{1}{\eps} \big(  E(u_\ast) -I_{x_{\pm}}^{u_{\pm}} + \gamma   \big)  \Big)
\end{align*}
for $\eps \leq \eps_0$, where $\eps_0$ satisfies the same uniformity assumptions as above. This finishes the argument.
\end{proof}

The proof of the lower bound is similar.

\begin{proof}[Proof of Proposition \ref{pr:LD2}]
To derive a lower bound on $  \mu^{u_-,u_+}_{\eps,(x_-,x_+)}\big(\Aa\big)$ for a given $\gamma$ we choose $u_\gamma$ as in~\eqref{e:coneta}. Then we can write using~\eqref{e:LDLB}
\begin{align*}
\mathbb{E}_{(x_{-},x_{+})}^{\mathcal{W}_\eps,u_{-},u_{+}}& \Big( \mathbf{1}_{B(\Aa,\delta)} (u)\,   \exp \Big(- \frac{1}{\eps} \int_{x_-}^{x_+} V(u) \, dx \Big)    \Big)\\
&\geq \mathbb{E}_{(x_{-},x_{+})}^{\mathcal{W}_\eps,u_{-},u_{+}} \Big( \mathbf{1}_{B(u_\gamma,\delta)} (u)\,   \exp \Big(- \frac{1}{\eps} \int_{x_-}^{x_+} V(u) \, dx \Big)    \Big)\\
& \geq \exp \Big( - \frac{1}{\eps} \big(  \inf_{u \in \Aa} E(u) - I_{x_{\pm}}^{u_{\pm}}+2 \gamma \big) \Big),
\end{align*}
for $\eps \leq \eps_0$ where $\eps_0$ can again be chosen uniformly.

To derive a uniform upper bound on the normalization constant we only need to observe that for any $M<\infty$ there exists an $R < \infty$ such that for all $u_\pm\in[-M,M]$, we have
\begin{equation*}
\inf_{u \in\mathcal{A}^{\rm bc}} E(u) \leq R.
\end{equation*}
Then \eqref{e:LDUB} implies that there exists $\eps_0>0$ such that uniformly for $\eps \leq \eps_0$
\begin{align*}
\mathbb{E}_{(x_{-},x_{+})}^{\mathcal{W}_\eps,u_{-},u_{+}}&  \Big(   \exp \Big(- \frac{1}{\eps} \int_{x_-}^{x_+} V(u) \, dx \Big)    \Big) \\
& \leq \,     \exp  \Big(-\frac{1}{\eps} \big(  \inf_{u\in\mathcal{A}^{\rm bc}} E(u) -I_{x_{\pm}}^{u_{\pm}} + \gamma   \big)  \Big).
\end{align*}
This establishes \eqref{e:LDLB1}.

\end{proof}

\subsection{Proof of the one-point distribution lemma}\label{ss:onept}
\begin{proof}[Proof of Lemma~\ref{le:onept}]

First we remark that, heuristically, the ``most difficult'' point to consider is $x_0=0$.  We present the following proof for precisely this case.  The same proof carries over for any point $x_0$ (with only trivial modifications), but we present it for $x_0=0$ since it simplifies the notation slightly and makes the main ideas stand out.

Also notice that by the symmetry of the potential (cf. Assumption~\ref{ass:V}) and the representation~\eqref{m30.1}, it suffices to prove
\begin{align*}
\mu^{-1,1}_{\eps,(-\Le,\Le)}\Big(u(x_0)\geq M\Big)\,\leq \,\exp\left(-\frac{M}{\eps\,C_2}\right).
\end{align*}
In fact, it will be convenient to establish the estimate in the form
\begin{align}
\mu^{-1,1}_{\eps,(-\Le,\Le)}\Big(u(0)\geq 4M\Big)\,\leq \,\exp\left(-\frac{M}{\eps\,\tilde{C}_2}\right),\label{onept2}
\end{align}
which is of course equivalent for $C_2:=4\tilde{C}_2$.
Thus consider the set of functions
\begin{align}
\mathcal{A}:=\{u\in C([-\Le,\Le])\colon u(-\Le)=-1, u(\Le)=1, \,\text{and}\, u(0)\geq 4M\}.
\end{align}
Define $x^{3M}_\pm$ as follows:
\begin{align*}
x^{3M}_-:=\sup\{x\leq 0 \colon u(x)\leq 3M\}\quad \text{and}\quad x^{3M}_+:=\inf\{x\geq 0 \colon u(x)\leq 3M\}.
\end{align*}
Notice that we may assume without loss of generality that $M\geq 1$, and hence, because of the boundary conditions $u(-\Le)=-1$ and $u(\Le)=1$,  the points $x^{3M}_- < 0 < x^{3M}_+ $ are well-defined for every $u \in \Aa$.  The set $\Aa$ can then be divided into the following two sets:
\begin{align*}
&\mathcal{A}_1:=\{u\in\mathcal{A}\colon\max\{|x^{3M}_-|,x^{3M}_+\}>1\},\\
&\mathcal{A}_2:=\{u\in\mathcal{A}\colon\max\{|x^{3M}_-|,x^{3M}_+\}\leq 1\}.
\end{align*}
To bound the probability of $\Aa_1$, we will use bounds on the potential and a reflection argument.  For $\Aa_2$ we will use a rescaling argument and the large deviation bound \eqref{e:LDUB1}. The two cases are illustrated in Figure~\ref{fig:6}.
 \begin{figure}
\begin{subfigure}[b]{0.46\textwidth}
                \centering
\begin{tikzpicture}[xscale=0.011, yscale=1.38, >=stealth]
%
%
\draw[black] plot file{./Graphs/P6G1.txt};
\draw[black,style=dashed] plot file{./Graphs/P6G2.txt};
%
%
%
\draw[black,  style=dotted] (0,0) -- (500,0) node[anchor=north east]{\small $x$};
\draw[black] (200,0.1) -- (200,-0.1) node[fill=white,anchor= north]{$-1$} ;
\draw[black] (300,0.1) -- (300,-0.1) node[fill=white,anchor= north]{$1$} ;%
%
%
%
%
%
\draw[black,  style=dotted]  (1,4)  -- (500,4)  node[anchor=east,fill=white]{$4M$} ;
\draw[black,  style=dotted] (1,3)   -- (500,3) node[anchor = east,fill=white] {$3M$}  ;
\draw[black,  style=dotted] (127,0)  -- (127,3)  node[fill=white,anchor= south east]{$x^{3M}_-$}   ;
\draw[black,  style=dotted] (333,0) -- (333,3) node[fill=white,anchor= south west]{$x^{3M}_+$}   ;
\draw[black,  style=dotted]  (1,2)  -- (500,2)  node[anchor=east,fill=white]{$2M$} ;
\draw[black,  style=dotted] (113,0)  -- (113,2)  node[fill=white,anchor= south east]{$x^{2M}_-$}   ;
\draw[black,  style=dotted] (355,0) -- (355,2) node[anchor= south west]{$x^{2M}_+$}   ;
\end{tikzpicture}
\caption{A path in $\Aa_1$.}
\end{subfigure}
\begin{subfigure}[b]{0.46 \textwidth}
\centering
\begin{tikzpicture}[xscale=0.011, yscale=1.38, >=stealth]
%
%
%
%
%
\draw[black,  style=dotted] (0,0) -- (500,0) node[anchor=north east]{\small $x$};
\draw[black] (200,0.1) -- (200,-0.1) node[fill=white,anchor= north]{$-1$} ;
\draw[black] (300,0.1) -- (300,-0.1) node[fill=white,anchor= north]{$1$} ;
%
%
%
%
%
\draw[black] plot file{./Graphs/P6G3.txt};
%
%
%
\draw[black,  style=dotted] (1,4) -- (500,4) node[anchor=east,fill=white]{$4M$}  ;
\draw[black,  style=dotted] (1,3) -- (500,3) node[anchor=east,fill=white] {$3M$} ;
\draw[black,  style=dotted] (228,0)  -- (228,3)  node[fill=white,anchor= south east]{$x^{3M}_-$}   ;
\draw[black,  style=dotted] (265,0) -- (265,3) node[fill=white,anchor= south west]{$x^{3M}_+$}   ;
\end{tikzpicture}
\caption{A path in $\Aa_2$.}
\end{subfigure}
\caption{The two different cases. To show that $\Aa_1$ has small probability, we reflect between the $x_{\pm}^{2M}$. This decreases the \emph{potential energy}. The probability of $\Aa_2$ can be bounded using a large deviation argument.  }
\label{fig:6}
\end{figure}
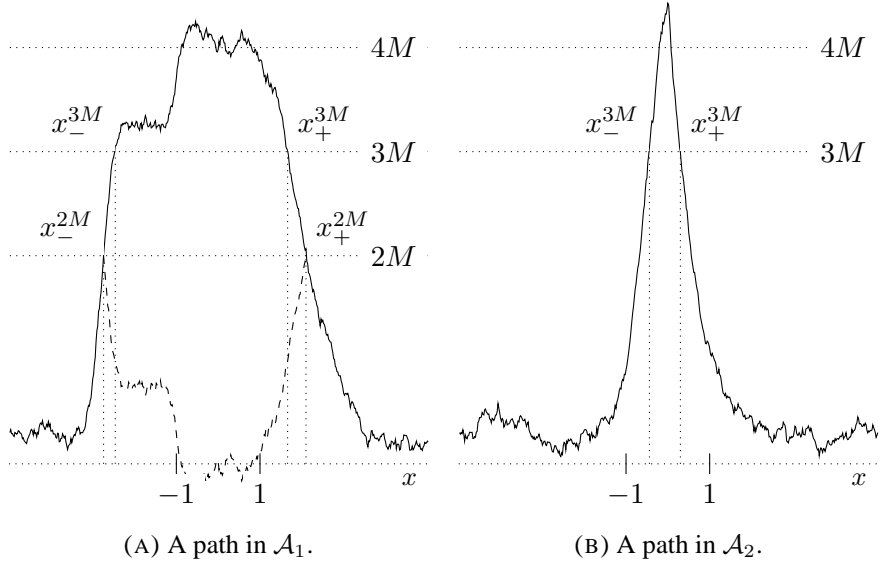

\underline{Step 1}.  We treat $\Aa_1$ first. For $u \in \Aa_1$ we have
\begin{align}
x^{3M}_+-x^{3M}_-\geq 1.\label{m.set}
\end{align}
The idea is to introduce a reflection over the line $u=2M$ that preserves the Gaussian measure, and use the decrease of the energy~\eqref{energy} under this reflection.

We begin by collecting some facts about the potential $V$.  To begin with, according to the growth estimate in~\eqref{e:assumptions}, $V$ grows superlinearly at infinity.  Hence, we may choose $C_3$ sufficiently large so that the following two properties are satisfied.  On the one hand, $V$ grows at least linearly on $[C_3,\infty)$, i.e., there exists $C_4<\infty$ such that for $u_1\geq u_2\geq C_3$, there holds
\begin{align}
V(u_1)-V(u_2)\,\geq\, 1/C_4\,(u_1-u_2).\label{m.1}
\end{align}
On the other hand, $V(C_3)\geq V(0)$, so that in particular
\begin{align}
V(C_3)=\sup_{u\in[0,C_3]}V(u).\label{m.2}
\end{align}
We will use the fact that~\eqref{m.1} and~\eqref{m.2} together imply that as long as $u_1\geq C_3$, then
\begin{align}
u_1\geq |u_2|\Rightarrow V(u_1)\geq V(|u_2|).\label{m.2p}
\end{align}

Now we are ready to reflect.  Define $x^{2M}_\pm$ analogously to $x^{3M}_\pm$ (noting as above that they are well-defined for paths in the set of interest). Consider the reflection operator $R^{x^{2M}_+}_{x^{2M}_-}$ defined as
\begin{equation}\label{e:Refl1}
R^{x^{2M}_+}_{x^{2M}_-} u (x) :=
\begin{cases}
u(x) \qquad & \text{if $ x \notin (x^{2M}_-,x^{2M}_+ )  $}\\
4M - u(x) \qquad & \text{if $ x \in (x^{2M}_-,x^{2M}_+ )  $}
\end{cases},
\end{equation}
which for the purposes of this lemma we will abbreviate with $\mathsf{R}$. In order to have $\mathsf{R}$ well-defined for all continuous paths $u$, we define it to be the identity for those paths $u$  that never exceed the level $2M$.

Notice that $x_{-}^{2M}$ is a \emph{right} but not a \emph{left}  stopping point, and similarly $x_{+}^{2M}$ is a \emph{left} but not a \emph{right}  stopping point. In particular, the strong Markov property \eqref{e:strMarkovW} does not directly imply that $\mathsf{R}$ leaves $\mathcal{W}^{-1,1}_{\eps, (-\Le, \Le)}$ invariant. Indeed, it is not true that under $  \mathcal{W}^{-1,1}_{\eps, (-\Le,\Le)}$ the conditional distribution of $u(x)$ for $x \in [x^{2M}_-,x^{2M}_+]$, given the path outside of this interval, is a Brownian bridge.

Still, it is true that the reflection operator $\mathsf{R}$ preserves $  \mathcal{W}^{-1,1}_{\eps, (-\Le, \Le)}$. To see this, introduce auxiliary stopping points
\begin{align*}
\chi^{2M}_-:=\inf\{x\geq -\Le \colon u(x) = 2M\}\quad \text{and}\quad \chi^{2M}_+:=\sup\{x \leq \Le  \colon u(x)= 2M\}.
\end{align*}
As above in \eqref{e:defstp}, we use the convention that $\chi_{\pm}^{2M} = \mp \Le$ if these sets are empty.

On $\Aa$, these points are well-defined and we automatically have $[x^{2M}_-,x^{2M}_+ ]  \subseteq [\chi^{2M}_-,\chi^{2M}_+ ] $. The points $\chi^{2M}_{\pm}$ are left and right stopping points. Therefore, \eqref{e:strMarkovW} implies that the reflection operators $ R^{\chi^{2M}_+}_{\chi^{2M}_-} $,  $ R^{\chi^{2M}_+}_{x^{2M}_+} $, and  $ R^{x^{2M}_-}_{\chi^{2M}_-} $ (defined in the same  way as $\mathsf{R}$ ) preserve $  \mathcal{W}^{-1,1}_{\eps, (-\Le, \Le)}$. Observing that
\begin{equation*}
\mathsf{R} =  R^{\chi^{2M}_+}_{\chi^{2M}_-} \circ R^{\chi^{2M}_+}_{x^{2M}_+} \circ  R^{x^{2M}_-}_{\chi^{2M}_-} ,
\end{equation*}
we conclude that $ \mathsf{R}$ also preserves  $  \mathcal{W}^{-1,1}_{\eps, (-\Le, \Le)}$.

We now develop a quantitative, pointwise estimate of the effect of $\mathsf{R}$ on the ``bulk energy'' $V(u)$.  By the definition of $x^{2M}_\pm$, we have that $u(x)\geq 2M$ for all $x\in[x^{2M}_-,x^{2M}_+]$, the set where $\mathsf{R}$ acts.  Hence, it suffices to consider the effect of $\mathsf{R}$ when $u(x)\geq 3M$ and when $u(x)\in[2M,3M)$.  We will first establish that on the set
$$\{x\in[x^{2M}_-,x^{2M}_+]\colon u(x)\geq 3M\},$$
$\mathsf{R}$ decreases the bulk energy significantly.  Indeed, on this set, $|\mathsf{R}u|\leq u-2M$ and $u-2M\geq M\geq C_3$, so that
\begin{equation}
V(|\mathsf{R}u|)\overset{\eqref{m.2p}}\leq V(u-2M),\label{m.3}
\end{equation}
which together with~\eqref{m.1} implies that for $u(x)\geq 3M$,
\begin{eqnarray}
V(u(x))-V(\mathsf{R}u(x))&\overset{\eqref{e:assumptions}}=&V(u(x))-V(|\mathsf{R}u(x)|)\nonumber\\
&\overset{\eqref{m.3}}\geq&V(u(x))-V(u(x)-2M)\nonumber\\
&\overset{\eqref{m.1}}{\geq}&2\,M/C_4,\label{m.4}
\end{eqnarray}
which holds in particular on all of $[x^{3M}_-,x^{3M}_+]$. On the other hand, if instead $u(x)\in[2M,3M)$, then the bulk energy still decreases under $\mathsf{R}$.  Indeed, we have for $u(x)\in[2M,3M)$ that $u(x)\geq \mathsf{R}u(x)\geq M$, so that by~\eqref{m.2p} we know
\begin{align}
V(u(x))-V(\mathsf{R}u(x))\,\geq\, 0.\label{m.4p}
\end{align}
Combining~\eqref{m.4} and~\eqref{m.4p} implies that for all $u\in\mathcal{A}_1$, we have
\begin{eqnarray}
\lefteqn{\int_{(-\Le,\Le)}\big(V(u)-V(\mathsf{R}u)\big)\,dx}\nonumber\\
&=&\int_{(x^{2M}_-,x^{2M}_+)}\Big( V(u)-V(\mathsf{R}u)\Big)\,dx\nonumber\\
&\geq& 2M(x^{3M}_+-x^{3M}_-)/C_4\overset{\eqref{m.set}}\geq 2M/C_4.\label{m.5}
\end{eqnarray}
We are now ready to estimate the probability of $\mathcal{A}_1$.
Indeed, we have that

\begin{eqnarray*}
1&\geq& \mu^{-1,1}_{\eps,(-\Le,\Le)}\big(\mathsf{R}\mathcal{A}_1\big)\\
&\overset{\eqref{m30.1}}=&  \frac{1}{   \mathcal{Z}}  \, \mathbb{E}^{\mathcal{W}_\eps,-1,1}_{( -\Le,\Le)}\left[\mathbf{1}_{ \mathsf{R}\mathcal{A}_1}(u) \,  \exp\Big(-\frac{1}{\eps}\int V(u)\,dx\Big)\right] \\
&\overset{\text{inv. of $\mathsf{R}$} }=&\frac{1}{  \mathcal{Z}} \,  \mathbb{E}^{\mathcal{W}_\eps,-1,1}_{( -\Le,\Le)}\left[\mathbf{1}_{ \mathcal{A}_1}(u)\,   \exp\Big(-\frac{1}{\eps}\int V(\mathsf{R}u)\,dx\Big)\right]\\
&=& \frac{1}{  \mathcal{Z}} \,  \mathbb{E}^{\mathcal{W}_\eps,-1,1}_{( -\Le,\Le)}\bigg[\mathbf{1}_{\mathcal{A}_1}(u) \exp\Big(-\frac{1}{\eps}\int \big(V(\mathsf{R}u)-V(u)\big)\,dx\\
&& \qquad \qquad \qquad \qquad  \qquad -\frac{1}{\eps}\int V(u)\,dx\Big)\bigg] \\
&\overset{\eqref{m.5}}\geq&\exp\left(\frac{2M}{ C_4 \, \eps }\right)   \frac{1}{  \mathcal{Z}} \,  \mathbb{E}^{\mathcal{W}_\eps,-1,1}_{( -\Le,\Le)}\left[\mathbf{1}_{ \mathcal{A}_1}(u) \, \exp\Big(-\frac{1}{\eps}\int V(u)\,dx\Big)\right]\\
&=&\exp\left(\frac{2M}{  C_4 \, \eps}\right)\mu^{-1,1}_{\eps,(-\Le,\Le)}\big(\mathcal{A}_1\big),
\end{eqnarray*}
where $\mathcal{Z}= \mathcal{Z}^{-1,1}_{\eps, (-\Le,\Le)}$ is the normalization constant for $\mu^{-1,1}_{\eps,(-\Le,\Le)}$ and all of the integrals are over $[-\Le,\Le]$.  Moving the exponential to the other side of the inequality, we get that
\begin{align}
\mu^{-1,1}_{\eps,(-\Le,\Le)}\big(\mathcal{A}_1\big)\leq\exp\left(-\frac{2M}{C_4 \, \eps}\right), \label{case1}
\end{align}
which gives~\eqref{onept2} for $\Aa_1$ with $\tilde{C}_2=C_4/2$ .

\medskip

\underline{Step 2}. Now consider the set $\Aa_2$.  Here we will use a rescaling argument and the large deviation bound \eqref{e:LDUB1}. For $u \in \Aa_2$,  we can define
\begin{align*}
\chi_-:=\inf\{x\in[-1,0]\colon u(x)=3M\},\quad \chi_+:=\sup\{x\in[0,1]\colon u(x)=3M\},
\end{align*}
with the understanding that $\chi_{\pm} = 0$ if these sets are empty.  These random variables are left and right stopping points. Hence, the strong Markov property \eqref{e:strMarkovmu} implies that
\begin{align}
\lefteqn{\mu^{-1,1}_{\eps,(-\Le,\Le)}\big(\Aa_2\big)}\nonumber\\
&=\mu^{-1,1}_{\eps,(-\Le,\Le)}\Big(u(0)\geq 4M \;\text{and}\; \chi_\pm \neq 0\Big)\nonumber\\
&=  \mathbb{E}_{(-\Le,\Le)}^{\mu_\eps,-1,1} \Big( \mu_{\eps,(\chi_{-},\chi_{+})}^{3M,3M} \big( u(0) \geq 4M \big)  \,\mathbf{1}_{ \{ \chi_{\pm} \neq 0 \}}  \Big) .      \label{mm.sq}
\end{align}
Therefore, if we can show that
\begin{align}
\mu^{3M,3M}_{\eps,(x_-,x_+)}\Big(u(0)\geq 4M\Big)\leq\exp\left(-\frac{M}{\tilde{C}_2\eps}\right)\label{mm.tr}
\end{align}
for all $\eps$ sufficiently small (uniformly for $-x_-,x_+\in (0,1] $), then the combination of~\eqref{mm.sq} and~\eqref{mm.tr} concludes the proof of~\eqref{onept2}.  We can see~\eqref{mm.tr} by rescaling.  Indeed, if we transform $(x_-,x_+)$ into $[-1,1]$ by applying the affine change of variables $x\rightarrow  \frac{\Delta x}{2}  \,x+\frac{x_-+x_+}{2}$ where $\Delta x:=  x_+-x_-$,  we see that
\begin{align}
&\mu^{3M,3M}_{\eps,(x_-,x_+)}  \Big(u(0)\geq 4M\Big)\notag\\
  =\, &  \frac{1}{ \mathcal{Z} }  \mathbb{E}_{(x_{-},x_{+})}^{\mathcal{W}_\eps,3M,3M} \Big( \mathbf{1}_{\{u(0)>4M \}} \exp\Big(-\frac{1}{\eps}\int_{x_-}^{x_+} V\big( u(x) \big) \, dx \Big)  \Big)\notag\\
 =\, &  \frac{1}{  \mathcal{Z}} \,   \mathbb{E}_{(-1,1)}^{\mathcal{W}_{\tilde{\eps}},3M,3M} \Big( \mathbf{1}_{\{\hat{u}(\frac{x_-+x_+}{2})>4M \}} \exp\Big(- \frac{(\Delta x)^2}{ 4\tilde{\eps} }   \int_{-1}^1 V\big( \hat{u}(x) \big) \, dx \Big)  \Big) ,\label{vor}
\end{align}
where $\mathcal{Z}=\mathcal{Z}^{3M,3M}_{\eps, (x_-,x_+)}$ is the normalization constant for $\mu_{\eps,(x_-,x_+)}^{3M,3M}$ and  $\tilde{\eps} :=   \frac{1}{2} \eps \Delta x$. Now, we observe that the  family of potentials
\begin{equation*}
  \Big\{ \frac{(\Delta x)^2}{4} V \colon  0 < \Delta x \leq 2 \Big\}
\end{equation*}
is locally uniformly Lipschitz. In particular, applying Proposition \ref{pr:LD1} for $\gamma$ and $\delta$ fixed to say $\gamma=\delta=1$, there exists $\eps_0>0$ such  that, for $\tilde{\eps} \leq \eps_0$ and uniformly in $x_{\pm}$, we have
\begin{equation*}
\mu^{3M,3M}_{\tilde \eps,(-1,1)}(\mathcal{B}_+^{\rm bc})  \,  \leq \,  \exp\Big( -\frac{1}{\tilde{\eps}} \big( \inf_{\hat{u}\in B(\mathcal{B}_+^{\rm bc},1)}E_{\Delta x}(\hat{u})-\inf_{\hat{u}\in \mathcal{B}^{\rm bc}}E_{\Delta x}(\hat{u}) - 1 \big) \Big).
\end{equation*}
Note that the choice of $\eps_0$ depends on $M$.

Here we use the notation
\begin{align*}
E_{\Delta x }(\hat{u}) \, := \int_{-1}^{1} \left(\frac{1}{2}(\partial_x \hat{u})^2+\frac{(\Delta x)^2}{4}V(\hat{u})\right)\,dx
\end{align*}
and
\begin{align*}
\mathcal{B}^{\rm bc}&:=\{\hat{u}\in C([-1,1])\colon \hat{u}(\pm 1)=3M\},\\
\mathcal{B}_+^{\rm bc}&:=\big\{\hat{u}\in C([-1,1])\colon \hat{u}(\pm 1)=3M,\;\hat{u}\big((x_- + x_+)/2 \big)\geq 4M\big\}.
\end{align*}

Hence, as $\tilde{\eps} \leq  \eps $, to establish~\eqref{mm.tr} it will be sufficient for us to show
\begin{align*}
\inf_{B(\mathcal{B}_+^{\rm bc},1)}E_{\Delta x}-\inf_{\mathcal{B}^{\rm bc}}E_{\Delta x}\,\geq\, \frac{M}{\tilde{C}_2},
\end{align*}
and we will in fact establish the stronger bound
\begin{equation*}
\inf_{B(\mathcal{B}_+^{\rm bc},1)}E_{\Delta x}-\inf_{\mathcal{B}^{\rm bc}}E_{\Delta x} \,\geq (M- 1)^2\overset{M\geq 4}{\geq}\frac{M^2}{2}.
\end{equation*}
We will establish the first inequality  by way of a variational argument.  Notice that we may assume that the infima are achieved (if not, a simple approximation argument suffices), and so let
$$\hat{u}_1:=\underset{B(\mathcal{B}_+^{\rm bc},1)}{\rm argmin}\,E_{\Delta x},\qquad \hat{u}_2:=\underset{\mathcal{B}^{\rm bc}}{\rm argmin}\,E_{\Delta x} .$$
Observe that automatically $\hat{u}_1\big( (x_-+x_+)/2\big) \geq 4M- 1$.

We define the auxiliary function $\hat{u}_3 := \min\{ \hat{u}_1, 3M\} $. Notice that according to the growth assumption~\eqref{e:assumptions} (or see~\eqref{m.1}):
\begin{align}
V(\hat{u}_1(x))\geq V(3M)\qquad \text{on } \{ \hat{u}_1 \geq 3M \} .\label{mm.gtr}
\end{align}
On the other hand, since $\hat{u}_3\in\mathcal{B}^{\rm bc}$ and as $\hat{u}_2$ is the minimizer over $\mathcal{B}^{\rm bc}$, we have
\begin{eqnarray*}
E_{\Delta x}(\hat{u}_1)\!\!\!\!&-&\!\!\!\!E_{\Delta x}(\hat{u}_2)\\
&\geq& E_{\Delta x}(\hat{u}_1)-E_{\Delta x}(\hat{u}_3)\\
&=&  \int_{ \{ u \geq 3M\}}\Big( (\partial_x \hat{u}_1)^2+\frac{(\Delta x)^2}{4} \big(V(\hat{u}_1)-V(3M)\big)\Big)\,dx\\
&\overset{\eqref{mm.gtr}}\geq&   \int_{-1}^1 (\partial_x\max\{ \hat{u}_1 -3M, 0\})^2\,dx\\
&\geq&  \big( \sup_{x \in [-1,1]} \max\{ \hat{u}_1 -3M , 0 \} \big)^2\\
&\geq& (M- 1)^2.
\end{eqnarray*}
This concludes the proof of~\eqref{onept2} for $\mathcal{A}_2$ and establishes the lemma.
\end{proof}

\subsection{Proofs of lemmas from the lower bound of Theorem~\ref{t:layers}.}\label{ss:negneg}
\begin{proof}[Proof of Lemma~\ref{l:negneg}]
Since the proof is similar to (and simpler than) the proof of the upper bound in Theorem~\ref{t:layers}, we will be somewhat brief. Our goal is to bound above by $2/3$ the complementary event, namely that $u(x)>0$ for some $x\in[-\Le,-2\ell]$ or that $|u(x_k)|>M$ for some $k$ in the index set. As in the proof of the upper bound, the probability that $|u(x_k)|>M$ can be shown to be exponentially small in $M/\eps$, cf.~\eqref{sq}. It remains to bound above the probability that $u(x)>0$ for some $x\in[-\Le,-2\ell]$ and $|u(x_k)|\leq M$ for all $k$.

Now fix $\delta>0$ sufficiently small so that the estimates from the upper bound of Theorem~\ref{t:layers} apply. The set $\{u\in\complement\mathcal{A}_1\colon u(x) >0\text{ for some }x\in[-\Le,-2\ell] \}$ is contained within the union of:
\begin{enumerate}
 \item functions with more than one $\delta^-$ layer (exponentially unlikely by the upper bound of Theorem~\ref{t:layers}) ,

  \item  functions with a $\delta^-$ layer longer than $2\ell$ (exponentially unlikely for $\delta^2\ell$ large, according to the calculation in Step 3 of the proof of the upper bound, cf.~\eqref{e:A2bou}),

 \item functions with one and only one $\delta^-$ layer, which is at most length $2\ell$ and is contained in $[-\Le,0]$,

  \item functions with one and only one $\delta^-$ layer,  which is at most length $2\ell$ and is contained in  $[-2\ell,\Le]$, and such that $u(x)>0$ for some $x\in[-\Le,-2\ell]$.
\end{enumerate}
By symmetry properties of the measure, i.e. the symmetry with respect to point reflection of the graph at $x=0$ and $u=0$, the probability of a $\delta^-$ layer contained in $[-\Le,0]$ is equal to the probability of a $\delta^-$ layer contained in $[0,\Le]$, hence neither can be more than $1/2$. Therefore, the probability of the event described in point (3) is less than or equal to $1/2$.

By the calculations referred to above, the sum of the probabilities of the sets described in (1)-(3) is bounded by $1/2$ plus exponentially small terms, so we are finished if we can show that the probability of the set described in (4) is also exponentially small, namely, the probability that: $u(x)>0$ for some $x\in[-\Le,-2\ell]$, $|u(x_k)|\leq M$ for all $k$, and there is one and only one $\delta^-$ layer, which is at most $2\ell$ and is contained in $[-2\ell,\Le]$. Note that the latter implies that $u \leq 1-\delta$ on $[-\Le, -2\ell]$.

This bound is easy to obtain by breaking into subintervals (using conditioning) and using the  large deviation estimate~\eqref{e:LDUB1}. Indeed, we reduce to probabilities of the form
\begin{align*}
 \mu_{\eps,(x_{k-2},x_{k+2})}^{u_{k-2},u_{k+2}}\Big(u\leq 1-\delta\text{ and }u(\hat{x})=0\text{ for some }\hat{x}\in(x_{k-1},x_{k+1})\Big),
\end{align*}
where $u_{k-2}$ and $u_{k+2}$ are arbitrary boundary values in $[-M,1-\delta]$ and $k\in\{-(N_\eps-2),-(N_\eps-3),\ldots,-3\}$. (We also need to consider the boundary interval, where $\hat{x}\in (x_{-N_\eps},x_{-(N_\eps-2)})$. As usual, this is no more difficult than the bound for the interior intervals.) After applying Proposition~\ref{pr:LD1} (with $\tilde{\delta}=\delta/2$), it remains only to introduce an energetic bound. The  bound from Lemma~\ref{l:last} below suffices.

Before stating the energy lemma, we explain the idea in words: If we take a $\delta/2$ ball around the set of interest, then on $[x_{k-1},x_{k+1}]$, there is a point $x_0$ such that $u(x_0)\geq -\delta/2$. For $\ell$ large, the energy minimizer needs to come very close to $\pm 1$ someplace in $[x_{k-2},x_{k-1}]$ and $[x_{k+1},x_{k+2}]$, (say within $\delta/4$), and since it cannot come this close to $+1$, it is forced into a small neighborhood of $-1$. Consequently, the large excursion from $-1$ at $x_0$ costs almost $c_0$ energy. We give the precise statement below and prove the lemma at the end of the subsection.
\begin{lemma}\label{l:last}
There exists $C<\infty$ with the following property. For any $M$ large enough and $\delta>0$ small enough, consider the boundary conditions $u_\pm\in[-M,1-\delta]$ and define the sets
\begin{align*}
\mathcal{A}^{\rm bc}&:=\{u\in C([-2\ell, 2\ell]) \colon u(-2\ell)=u_{-}  \text{ and  }u(2 \ell) = u_+\},\\
\mathcal{A}_0^{\rm bc}&:=\{u\in \mathcal{A}^{\rm bc}  \colon   u(x) \leq 1 - \delta/2 \text{ for all  $ x \in  [-2\ell, 2 \ell]$ and    } \\
& \qquad \qquad \text{ there is an $x_0 \in [-\ell, \ell]$  such that $u(x_0)\geq -\delta/2 $} \}.
\end{align*}
Then there exists $\ell_0=\ell_0(M,\delta)$ such that for $\ell\geq \ell_0$ there holds
\begin{equation*}
\inf_{u \in \mathcal{A}_0^{\rm bc} } E_{(-2\ell, 2 \ell)}(u) - \inf_{u \in \mathcal{A}^{\rm bc} }  E_{(-2\ell, 2 \ell)}(u) \geq  c_0-C\delta .
\end{equation*}
\end{lemma}

Proposition~\ref{pr:LD1} and Lemma~\ref{l:last} together give
\begin{align*}
 \mu_{\eps,(x_{k-2},x_{k+2})}^{u_{k-2},u_{k+2}}\Big(&u\leq 1-\delta \text{ and }u(x_0)>0\text{ for some }x_0\in(x_{k-1},x_{k+1})\Big)\\
 &\leq \exp\left(-\frac{c_0 - C\delta-\gamma}{\eps}\right).
\end{align*}
Finally, we now choose $\gamma$ and  $\delta$ sufficiently small and sum over the order $N_\eps\sim \Le$ intervals. Bearing in mind the bound~\eqref{lbd} on $\Le$, we observe that there is also an exponentially small probability of the final set that we have studied.
\end{proof}

\begin{proof}[Proof of Lemma~\ref{l:last}]
We will be brief, since the proof is similar to the proof of Lemma~\ref{l:cl}. 

First of all, fix $M$ large and $\delta$ small. The infimum of the energy over $\mathcal{A}^{\rm bc}$ is less than or equal to the minimum of the energy over functions with $u(\pm 2\ell)=u_{\pm}$ and $u(0)=-1$. By a standard construction, we have
\begin{align*}
\inf_{\Aa^{\rm bc}} E_{(-2\ell,2\ell)}(u)\leq \varphi_{-1}(u_-)+\varphi_{-1}(u_+)+o(1)_{\ell\uparrow\infty}.
\end{align*}
In particular, for $\ell_0$ large enough and $\ell\geq\ell_0$, one has
\begin{align}
\inf_{\Aa^{\rm bc}} E_{(-2\ell,2\ell)}(u)\leq \varphi_{-1}(u_-)+\varphi_{-1}(u_+)+\delta.\label{upinf}
\end{align}
On the other hand, on $\Aa_0^{\rm bc}$, either there exist $x_-\in[-2\ell,-\ell]$ and $x_+\in[\ell,2\ell]$ such that
\begin{align*}
|u(x_\pm)+1|\leq \delta/2
\end{align*}
or we have $u\in[-1+\delta/2,1-\delta/2]$ on an interval of length $\ell$. In the latter case, we get easily
\begin{align*}
E_{(-2\ell,2\ell)}(u)\gtrsim \ell\delta^2.
\end{align*}
Since this is higher order for $\ell$ large, we may assume that we are in the former case.

In the former case, we may assume without loss of generality that $u(x_{\pm})=-1+\delta$ and $u(x_0)=-\delta/2$. We then use the Modica-Mortola trick to connect the values (a) $u_-$ and $u(x_-)$, (b) $u(x_-)$ and $u(x_0)$, (c) $u(x_0)$ and $u(x_+)$, and (d) $u(x_+)$ and $u_+$. We conclude in the usual way that
\begin{align*}
\inf_{\Aa_0^{\rm bc}} E_{(-2\ell,2\ell)}(u)\geq \varphi_{-1}(u_-)+\varphi_{-1}(u_+)+c_0-C\delta.
\end{align*}
Together with~\eqref{upinf}, this completes the proof of Lemma~\ref{l:last}.
\end{proof}

\subsection{Proof of lemmas related to the uniform distribution}\label{ss:unilem}

\begin{proof}[Proof of Lemma~\ref{l:smallu}]
Our argument relies on an iterated rescaling,  illustrated in Figure \ref{fig:7}.

 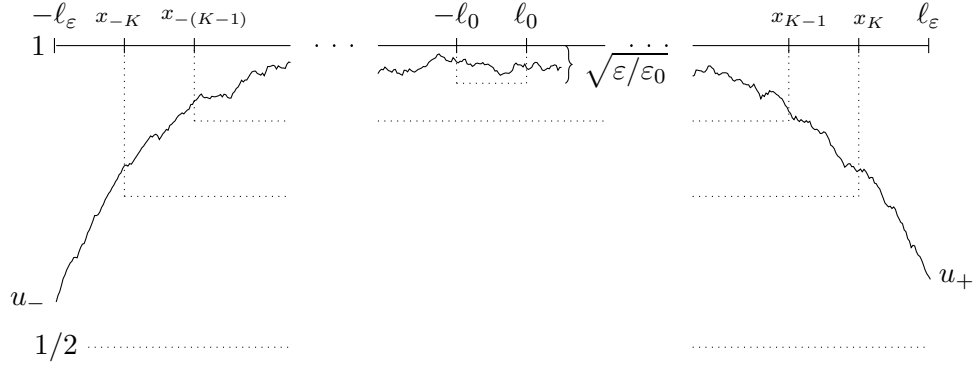
\begin{figure}
\tikzset{decorate sep/.style 2 args=
{decorate,decoration={shape backgrounds,shape=circle,shape size=#1,shape sep=#2}}}

\begin{tikzpicture}[xscale=0.023, yscale=2, >=stealth]

\draw[black] plot file{./Graphs/P7G1.txt};
\draw[black] plot file{./Graphs/P7G2.txt};
\draw[black] plot file{./Graphs/P7G3.txt};

\draw (1,.3) node[anchor=east]{$u_-$};
\draw (500,.45) node[anchor=west]{$u_+$};

\draw [decorate,decoration={brace,amplitude=2.5pt,mirror,raise=6pt},yshift=0pt]
(283,1.75) -- (283,2) node [fill=white,midway,xshift=28.5pt, yshift =-2pt] {
$\sqrt{\varepsilon/\varepsilon_0}$};

%
\draw[black] (1,2) node[fill=white, anchor = east]{$1$} -- (135,2);
\draw[black] (365,2) -- (500,2);
\draw[black] (185,2) -- (315,2);

\foreach \x in {150,160,170,330,340,350}
\filldraw[black] (\x,2) ellipse (12pt and .15pt);

\draw[black,style=dotted] (1,0) node[fill=white]{$1/2$} -- (135,0);
\draw[black,style=dotted] (365,0) -- (500,0);

\draw[black,style=dotted] (40,1) -- (135,1);
\draw[black,style=dotted] (365,1) -- (460,1);
\draw[black,style=dotted] (460,1) -- (460,2);
\draw[black,style=dotted] (40,1) -- (40,2);

\draw[black,style=dotted] (80,1.5) -- (135,1.5);
\draw[black,style=dotted] (365,1.5) -- (420,1.5);
\draw[black,style=dotted] (420,1.5) -- (420,2);
\draw[black,style=dotted] (80,1.5) -- (80,2);

\draw[black] (420,1.95) -- (420,2.05)node[fill=white, xshift = 4pt, anchor = south]{\scriptsize$x_{K-1}$};
\draw[black] (460,1.95) -- (460,2.05)node[fill=white, xshift = 4pt, anchor = south]{\scriptsize$x_{K}$};
\draw[black] (500,1.95) -- (500,2.05) node[fill=white, anchor = south]{$\ell_\eps$};

\draw[black] (40,1.95) -- (40,2.05) node[fill=white, xshift = -2pt, anchor = south]{\scriptsize$x_{-K}$};
\draw[black] (80,1.95) -- (80,2.05) node[fill=white, xshift = 4pt, anchor = south]{\scriptsize$x_{-(K-1)}$};
\draw[black] (0,1.95) -- (0,2.05) node[ anchor = south]{$-\ell_\eps$};

\draw[black,style=dotted] (185,1.5) -- (315,1.5);

\draw[black,style=dotted] (230,1.75) -- (270,1.75);
\draw[black,style=dotted] (230,1.75) -- (230,2);
\draw[black,style=dotted] (270,1.75) -- (270,2);
\draw[black] (230, 1.95) -- (230,2.05) node[fill=white, anchor = south]{$-\ell_0$};
\draw[black] (270,1.95) -- (270,2.05) node[fill=white, anchor = south]{$\ell_0$};

\end{tikzpicture}
\caption{By  iterated rescaling and application of the large deviation bounds we show that the paths relax to a $O( \eps^{1/2} )$ - neighbourhood of $1$ within a distance of $O |\log(\eps)|$.}
\label{fig:7}
\end{figure}

We will define $K=K_\eps\geq 1$ below.  We begin by enumerating the partition $\{x_{k}\}_{k=-(K+1)}^{K+1}$ of $(-\ell_\eps,\ell_\eps)$ with width $2\ell_0$, so that
$$x_{\pm 1}=\pm \ell_0,\;x_{\pm 2}=\pm 3\ell_0,\;\ldots,\;x_{\pm K}=\pm(2K-1)\ell_0,\; x_{\pm(K+1)}=\pm\ell_\eps.$$ For brevity of notation, let  $$\mathcal{A}:=\Big\{u\colon |u(x_k)-1|\leq \frac{1}{2}\,\text{for all}\,k\in \{-(K+1),-K,\ldots,K+1\} \Big\}.$$
We will use the  elementary facts from probability that for any  sets $A_1$, $A_2$, and $A_3$, we have
  \begin{align}
   P(A_1\cap A_2)&\leq P(A_1\,\cap\,A_3)+P(\complement A_3\cap A_2),\label{prob1}\\
        P(A_1\cap A_2\cap A_3)&\leq P(A_1\cap A_2|A_3).\label{prob2}
   \end{align}

We also use the Markov property from Lemma~\ref{le:Markov1b} to deduce the following property for conditional measures. If $\mathcal{A}_{in}$ and $\tilde{\mathcal{A}}_{in}$ are in $\F_{[-x_2,x_2]}$ and $\mathcal{A}_{out}$ is in $\F_{[-\ell_\eps,-x_2]}\vee\F_{[x_2,\ell_\eps]}$, then
\begin{align}
 &\mu_{\eps,(-\ell_\eps,\ell_\eps)}^{u_-,u_+}\left(\Aa_{in} \bigg| u\in\tilde{\Aa}_{in}\cap\Aa_{out}\;\text{and}\;u(\pm x_2)\in (a,b) \right)\notag\\
&=\frac{\E_{(-\ell_\eps,\ell_\eps)}^{\mu_\eps,u_-,u_+}\left(\mathbf{1}_{\Aa_{out}}\mathbf{1}_{u(\pm x_2)\in (a,b)}\E_{(-x_2,x_2)}^{\mu_\eps,u^2_-,u^2_+}\big(\mathbf{1}_{\Aa_{in}}\mathbf{1}_{\tilde\Aa_{in}}\big)\right)}{\E_{(-\ell_\eps,\ell_\eps)}^{\mu_\eps,u_-,u_+}
\left(\mathbf{1}_{\Aa_{out}}\mathbf{1}_{u(\pm x_2)\in (a,b)}\E_{(-x_2,x_2)}^{\mu_\eps,u^2_-,u^2_+}\big(\mathbf{1}_{\tilde\Aa_{in}}\big)\right)}\notag\\
&\leq\sup_{u^2_\pm\in (a,b)}\mu_{\eps,(-x_2,x_2)}^{u^2_-,u^2_+}\big(u\in\Aa_{in}\big| u\in\tilde\Aa_{in}\big).\label{condmark}
\end{align}

Keeping these preliminaries in mind, we now observe that we can make the following decomposition:
\begin{eqnarray*}
&&\mu_{\eps,(-\ell_\eps,\ell_\eps)}^{u_-,u_+}\left(\sup_{x\in[-x_1,x_1]}|u(x)-1|\geq\frac{1}{2^{K+1}}\bigg|\,u\in\mathcal{A}\right)\\
&\overset{\eqref{prob1}}\leq&\mu_{\eps,(-\ell_\eps,\ell_\eps)}^{u_-,u_+}\Bigg(\sup_{x\in[-x_1,x_1]}|u(x)-1|\geq\frac{1}{2^{K+1}}\;\;\text{and}\,\\
&& \qquad \qquad \qquad \sup_{x\in[-x_2,x_2]}|u(x)-1|\leq \frac{1}{2^K}\bigg|u\in\mathcal{A}\Bigg)\\
&&\qquad +\mu_{\eps,(-\ell_\eps,\ell_\eps)}^{u_-,u_+}\left(\sup_{x\in[-x_2,x_2]}|u(x)-1|\geq \frac{1}{2^K}\bigg|\;u\in\mathcal{A}\right).
\end{eqnarray*}

For the first term, we can now send the smallness condition into the boundary conditions in the following way:
\begin{eqnarray*} &&\mu_{\eps,(-\ell_\eps,\ell_\eps)}^{u_-,u_+}
\bigg(\sup_{x\in[-x_1,x_1]}|u(x)-1|\geq\frac{1}{2^{K+1}}\;\; \text{and}\,\\
&& \qquad \qquad \qquad \sup_{x\in[-x_2,x_2]}|u(x)-1|\leq \frac{1}{2^K} \bigg| u\in\mathcal{A}\bigg)\\
&\overset{\eqref{prob2}}\leq& \mu_{\eps,(-\ell_\eps,\ell_\eps)}^{u_-,u_+}
\bigg(\sup_{x\in[-x_1,x_1]}|u(x)-1|\geq\frac{1}{2^{K+1}}\;\; \text{and}\,\\
&& \qquad \qquad  \sup_{x\in[-x_2,x_2]}|u(x)-1|\leq \frac{1}{2^K} \bigg|\\
&& \qquad \qquad
 |u(x)-1|\leq \frac{1}{2^K}\;\text{for}\,x\in\{\pm x_1,\pm x_2\}\;\text{and}\;u\in\Aa\bigg)\\
&\overset{\eqref{condmark}}\leq& \sup_{u^2_{\pm}\in[1-2^{-K},1+2^{-K}]}\mu_{\eps,(-x_2,x_2)}^{u_-^2,u_+^2}\bigg(\sup_{x\in[-x_1,x_1]}|u(x)-1|\geq\frac{1}{2^{K+1}} \\
&& \qquad \qquad  \;\;\text{and}\,
 \sup_{x\in[-x_2,x_2]}|u(x)-1|\leq \frac{1}{2^K}\bigg|
 |u(\pm x_1)-1|\leq \frac{1}{2^K}\bigg)\\
&\leq&\sup_{u^2_{\pm}\in[1-2^{-K},1+2^{-K}]}\  \mu_{\eps,(-x_2,x_2)}^{u_-^2,u_+^2}\bigg(\sup_{x\in[-x_1,x_1]}|u(x)-1|\geq\frac{1}{2^{K+1}} \bigg| \\
&&\qquad \qquad |u(\pm x_1)-1|\leq \frac{1}{2^K}\bigg).
\end{eqnarray*}
We can iterate this argument to reduce the probability to the form:
\begin{align}
&\mu_{\eps,(-\ell_\eps,\ell_\eps)}^{u_-,u_+}\left(\sup_{x\in[-x_1,x_1]}|u(x)-1|\geq\frac{1}{2^{K+1}}\bigg|
\,u\in\mathcal{A}\right)\notag\\
&\leq\sum_{k=1}^K \sup_{u_{\pm}^k \in [1- 2^{-k}, 1 + 2^{-k} ]} \notag\\
&\qquad  \mu_{\eps,(x_{-(K-k+2)},x_{K-k+2})}^{u_-^{k},u_+^{k}}\bigg(\sup_{x\in[x_{-(K-k+1)},x_{K-k+1}]}\!\!\!\!\!\!\!\! |u(x)-1|\geq
\frac{1}{2^{k+1}}\bigg|\notag\\
& \hspace{60pt}|u(x_{\pm (K-k+2)})-1|\leq \frac{1}{2^{k}}\bigg).\label{deco}
\end{align}
Hence it remains to estimate the individual terms in the sum.
The argument involves three steps: a large deviation estimate, concatenation, and an iterated rescaling of the deviation of $u$ from $1$.

\medskip

\underline{Step 1: Large deviation estimate}. The first step is to derive a uniform large deviation bound for the measures $\mu_{\eps,(-3\ell_0,3\ell_0)}^{u_-,u_+}$. We show that there exists $C<\infty$ such that for every $\ell_0<\infty$ sufficiently large, there exists $\eps_0'>0$ such that for any $u_{\pm}\in[1/2,3/2]$ and $\eps\leq\eps_0'$, we get
\begin{align}
 \mu_{\eps,(-3\ell_0,3\ell_0)}^{u_{-},u_{+}}& \bigg( \sup_{x \in [-\ell_0,\ell_0]}   |u(x)-1|\geq\frac{1}{4}\bigg|
|u(\pm \ell_0)-1|\leq\frac{1}{2}\bigg)\notag\\
 &\qquad\qquad\leq \exp\left(-\frac{1}{C\eps}\right). \label{f2}
\end{align}
In the next steps, we will always assume $\eps_0\leq\eps_0'$ to be sufficiently small in this sense, and this is the only restriction on $\eps_0$ in the proof of the lemma.

To bound the conditional probability in~\eqref{f2} it suffices to establish an upper bound on
\begin{align}
 \mu_{\eps,(-3\ell_0,3\ell_0)}^{u_{-},u_{+}}& \bigg( \sup_{x \in [-\ell_0,\ell_0]}   |u(x)-1|\geq\frac{1}{4}\;\;\text{and}\;\; |u(\pm\ell_0)-1|\leq\frac{1}{2}\bigg)\label{ff2}
\end{align}
and a lower bound on
\begin{align}
 \mu_{\eps,(-3\ell_0,3\ell_0)}^{u_{-},u_{+}}& \bigg( |u(\pm \ell_0)-1|\leq\frac{1}{2}\bigg), \label{ff3}
\end{align}
uniformly with respect to $u_\pm\in[1/2,3/2])$.
To this end, we turn to the uniform large deviation estimates from Propositions~\ref{pr:LD1} and~\ref{pr:LD2}. In fact, we do not even need the second condition in~\eqref{ff2}, and it suffices to  bound the probability of the larger set
\begin{align*}
\Aa_0:=\Big\{u\in C([-3\ell_0,3\ell_0])\colon u(\pm 3\ell_0)=u_\pm,\;\;\sup_{x \in [-\ell_0,\ell_0]}   |u(x)-1|\geq\frac{1}{4}\Big\},
\end{align*}
The estimate~\eqref{e:LDUB1} gives that for any $\gamma, \,\delta>0$, we have for sufficiently small $\eps$ that
\begin{equation}
 \mu^{u_-,u_+}_{\eps,(-3 \ell_0,3 \ell_0)}(\Aa_0) \leq \exp  \Big(-\frac{1}{\eps} \big(  \DE\big(B( \Aa_0,\delta)  \big) - \gamma  \big)  \Big),\label{fld}
\end{equation}
where $\DE$ is defined in~\eqref{endiff} and
$$\A^{\rm bc}=\{u\in C([-3\ell_0,3\ell_0])\colon u(\pm 3\ell_0)=u_\pm\}.$$
Consider now a small $\delta>0$ to be fixed below and a function $u\in B(\Aa_0,\delta)$. Because  the boundary conditions are in $[1/2,3/2]$ and $\ell_0$ is large, the infimum of the energy must take place over functions such that
\begin{align}
\max\bigg\{\min_{x\in[-3\ell_0,-\ell_0]}|u(x)-1|,\;\min_{x\in[\ell_0,3\ell_0]}|u(x)-1|\bigg\}\lesssim \frac{1}{\sqrt{\ell_0}}.\label{felixstar}
\end{align}
(Indeed, $u$ must be close to either $1$ or $-1$ at some point in each of the intervals, and if $u$ were instead close to $-1$ on either  interval, satisfying the boundary conditions would lead to an even greater energetic cost than the one we will arrive at below.) Let us label the minimizing points $x_-$ and $x_+$. Moreover, let us define $x_*$ to be a point in $(-\ell_0,\ell_0)$ such that
\begin{align*}
|u(x_*)-1|\geq \frac{1}{4}-\delta.
\end{align*}
As above in Subsection \ref{ss:enlem}, we now define $\varphi(u):=|\int_u^1\sqrt{2V(s)}\,ds|$ and apply the ``Modica-Mortola trick'' on $(-3\ell_0,x_-)$, $(x_-,x_*)$, $(x_*,x_+)$, and $(x_+,3\ell_0)$ to recover
\begin{eqnarray}
\inf_{B(\Aa_0,\delta)}E(u)&\geq& \varphi(u_-)-\varphi(u(x_-))+\varphi(u(x_*))-\varphi(u(x_-))\notag\\
&&\qquad+\varphi(u(x_*))-\varphi(u(x_+))+\varphi(u_+)-\varphi(u(x_+))\notag\\
&\overset{\eqref{felixstar}}\geq& 2\varphi(u(x_*))+\varphi(u_-)+\varphi(u_+)-o(1)_{\ell_0\uparrow\infty}\notag\\
&\geq& 2\varphi_{1/4}+\varphi(u_-)+\varphi(u_+)-o(1)_{\ell_0\uparrow\infty}-o(1)_{\delta\downarrow 0},\label{flow}
\end{eqnarray}
where $$\varphi_{1/4}:=\min\{\varphi(3/4),\varphi(5/4)\}.$$

On the other hand, a standard construction gives
\begin{equation}
\inf_{\mathcal{A}^{\rm bc}} E(u)\leq \varphi(u_-)+\varphi(u_+)+o(1)_{\ell_0\uparrow\infty}.\label{fup}
\end{equation}
Now fixing $\delta>0$ and $\gamma>0$ sufficiently small, the combination of~\eqref{fld},~\eqref{flow}, and~\eqref{fup} gives for sufficiently small $\eps$ that
\begin{equation}
 \mu^{u_-,u_+}_{\eps,(x_-,x_+)}(\Aa_0) \leq \exp\left( -\frac{3/2\;\varphi_{1/4}}{\eps}\right). \label{fiq}
 \end{equation}

We now remark that the lower bound on~\eqref{ff3} follows easily from Proposition~\ref{pr:LD2}. Indeed, for a fixed $0<\delta<\frac12$, the set of interest can be written as the $\delta$ ball around the set $\mathcal{A}_1$ defined as
\begin{align*}
\Aa_1:=\Big\{u\colon u(\pm 3\ell_0)=u_\pm,\;\; |u(\pm \ell_0)-1|\leq\frac{1}{2}-\delta\Big\}.
\end{align*}
We recover for any $\gamma>0$ and for $\eps>0$ sufficiently small that
\begin{equation}\label{feak}
 \mu^{u_-,u_+}_{\eps,(x_-,x_+)} \big(B(\Aa_1,\delta)\big) \geq \exp  \Big(-\frac{1}{\eps} \big(  \DE \big( \Aa_1  \big) + \gamma  \big)  \Big),
\end{equation}
where $\DE$ is defined in~\eqref{endiff}. The constraint in $\Aa_1$ is inactive/slack in the optimization for $\ell_0$ sufficiently large, and the usual construction together with the usual Modica-Mortola estimate thus gives
\begin{align*}
\DE \big(\Aa_1\big)\leq o(1)_{\ell_0\uparrow\infty}.
\end{align*}
Plugging back into~\eqref{feak} gives
\begin{align*}
\mu^{u_-,u_+}_{\eps,(x_-,x_+)} \big(B(\Aa_1,\delta)\big) \geq \exp  \Big(-\frac{ 2 \gamma }{\eps}  \Big),
\end{align*}
which together with~\eqref{fiq} gives~\eqref{f2} with $C=1/\varphi_{1/4}$ as long as $\gamma$ is chosen sufficiently small.

\medskip

\underline{Step 2: Concatenation}. The next step is to prove for any $K \in \N$ that
\begin{align}
\mu_{\eps,(x_{-(K+1)},x_{K+1})}^{u_-,u_+}&\bigg(\sup_{x\in[x_{-K},x_K]}|u(x)-1|\geq \frac{1}{4}\;\bigg|u\in\mathcal{A}\bigg)\notag\\
&\qquad\leq 2\,K\exp\left(-\frac{1}{C\eps}\right),\label{bigcat}
\end{align}
uniformly for $u_{\pm}\in(1/2,3/2)$.
As usual, the idea is to break up the larger interval by conditioning on the boundary values. The restriction of the boundary values on each subinterval to $(1/2,3/2)$ because of $u\in\mathcal{A}$ will allow us to apply the uniform estimate from Step 1.

We will consider the non-overlapping subintervals $[x_k,x_{k+1}]$ for $k=-K,\ldots,K-1$. Decomposing the interval in this way gives
\begin{align}
&\mu_{\eps,(x_{-(K+1)},x_{K+1})}^{u_-,u_+}\bigg(\sup_{x\in[x_{-K},x_K]}|u(x)-1|\geq \frac{1}{4}\;\;\bigg|u\in\mathcal{A}\bigg) \notag\\
&\leq\sum_{k=1}^{2K} \mu_{\eps,(x_{-(K+1)},x_{K+1})}^{u_-,u_+}\bigg(\sup_{x\in[x_k,x_{k+1}]}|u(x)-1|\geq\frac{1}{4}\bigg|u\in\mathcal{A}\bigg) \label{e:p60firsteq}.
\end{align}
Now the  Markov property implies that for $k \in \{-K, \ldots ,K-1  \}$ we have
\begin{eqnarray*}
& &\mu_{\eps,(x_{-(K+1)},x_{K+1})}^{u_-,u_+}\bigg( \sup_{x\in[x_k,x_{k+1}]}|u(x)-1|\geq\frac{1}{4} \, \text{and} \,u\in\mathcal{A}\bigg) \\
& =& \E_{(x_{-(K+1)},x_{K+1})}^{\mu_\eps,u_-,u_+}\left( \prod_{j=-K}^{K}   \mathbf{1}_{  \{ |u(x_j) -1 | \leq \frac12 \}}   \mathbf{1}_{ \{ \sup_{[x_k,x_{k+1}]}|u(x)-1|\geq\frac{1}{4} \}} \right) \\
& \overset{\eqref{condmark}}\leq&  \sup_{u^k_{\pm} \in [1/2, 3/2]}   \mu^{u_-^k,u_+^k}_{\eps,(x_{k-1},x_{k+2})}  \Big( \sup_{x \in [x_k,x_{k+1}] }  \big| u(x) -1 \big|  \geq  \frac{1}{4}  \\
&&  \Big|  \, \max\{| u(x_k) -1| , | u(x_{k+1}) -1|\} \leq  \frac{1}{2}   \Big)  \mu_{\eps,(x_{-(K+1)},x_{K+1})}^{u_-,u_+}\big( \Aa \big) .
\end{eqnarray*}
Hence, using the translational invariance of the measures $  \mu^{u_-,u_+}_{\eps,(x_-,x_+)}$, we bound the right-hand side of equation \eqref{e:p60firsteq} by
\begin{eqnarray*}
&2\,K\!\!&  \sup_{u_{\pm} \in [1/2, 3/2]}  \mu_{\eps,(-3\ell_0,3\ell_0)}^{u_-,u_+}\bigg(\sup_{x\in[-\ell_0,\ell_0]}|u(x)-1|\geq\frac{1}{4} \bigg| |u(\pm\ell_0)-1|\leq \frac{1}{2}\bigg)\\
&\overset{\eqref{f2}}\leq&2\,K\,\exp\left(-\frac{1}{C\eps}\right),
\end{eqnarray*}
which is what we wanted to show.

\underline{Step 3: Rescaling and iteration}. In this step, we rescale the deviation of $u$ from $1$. We fix $k$ and consider the random variables
$$\hat{u}:=2^{k-1}(u-1)+1.$$
For $u$ distributed according to $\mu_{\eps,(-\ell,\ell)}^{u_-,u_+}$,the profile $\hat{u}$ is distributed according to a rescaled version of the measure. Indeed, the Radon-Nikodym density with respect to the Brownian bridge measure with \emph{modified noise strength}  $\hat \eps:=4^{k-1}\eps$ and rescaled boundary conditions is proportional to
\begin{align}
\exp\left(-\frac{1}{\hat \eps}\int_{-\ell}^{\ell}\,4^{k-1} V\Big(2^{-(k-1)}(\hat{u}-1)+1\Big)\,dx\right).
\end{align}
Let us give a name to the modified potential
\begin{align*}
\hat V(\hat{u}):=4^{k-1} V(2^{-(k-1)} (\hat{u}-1)+1)
\end{align*}
and the associated energy
\begin{align*}
\hat E(\hat{u}):=\int_{-\ell}^\ell\left(\frac{1}{2}(\partial_x \hat{u})^2+\hat{V}(\hat{u})\right)\,dx.
\end{align*}
We now make a series of observations that will allow us to apply the same large deviation bounds from Steps 1 and 2 to the rescaled random variables $\hat{u}$. 

First consider how the sets involved in~\eqref{deco} behave under the rescaling. Notice that $u$ satisfies
\begin{align*}
 \sup_{x\in[x_{-(K-k+2)},x_{K-k+2}]}|u(x)-1|\geq
\frac{1}{2^{k+1}}
\end{align*}
precisely when
\begin{align*}
 \sup_{x\in[x_{-(K-k+2)},x_{K-k+2}]}|\hat{u}(x)-1|\geq\frac{1}{4}.
\end{align*}
Similarly, for the set on which we condition, we have that $u$ satisfies
\begin{align*}
 |u(x_j)-1|\leq \frac{1}{2^{k}}\;\text{for}\;j\in\{-(K-k+3),\ldots,{K-k+3}\}
\end{align*}
precisely when $\hat u$ satisfies
\begin{align*}
 |\hat{u}(x_j)-1|\leq \frac{1}{2}\;\text{for}\;j\in\{-(K-k+3),\ldots,{K-k+3}\}.
\end{align*}
Hence each term in~\eqref{deco} can be bounded if we can establish that the bound from Step 2 also holds for the measure governing $\hat u$.

In order to show that the estimates from Step 1 and 2  hold uniformly for the measure of the rescaled random variables $\hat u$, we need to be able to invoke Propositions~\ref{pr:LD1} and~\ref{pr:LD2} (with uniform constants).
This in turn requires uniform control on the boundary values, the minimum  energy $\hat E$ over the sets of interest, and
the Lipschitz constant of $\hat V$. The boundary values are easy: On the sets of interest, the boundary values $u_\pm\in(1/2,3/2)$. On the other hand, the  minimum of the  energy $\hat{E}$ is bounded uniformly with respect to $k$ on the sets of interest. Indeed, 
consider 
\begin{align*}
\mathcal{C}:=\Big\{u\colon |u(x_j)-1|&\leq \frac{1}{2^{k}}\;\text{for}\;j\in\{-(K-k+3),\ldots,{K-k+3}\}\\ 
&\text{ and } \sup_{x\in[x_{-(K-k+2)},x_{K-k+2}]}|u(x)-1|\geq
\frac{1}{2^{k+1}}\Big\}
\end{align*}
and let $\hat{\mathcal{C}}$ denote the image of the set under the transformation $u\rightarrow\hat{u}$.
By the usual method (``Modical Mortola trick'' for the lower bound and construction for the upper bound), one can check that there exists $R<\infty$ such that, for every $k\in\N$, one has
\begin{align*}
\inf_{\hat{u}\in\mathcal{\hat{C}}}\,\hat{E}(\hat u)=4^{k-1}\inf_{u\in\mathcal{C}}\,E(u)\leq R.
\end{align*}
Finally, because of Assumption~\ref{ass:V}, we have a uniform bound on the Lipschitz constant of $V$. Indeed, let $C:=3/2+2\ell_0 R+1$. Then uniformly with respect to $k\in\N$, the potential $\hat V$ satisfies
\begin{align*}
\sup_{|\hat u|\leq C} |\hat V'(\hat u)|&\leq\sup_{|u-1|\leq 2^{-k+1}(C+1)}\,2^{k-1}\,|V'(u)|\\
&\leq \,\sup|V''(\tau)|(C+1),
\end{align*}
where the supremum is taken over $\tau\in[1-2(C+1),1+2(C+1)].$

Hence, the potential satisfies the requirements of Propositions~\ref{pr:LD1} and~\ref{pr:LD2}. The remaining requirement in order to invoke large deviation theory is that
$$4^{k-1}\eps\leq\eps_0\qquad\text{for all}\;k\leq K,$$
which is true if
$$4^{K-1}\eps\leq\eps_0.$$
Therefore we choose $K$ to be an integer satisfying
\begin{align}
\frac{1}{2^{K+1}}\leq\sqrt{\frac{\eps}{\eps_0}}\leq\frac{1}{2^{K-1}}.\label{kook2}
\end{align}
With the restriction~\eqref{kook2} on $K$, the arguments used in Step 1 and Step 2 carry over to the rescaled measures governing the $\hat{u}$.

We are now ready to complete the argument. Indeed, recalling the decomposition from~\eqref{deco}, we have
\begin{eqnarray}
&&\mu_{\eps,(-\ell_\eps,\ell_\eps)}^{u_-,u_+}\left(\sup_{x\in[-\ell_0,\ell_0]}|u(x)-1|\geq\sqrt{\frac{\eps}{\eps_0}}\;\bigg|\,u\in\mathcal{A}\right)\notag\\
&\overset{\eqref{kook2}}\leq&\mu_{\eps,(-\ell_\eps,\ell_\eps)}^{u_-,u_+}\left(\sup_{x\in[-\ell_0,\ell_0]}|u(x)-1|\geq\frac{1}{2^{K+1}}\;\bigg|
\,u\in\mathcal{A}\right)\notag\\
&\overset{\eqref{deco}}\leq&\sum_{k=1}^K \sup_{u_{\pm}^k \in [1- 2^{-k}, 1 + 2^{-k} ]} \notag\\
&& \hspace{12pt}\mu_{\eps,(x_{-(K-k+2)},x_{K-k+2})}^{u_-^{k},u_+^{k}}\bigg(\sup_{x\in[x_{-(K-k+1)},x_{K-k+1}]}|u(x)-1|\geq
\frac{1}{2^{k+1}}\bigg|\notag\\
&& \hspace{12pt}|u(x_{\pm(K-k+2)})-1|\leq \frac{1}{2^{k}}\bigg).\label{cow}
\end{eqnarray}
From the preceding argument, we can now apply the estimate~\eqref{bigcat} for the rescaled measures to bound the $k^{th}$ summand above by
\begin{align*}
2(K-k+1)\exp\left(-\frac{1}{C\,4^{k-1}\eps}\right).
\end{align*}
Substituting into the right-hand side of~\eqref{cow}, we deduce
\begin{eqnarray*}
&&\mu_{\eps,(-\ell_\eps,\ell_\eps)}^{u_-,u_+}\left(\sup_{x\in[-\ell_0,\ell_0]}|u(x)-1|\geq\sqrt{\frac{\eps}{\eps_0}}\;\bigg|\,u\in\mathcal{A}\right)\\
&\leq&\sum_{k=1}^{K} 2(K-k+1)\exp\left(-\frac{1}{C\,4^{k-1}\eps}\right)\\
&=&2\sum_{k=0}^{K-1}(K-k)\left(\exp\left(-\frac{1}{C\,4^K\,\eps}\right)\right)^{4^{K-k}}\\
&\overset{\eqref{kook2}}\leq&2\sum_{k=0}^{K-1}(K-k)r^{4^{K-k}},\qquad \text{for}\;r:=\exp\left(-\frac{1}{C\,\eps_0}\right)\\
&\leq&2\sum_{k'=1}^{\infty}k'r^{k'}=\frac{2r}{(r-1)^2}\leq 4r\qquad\text{for}\;r\in(0,1/4].
\end{eqnarray*}
\end{proof}

\begin{proof}[Proof of Lemma \ref{l:hittingzero}]  We start by defining some sets. We denote the set of paths that  we condition on by
\begin{align*}
\Aa &:= \{ u \in C([-\ell_\eps, \ell_\eps]) \colon    |u(\pm(2k-1)\ell_0)-1|\leq \frac{1}{2},\,k=1,2,\ldots, K_\eps  \}.
\end{align*}
For $\eps, \eps_0>0$ let us also fix the following subset of $\Aa$
 \begin{align*}
\Aa_\eps &:= \bigg\{ u  \in \Aa \colon    |u(\pm \ell_0)-1|\leq \left(\frac{\eps}{\eps_0}\right)^{1/2} \bigg\}.
\end{align*}
Then  Lemma \ref{l:smallu} implies in particular that, for a small but fixed $\eps_0>0$ and for  $\eps \leq \eps_0$, we have
\begin{equation}\label{e:hz10}
 \mu_{\eps,(-\ell_\eps,\ell_\eps)}^{u_{-},u_{+}}  \big( \Aa_\eps \big) \geq \frac12 \mu_{\eps,(-\ell_\eps,\ell_\eps)}^{u_{-},u_{+}}  \big( \Aa \big).
\end{equation}
From now on, we fix an $\eps_0$ such that this identity holds. This will be the only restriction on $\eps_0$.

Let us also introduce a notation for the set of paths that have a hitting point of $1$ in $[-\ell_0,\ell_0]$
\begin{align*}
\Bb &:= \{ u \in C([-\ell_\eps, \ell_\eps]) \colon  \exists x \in [-\ell_0, \ell_0] \text{ such that } u(x) = 1  \}.
\end{align*}
As a slight abuse of notation we will use the same letter $\Bb$ to denote the set of paths $u\in \Bb$ restricted to $[-\ell_0,\ell_0]$.

Using the Markov property \eqref{e:Markovmu}, we get for any $u^{\pm} \in [1/2, 3/2]$ that
\begin{align}
\mu_{\eps,(-\ell_\eps,\ell_\eps)}^{u_{-},u_{+}}  \big( \Aa \cap \Bb \big)  &\geq  \mu_{\eps,(-\ell_\eps,\ell_\eps)}^{u_{-},u_{+}}  \big( \Aa_\eps \cap \Bb \big)  \label{e:hz1}\\
 &  =  \mathbb{E}_{(- \ell_\eps, \ell_\eps)}^{\mu_\eps,u_{-},u_{+}} \left(   \mathbf{1}_{\Aa_\eps }(u)   \,  \mu_{\eps, (-\ell_0,\ell_0)}^{u(-\ell_0) ,u(+\ell_0)} \big( \Bb \big)  \,  \right)   \notag.
\end{align}
Our  main task is thus to derive a lower bound for the probabilities
\begin{equation}\label{e:ProbBb}
 \mu_{\eps, (-\ell_0,\ell_0)}^{u_- ,u_+} \big( \Bb \big)
\end{equation}
that holds uniformly in the boundary conditions.  In view of the definition of $\A_\eps$, it is sufficient to consider boundary conditions $u_{\pm}$ that are $O(\eps^{1/2})$ close to $1$:
\begin{equation}\label{e:conbcc}
1- \left(\frac{ \eps}{\eps_0}\right)^{1/2} \leq u_{\pm} \leq 1 + \left(\frac{\eps}{\eps_0}\right)^{1/2}.
\end{equation}
 As in the proof of Lemma \ref{l:smallu}, we rescale the random profile $u$ around $1$, this time by a factor $\eps^{-\frac12}$. More precisely,  we consider the  transformation
\begin{equation*}
 \hat{u}(x) :=   \eps^{-1/2}  (u(x)-1) +1.
\end{equation*}
According to its definition, a path $u$ is in the set $\Bb$ if and only if $\hat{u}$ is in $\Bb$. Hence, we can express the probability \eqref{e:ProbBb} in terms of $ \hat{u}$.

The random variable $\hat{u}$ is distributed according to a rescaled version of $\mu_{\eps, (-\ell_0,\ell_0)}^{u_- ,u_+} $. The variance of the Gaussian reference measure becomes one and the rescaled boundary values are
\begin{equation*}
 \hat{u}_{\pm}:= \eps^{-1/2}  (u_{\pm}-1) +1.
\end{equation*}
Note that the condition \eqref{e:conbcc} implies that these rescaled boundary conditions take values in an order-one interval around $1$.
More precisely, the distribution of $\hat{u}$ is absolutely continuous with respect to $ \mathcal{W}^{\hat{u}_{-}, \hat{u}_{+}}_{ 1, (-\ell_0,\ell_0)}$ and the  Radon Nikodym density of the rescaled measure is proportional to $\exp \Big( - \int^{\ell_0}_{-\ell_0} \hat{V} \big(\hat{u} \big) \, dx \Big)$,  where $\hat{V}(\hat{u}) := \frac{1}{\eps}  V\big(\eps^{1/2} (\hat{u} -1) +1 \big)$.
Hence we can rewrite
\begin{align}
& \mu_{\eps, (-\ell_0,\ell_0)}^{u_{-} ,u_+} \big( \Bb \big)  \notag\\
&=     \frac{ \mathbb{E}_{(-\ell_0,\ell_0)}^{\mathcal{W}_1,\hat{u}_{-},  \hat{u}_{+}} \Big( \mathbf{1}_{\Bb}(\hat{u})  \, \exp \Big( -  \int^{\ell_0}_{-\ell_0} \hat{V} ( \hat{u} ) \, dx \Big)\Big) }{  \mathbb{E}_{(-\ell_0,\ell_0)}^{\mathcal{W}_1,\hat{u}_{-},\hat{u}_{+}} \Big( \exp \Big( -  \int^{\ell_0}_{-\ell_0} \hat{V} ( \hat{u} )\, dx \Big)\Big) }. \label{e:hz5}
\end{align}
The denominator of this expression can be trivially bounded above by $1$. To get a lower bound for the  numerator, we can write for example
\begin{align}
& \mathbb{E}_{(-\ell_0,\ell_0)}^{\mathcal{W}_1,\hat{u}_{-},\hat{u}_{+}} \Big( \mathbf{1}_{\Bb}(\hat{u})  \, \exp \Big( -  \int^{\ell_0}_{-\ell_0} \hat{V} (\hat{u}) \, dx \Big) \Big) \notag \\
& \geq  \mathbb{E}_{(-\ell_0,\ell_0)}^{\mathcal{W}_1,\hat{u}_{-},\hat{u}_{+}} \Big( \mathbf{1}_{\Bb^*}(\hat{u})  \, \exp \Big( -  \int^{\ell_0}_{-\ell_0} \hat{V} ( \hat{u} ) \, dx \Big) \Big)\notag\\
 & \geq  \mathcal{W}_{1,(-\ell_0,\ell_0)}^{\hat{u}_{-},\hat{u}_{+}} \big( \Bb^* \big)\,   \inf_{\hat{u} \in \Bb^*}  \exp \Big( -  \int^{\ell_0}_{-\ell_0} \hat{V} ( \hat{u})\, dx \Big).\label{e:hz2}
\end{align}
Here we have made the probability smaller by restricting the integration to the set
\begin{equation*}
\Bb^*:= \bigg\{ \hat{u} \in \Bb \colon \sup_{x \in [- \ell_0,\ell_0] } |\hat{u}(x) -1| \leq 3 \eps_0^{-1/2} \bigg\}.
\end{equation*}
Using the translation invariance of the Gaussian measures, we can  get a lower bound on the Gaussian probabilities that holds uniformly in the boundary conditions. For example,  set
\begin{align*}
\Bb^{**}:= \big\{ \hat{u} &\in C([-\ell_0, \ell_0]) \colon \sup_{x \in [-\ell_0,\ell_0]} \hat{u}(x) \in (\eps_0^{-1/2}, 2\eps_0^{-1/2}), \\
& \quad \text{and}  \quad \inf_{x \in [-\ell_0,\ell_0]} \hat{u}(x) \in (-2\eps_0^{-1/2} ,- \eps_0^{-1/2})  \big\}.
\end{align*}
Then, on the one hand, for every path $\hat{u} \in \Bb^{**}$ and for all $u_{\eps,\pm} \in [1 - \eps^{-1/2}_0, 1+\eps^{-1/2}_0]$, the shifted paths $\hat{u} + h^{\hat{u}_{-}, \hat{u}_{+}}_{(-\ell_0,\ell_0)}$ lies in $\Bb^*$. (Recall the definition \eqref{e:Defh} of the affine profile $h^{\hat{u}_{-}, \hat{u}_{+}}_{(-\ell_0,\ell_0)}$). On the other hand, by definition, shifting by $h^{\hat{u}_{-}, \hat{u}_{+}}_{(-\ell_0,\ell_0)}$ transforms the measure $\mathcal{W}_{1,(-\ell_0,\ell_0)}^{0,0}$ into $ \mathcal{W}_{1,(-\ell_0,\ell_0)}^{\hat{u}_{-},\hat{u}_{+}}$. This implies that
\begin{equation}\label{e:hz7}
 \mathcal{W}_{1,(-\ell_0,\ell_0)}^{\hat{u}_{-},\hat{u}_{+}} \big( \Bb^* \big) \geq  \mathcal{W}_{1,(-\ell_0,\ell_0)}^{0,0} \big( \Bb^{**} \big)=: c>0.
\end{equation}

Hence it remains to get a lower bound on the second term in \eqref{e:hz2}. As above in the proof of Lemma \ref{l:smallu}, Assumption \ref{ass:V} on $V$ and Taylor's formula imply that $\hat{ V}$ satisfies 
\begin{equation*}
\sup_{|u -1| \leq 3 \, \eps_0^{-1/2} } \;\sup_{\eps\in(0,1)} \eps^{-1} V\big( \eps^{1/2} (u-1)+1 \big) =: C< \infty.
\end{equation*}
Plugging this into \eqref{e:hz2}, we get
\begin{equation*}
\inf_{\hat{u} \in \Bb^*}  \exp \Big( -  \int^{\ell_0}_{-\ell_0} \hat{V} (  \hat{u} )\, dx \Big) \geq \exp(- 2C \ell_0). \label{e:hz3}
\end{equation*}
Hence, summarizing this calculation, we get uniformly for all $u_{\pm}$ satisfying \eqref{e:conbcc} that
\begin{equation*}
\mu_{\eps, (-\ell_0,\ell_0)}^{u_{-} ,u_+} \big( \Bb \big) \geq c \exp(-2C\ell_0).
\end{equation*}
Finally, plugging this back into \eqref{e:hz1}, we get
\begin{eqnarray*}
\mu_{\eps,(-\ell_\eps,\ell_\eps)}^{u_{-},u_{+}}  \big( \Aa \cap \Bb \big) & \geq& c \exp(-2C\ell_0) \,  \mu_{\eps,(-\ell_\eps,\ell_\eps)}^{u_{-},u_{+}}  \big( \Aa_\eps \big)\\
&\overset{\eqref{e:hz10} }{\geq}& \frac{1}{2}c \exp(-2C\ell_0)\, \mu_{\eps,(-\ell_\eps,\ell_\eps)}^{u_{-},u_{+}}  \big( \Aa \big) .
\end{eqnarray*}
Thus  we get the desired conclusion for $1 - \lambda := \frac{1}{2}c \exp(-2C\ell_0) $.
\end{proof}

\begin{proof}[Proof of Lemma \ref{l:reflection}]
\underline{Step 1.} We begin by ruling out long layers to the left and to the right of $Y$. Once we know that layers are bounded in length, we can use a reflection argument as in the proof of Theorem~\ref{t:layers} to turn them into wasted excursions and estimate their probability. To this end, we define the set $\Ayt$ of functions that are bounded away from $\pm 1$ on a whole subinterval outside of $Y$:
\begin{align*}
\Ayt&:=\left\{u\in J_Y\colon \text{ there exists a $k$ with}\right.\\
&\qquad \left. k\leq k_- \text{ or }k \geq k_+ -1 \text{ such that}\right.\\
&\qquad \left. u\in[-1+\delta,1-\delta]\text{ on all of }[x_k,x_{k+1}]\right\}.
\end{align*}
As usual, we note that $\Ayth$ is contained within $\Ayt\cup (\Ayth\cap \complement\Ayt)$. Our first step is to show that~\eqref{e:rightshapea} holds for $\Ayt$. In fact, $\Ayt$ is of higher order for $M$  and $\delta^2\ell$ sufficiently large.

The set $\Ayt$ can be written in the obvious way as the union of sets $\Ayt^k$ that have bad behavior on a given subinterval $[x_k,x_{k+1}]$. Without loss of generality, suppose that $k\leq k_-$.

Then we introduce the following sets for a Markovian decomposition:
\begin{align*}
\Aa^{\ominus}_k:= &\left\{ u \colon |u(x_j)|\leq M\text{ for all }j\leq k-1\right\},\\
\Aa^{\oplus}_k:= &\left\{ u\colon |u(x_j)|\leq M \text{ for all }j\geq k+2,\right. \\
& \qquad \left.\text{ and at least one }\delta^- \text{ up layer $\leq 2\ell$ in }Y\right\},\\
\Aa^{\odot}_k:= &\left\{ u \colon  |u(x_j)|\leq M\text{ for }j=k-1, \ldots, k+2\right\},\\
\Aa^{\odot}_{\delta,k}:= &\left\{ u \in\Aa^{\odot}_k\colon   u\text{  $\in [-1+\delta,1-\delta]$ on all of $[x_{k},x_{k+1}]$}\right\}.
\end{align*}
We remark that
\begin{align*}
&\Aa^{\ominus} \in\F_{[-\Le,x_{k-1}]},\quad\text{and}\quad \Aa^\oplus\in\F_{[x_{k+2},\Le]},\\
\text{while}\quad&\Aa^\odot\in\F_{[x_{k-1},x_{k+2}]}\quad\text{ and }\quad\Aa^\odot_{\delta,k}\in\F_{[x_{k-1},x_{k+2}]}.
\end{align*}
 Consequently, the decompositions $\Ayt^{k}=\Aa^\ominus_k\cap\Aa^\odot_{\delta,k}\cap\Aa^\oplus_k$ and $\JY=\Aa^\ominus_k\cap\Aa^\odot_k\cap\Aa^\oplus_k$  lend themselves to an application of the Markov property from Lemma~\ref{le:Markov1b}.
We will often use such decompositions in the proofs below.

In the proof at hand, the Markov property from Lemma~\ref{le:Markov1b} gives
\begin{align*}
\mu_{\eps,(-\Le,\Le)}^{-1,1}\left(\Ayt^{k}\right)\leq \sup_{u_\pm \in [-M,M]}\frac{\mathbb{E}^{\mu_\eps,u_-,u_+}_{(x_{k-1},x_{k+2})}(\textbf{1}_{\Aa_{\delta,k}^\odot})}{\mathbb{E}^{\mu_\eps,u_-,u_+}_{(x_{k-1},x_{k+2})}
(\textbf{1}_{\Aa_k^\odot})}\;\mu_{\eps,(-\Le,\Le)}^{-1,1}(\JY).
\end{align*}
It suffices to bound the ratio of expectations on the right-hand side. For the denominator,
we observe that
\begin{align}
\inf_{u_{\pm}\in [-M,M]} \mathbb{E}^{\mu_\eps,u_-,u_+}_{(x_{k-1},x_{k+2})}(\textbf{1}_{\Aa_{k}^\odot})\gtrapprox 1\label{http}
\end{align}
for $M$  sufficiently large. In fact, this bound follows immediately from the large deviation bound \eqref{e:LDUB1} and a simple energy estimate applied to the complement.

Hence, it suffices to bound the numerator.
Recalling the bound~\eqref{jn20.2}, the expectation in the numerator can be estimated by
\begin{align*}
\exp\left(- \frac{1}{\eps}\bigg(\frac{\delta^2\ell}{ C_1}-2\gamma\bigg)\right)\leq \exp\left(-\frac{\delta^2\ell -1}{\eps C_1} \right).
\end{align*}
For $\delta^2\ell$ sufficiently large, this drops below the threshold expressed in the exponential in~\eqref{e:rightshape2}.
Hence, summing the probabilities of $\Ayt^k$ over $k$, the probability of $\Ayt$ is negligible in the sense that, in order to establish~\eqref{e:rightshape2}, it suffices to show that it holds for $\tilde{A}_{Y,3}:=\Ayth \setminus \Ayt$.  For ease of notation, we drop the tildes for the remainder of the proof of the lemma.

\underline{Step 2.}
We will now show the desired bound for $\Ayth$. That is, we  will show that for any $\gamma>0$ there exists an $\eps_0 >0$ such that for all $\eps \leq \eps_0$  we have
\begin{align*}
 \mu^{-1,1}_{\eps,(-\Le,\Le)} &\big( \Ayth  \big)  \ls  \Le \exp \Big( - \frac{c_0-\gamma}{\eps}  \Big)  \mu^{-1,1}_{\eps,(-\Le,\Le)}\big( \JY \big).
\end{align*}
The proof uses a reflection argument very similar to the argument in the proof of the upper bound in Theorem~\ref{t:layers}.

As above in~\eqref{Apm} the set $\Ayth$ can be expressed as $\Ayth= \Ayth^- \cup \Ayth^+$ where
\begin{align*}
\Ayth^-&= \left\{ u \in \complement \Aa_1 \cap \complement \Ayt \colon u\text{ has a }\delta^-\text{ up layer} \right.\\
&\left.\text{ contained in }[-\Le, k_-\ell]
\text{ and a $\delta^-$ up layer $\leq 2\ell$ in $J_{Y}$ }\right\},\\
\Ayth^+&= \left\{ u \in \complement \Aa_1 \cap \complement \Ayt \colon u\text{ has a }\delta^-\text{ down layer} \right.\\
&\left.\text{ contained in }[ k_+\,\ell,\Le]
\text{ and a $\delta^-$ up layer $\leq 2\ell$ in $J_{Y}$ }\right\}.
\end{align*}
We will only give the bound for the set $\Ayth^-$. The proof of the corresponding bound for $\Ayth^+$ follows in the same way. The set $\Ayth^-$ is contained in the union of $k$ from $-(N_\eps-1)$ to $k_-$ of the sets
\begin{align*}
\Ayth^{-,k}&:=
\left\{u\in\complement\Aa_1 \colon u\text{ has a }\delta^-\text{ up layer }\right.\\
&\left.\text{contained in }[x_{k-1},x_{k+1}]\text{ and a $\delta^-$ up layer $\leq 2\ell$ in $Y$ }\right\}.
\end{align*}

As in the proof of Theorem~\ref{t:layers}, we will transform the additional $\delta^-$ transition layer into a wasted $\delta^-$ excursion to control the probability.  We need to reflect in such a way as to (a) create a wasted excursion in $[x_{k-1},x_{k+1}]$ and (b) leave at least one $\delta^-$ up layer in $Y$. To this end, we define the left stopping point capturing the additional $\delta^-$ up layer
\begin{align*}
\chi_{-} = \inf \{ &x>x_{k-1}\colon u(x)=0\\
&\text{ and }u(y_1)=-1+\delta\text{ for some }y_1\in(x_{k-1},x) \}
\end{align*}
and the right stopping point
\begin{align*}
\chi_+:=\sup\big\{ & x\leq y_+ \colon u(x)=0\text{ and there exist } y_1< y_2 \\
&\text{ both in } (x,y_+)\text{ with }u(y_1)=-1+\delta,\,u(y_2)=1-\delta \big\},
\end{align*}
where $y_+:= \sup Y$ is the right boundary of $Y$. As before we will use the convention that $\chi_{\pm} = \mp \Le$ if these sets are empty.
As in the proof of Theorem~\ref{t:layers}, the reflection  operator
\begin{equation*}
\mathsf{R} = R^{\chi_+}_{\chi_-},
\end{equation*}
reflects the paths $u$ between the stopping points $\chi_{\pm}$ while preserving $\mu^{-1,1}_{\eps,(-\Le,\Le)} $. On the other hand, it maps the set $\Ayth^{-,k}$ into the set
\begin{align*}
\hat{\Aa}_{Y,3}^{-,k}:=&\left\{ u \in \complement \Aa_1  \colon
\text{at least one wasted  $\delta^-$ excursion in  $[x_{k-1}, x_{k+1}]$}\right. \\
&\qquad\left.\text{ and at least one }\delta^-\text{ up layer $\leq 2\ell$ in $\JY$}\right\}.
\end{align*}
Hence, the estimate~\eqref{e:rightshape2} will follow if we can establish, uniformly in $k$, that
\begin{align}
\mu_{\eps,(-\Le,\Le)}^{-1,1}\left(\hat{\Aa}_{Y,3}^{-,k}\right)\leq \exp\left(-\frac{c_0-\gamma}{\eps}\right)\mu_{\eps,(-\Le,\Le)}^{-1,1}(\JY),\label{intermd}
\end{align}
which will follow from the Markov property and a large deviation estimate. Indeed, let us define the following sets:
\begin{align*}
\Aa^{\ominus}_k:= &\left\{ u \colon |u(x_j)|\leq M \text{ for all $j \leq k-2$ }\right\},\\
\Aa^{\oplus}_k:= &\left\{ u\colon |u(x_j)|\leq M  \text{ for all $j \geq k+2$ }\text{ and at least one }\delta^- \text{ up layer $\leq 2\ell$ in }Y\right\},\\
\Aa^{\odot}_k:= &\left\{ u \colon  |u(x_j)|\leq M\text{ for }j=k-2, \ldots, k+2\right\},\\
\Aa^{\odot}_{w,k}:= &\left\{ u \in\Aa^{\odot}_k\colon   u\text{ has a wasted $\delta^-$ excursion in $[x_{k-1},x_{k+1}]$}\right\}.
\end{align*}
Then we can decompose $\hat{\Aa}_{Y,3}^{-,k}=\Aa^\ominus_k\cap\Aa^\odot_{w,k}\cap\Aa^\oplus_k$ and $\JY=\Aa^\ominus_k\cap\Aa^\odot_k\cap\Aa^\oplus_k$, so that applying the Markov property as in Lemma~\ref{le:Markov1b} gives
\begin{align}
\mu_{\eps,(-\Le,\Le)}^{-1,1}\left(\hat{\Aa}_{Y,3}^{-,k}\right)\leq \sup_{u_\pm \in [-M,M]}\frac{\mathbb{E}^{\mu_\eps,u_-,u_+}_{(x_{k-2},x_{k+2})}(\textbf{1}_{\Aa^\odot_{w,k}})}
{\mathbb{E}^{\mu_\eps,u_-,u_+}_{(x_{k-2},x_{k+2})}(\textbf{1}_{\Aa_{k}^\odot})}\mu_{\eps,(-\Le,\Le)}^{-1,1}(\JY).\label{crsc}
\end{align}

It now remains to estimate the ratio of expectations. Recalling~\eqref{http}, it suffices to bound the numerator.
For this purpose, we remark that~\eqref{5pm} yields that for any $\gamma>0$ and for $\delta>0$ sufficiently small, we have
\begin{align*}
\mathbb{E}^{\mu_\eps,u_-,u_+}_{(x_{k-2},x_{k+2})}(\textbf{1}_{\Aa^\odot_{w,k}})\leq\exp\big(-\frac{1}{\eps}(c_0-\gamma)\big)
\end{align*}
(where, as usual, we have redefined $\gamma$ by a factor of two). Substituting these upper and lower bounds,~\eqref{crsc} improves to~\eqref{intermd}, and the proof of Lemma~\ref{l:reflection} is complete.
\end{proof}
\begin{proof}[Proof of Lemma~\ref{l:rightshape}]
We will show~\eqref{e:rightshape}. The proof of~\eqref{e:leftshape} is similar. We can assume that the interval $J_{y,-}^\eps$ is contained in $[-\Le, \Le]$; if it is not, the proof becomes even simpler.

Given the bound~\eqref{e:Defhe2} on $\big|  \mathcal{I}_{-}^{\eps}\big|$, it is clearly sufficient to prove that for any fixed $k \in \mathcal{I}_{-}^\eps$, we have
\begin{align}
 \mu^{-1,1}_{\eps,(-\Le,\Le)}& \bigg( u \in \mathcal{J}_{y,\eps} \cap \complement \Aa_{3,y}^- \colon |u(x_{k}) + 1| \geq \frac{1}{2}   \bigg) \notag\\
&  \leq  \, \exp\bigg(-  \frac{ 3c_1}{4\eps} \bigg)\,  \mu^{-1,1}_{\eps,(-\Le,\Le)}\big( \Be \big) \label{e:n678} .
\end{align}
This in turn will follow trivially from
\begin{align}
 \mu^{-1,1}_{\eps,(-\Le,\Le)}& \bigg( u \in \mathcal{J}_{y,\eps} \colon |u(x_{k}) + 1| \geq \frac{1}{2},\, u\leq 1-\delta\text{ on } [x_{k-1},x_{k+1}]   \bigg) \notag\\
&  \leq  \, \exp\bigg(-  \frac{ 3c_1}{4\eps} \bigg)\,  \mu^{-1,1}_{\eps,(-\Le,\Le)}\big( \Be \big) \label{e:rightshape1}.
\end{align}

In order to establish~\eqref{e:rightshape1}, we again introduce a decomposition. This time we define the sets
\begin{align*}
\Aa^{\ominus}_k:= &\left\{ u \colon |u(x_j)|\leq M\text{ for all }j\leq k-2\right\},\\
\Aa^{\oplus}_k:= &\left\{ u\colon |u(x_j)|\leq M \text{ for all }j\geq k+2,\right.\\
&\left.\text{ at least one }\delta ^- \text{ up layer $\leq 2\ell$ in }J_{y,\eps}\right\},\\
\Aa^{\odot}_k:= &\left\{ u \colon  |u(x_j)|\leq M\text{ for }j=k-2,k-1,\ldots, k+2,\right. \\
&\left. u \leq 1-\delta \text{ on }[x_{k-1},x_{k+1}] \right\},\\
\Aa_{1/2,k}^{\odot}:= &\bigg\{ u \in \Aa_k^{\odot} \colon  |u(x_{k}) + 1 | \geq \frac{1}{2}  \bigg\}.
\end{align*}
The set on the left-hand side of~\eqref{e:rightshape1} can be written as $\Aa_k^{\ominus} \cap  \Aa^{\odot}_{1/2,k} \cap \Aa_k^{\oplus}$, and we have the containment
\begin{align*}
\Aa_k^{\ominus} \cap  \Aa_k^{\odot} \cap \Aa_k^{\oplus}\subseteq\Be,\\
\end{align*}
so that applying the Markov property from Lemma~\ref{le:Markov1b} leads to
\begin{align*}
& \mu^{-1,1}_{\eps,(-\Le,\Le)} \big(  \Aa_k^{\ominus} \cap  \Aa^{\odot}_{1/2,k} \cap \Aa_k^{\oplus} \big)\\
& \leq  \sup_{u_{\pm} \in [-M, 1-\delta] }    \frac{\mathbb{E}_{(x_{k-2},x_{k +2})}^{\mu_\eps,u_-,u_+} \big( \mathbf{1}_{ \Aa^{\odot}_{1/2,k} } (u)\big)}{\mathbb{E}_{(x_{k-2},x_{k +2})}^{\mu_\eps,u_-,u_+} \big( \mathbf{1}_{ \Aa_k^{\odot} } (u)\big)}    \mu^{-1,1}_{\eps,(-\Le,\Le)}(\Be).
\end{align*}
Therefore, to show the desired estimate~\eqref{e:rightshape1}, it is sufficient to establish
\begin{equation}
 \sup_{u_{\pm} \in [-M, M] }    \frac{\mathbb{E}_{(x_{k-2},x_{k +2})}^{\mu_\eps,u_-,u_+} \big( \mathbf{1}_{ \Aa^{\odot}_{1/2,k} } (u)\big)}{\mathbb{E}_{(x_{k-2},x_{k +2})}^{\mu_\eps,u_-,u_+} \big( \mathbf{1}_{ \Aa_k^{\odot} } (u)\big)}  \leq \exp\bigg(  -\frac{3c_1}{ 4\eps} \bigg).\label{38}
\end{equation}

To get a lower bound for the denominator, we will as usual use the large deviation lower bound from  Proposition~\ref{pr:LD2}. For this, we note that
\begin{align*}
&\mathcal{A}_{k}^\odot=B(\mathcal{A},\delta)\\
&\text{where } \mathcal{A}=\{u\colon |u(x_j)|\leq M-\delta\text{ for }j=k-2,\ldots, k+2, \notag\\
&\quad\quad\qquad u\leq 1-2\delta\text{ on }[x_{k-1},x_{k+1}]\}.
\end{align*}
Therefore, the large deviation bound gives that for any $\gamma>0$ and for $\eps$ small enough
\begin{align}
\mu_{\eps,(x_{k-2},x_{k+2})}^{u_-,u_+}\bigg(\mathcal{A}_{k}^\odot\bigg)\geq \exp\left(-\frac{1}{\eps}\big(\Delta E(\mathcal{A})+\gamma\big)\right).\label{eny1}
\end{align}

To get an upper bound for the numerator of~\eqref{38}, on the other hand, we will use the large deviation upper bound from Proposition~\ref{pr:LD1}. For this, we observe that the closed $\delta/2$ ball around $\mathcal{A}_{1/2,k}^\odot$  is the set
\begin{align*}
\tilde{\mathcal{A}}:=&\bigg\{u\colon |u(x_k)|\leq M+\delta, \text{ for }j=k-2,\ldots, k+2, \notag\\
&\qquad  u\leq 1-\delta/2\text{ on }[x_{k-1},x_{k+1}],\,|u(x_k)+1|\geq \frac{1-\delta}{2}  \bigg\},
\end{align*}
so that the large deviation bound gives
\begin{align}
\mu_{\eps,(x_{k-2},x_{k+2})}^{u_-,u_+}\bigg(\mathcal{A}_{1/2,k}^\odot\bigg)\leq \exp\left(-\frac{1}{\eps}\big(\Delta E(\tilde{\mathcal{A}})+\gamma\big)\right).\label{eny2}
\end{align}
We substitute~\eqref{eny1} and~\eqref{eny2} into the ratio on the left-hand side of~\eqref{38} and observe that there is a cancellation of the second factor in the energy difference (see equation~\eqref{endiff}):
\begin{align*}
\frac{\mu_{\eps,(x_{k-2},x_{k+2})}^{u_-,u_+}\bigg(\mathcal{A}_{1/2,k}^\odot\bigg)}{\mu_{\eps,(x_{k-2},x_{k+2})}^{u_-,u_+}\bigg(\mathcal{A}_{k}^\odot\bigg)}
\leq \exp\left(-\frac{1}{\eps}\big(\inf_{u\in\tilde{\mathcal{A}}} E(u)-\inf_{u\in\mathcal{A}} E(u)-\gamma\big)\right).
\end{align*}

Hence, the final ingredient that we need is the following energetic fact.
\begin{lemma}\label{l:unpv}
There exists $C<\infty$ such that for any $M$ large enough and $\delta>0$ small enough, there exists $\ell_*<\infty$ with the following property. Let $\ell\geq\ell_*$ and consider the boundary conditions $u_\pm\in[-M,M]$. Define the sets $\mathcal{A}$ and $\tilde{\mathcal{A}}$ as
\begin{align*}
\mathcal{A}:=&\{u\colon |u(x)|\leq M-\delta\text{ for }x=-2\ell,\,-\ell,\ldots, 2\ell, \\
&\quad\qquad  u\leq 1-2\delta\text{ on }[-\ell,\ell]\},\\
\tilde{\mathcal{A}}:=&\bigg\{u\colon |u(x)|\leq M+\delta\text{ for }x=-2\ell,\,-\ell,\ldots,2\ell,\\
&\qquad\quad \, u \leq 1-\delta/2\text{ on }[-\ell,\ell],
 \, |u(0)+1|\geq \frac{1-\delta}{2}  \bigg\}.
\end{align*}
Then there holds
\begin{equation*}
\inf_{u \in\tilde{\mathcal{A}} } E_{(x_{k-2},x_{k+2})}(u) - \inf_{u \in \mathcal{A}}   E_{(x_{k-2},x_{k+2})}(u) \geq  c_1-C\delta,
\end{equation*}
where
\begin{align}
c_1:=2\min\left\{ \int_{-1}^{-1/2} \sqrt{2V(s)}\,ds,\,\int_{-3/2}^{-1}\sqrt{2V(s)}\,ds\right\}. \label{c1}
\end{align}
\end{lemma}
This lemma is virtually identical to Lemma~\ref{l:last}. The principal difference is that  here the excursion from $-1$ is only of magnitude $1/2$. This changes only the leading order cost (from $c_0$ to $c_1$). We omit the proof of the lemma.
\end{proof}

\begin{proof}[Proof of Lemma \ref{l:hittingz}]
We will prove only~\eqref{e:hz11}, the proof of~\eqref{e:hz29} being essentially the same. We will always assume that the left endpoint of the interval $\Hm$ is greater than or equal to $-\Le$ (since otherwise the boundary condition at $-\Le$ trivially implies the result).

Notice that the set of paths $u \in \Be$ that do not hit $-1$ in $\Hm$  is contained in the following two sets
\begin{itemize}
\item The set  of paths  (a)  in $\Aa_{y,3}^-$ (extra $\delta^-$ layers: recall~\eqref{Apm}) or (b) without extra layers but  more than $1/2$ away from $-1$ at a gridpoint for some $k$ in $\mathcal{I}_-^\eps$.

\item The set $\Aa_{y,4}^-$ of paths in $\Be$ that are within $1/2$ of $-1$ at all  gridpoints with $k\in\mathcal{I}_-^\eps$ but do not hit $-1$ in $\Hm$.

\end{itemize}
Hence, because of the bounds already established in  Lemmas~\ref{l:reflection} and~\ref{l:rightshape}, we will be done as soon as we show
\begin{align}
 \mu^{-1,1}_{\eps,(-\Le,\Le)}\Big( \Aa_{y,4}^- \Big) \leq  \lambda^{\bar{K}_\eps}   \, \mu^{-1,1}_{\eps,(-\Le,\Le)}\big( \Be \big). \label{done}
\end{align}
We remark for reference below that we may assume $M\geq 3/2$ so that $|u(x_k)-1|\leq 1/2$ implies $|u(x_k)|\leq M$.

The interval $\Hm$ can naturally be divided up into $\bar{K}_\eps$ subintervals of length $\ell (2K_\eps+1)$. We set
\begin{equation*}
\bar{k}_j := \klm + j (2 K_\eps+1) \;\text{ for }\;  j = 0, \ldots, \bar{K}_\eps,\;\;\text{and}\;\;
\bar{I}_j :=  [x_{\bar{k}_j} , x_{\bar{k}_{j+1}}]\text{ for }j \leq\bar{K}_\eps-1.
\end{equation*}
We want to use the Markov property and then apply Lemma \ref{l:hittingzero} on these subintervals. Therefore, as usual, we introduce some sets for a decomposition.
\begin{align*}
\Aa^{\ominus}:= &\big\{ u \colon |u(x_k)|\leq M\text{ for }k\leq \kxepsm \big\},\\
\Aa^{\oplus}:= &\big\{ u \colon   |u(x_k)|\leq M  \text{ for } k \geq \kxepsp ,\text{ $\delta ^-$ up layer $\leq 2\ell$ in $J_{y,\eps}$}\big\},\\
\Aa_{bc}^{\odot}:= & \big\{  u \colon \, |u(x_{\bar{k}_j}) -1|\leq  \frac12  \text{ for $ j = 0, \ldots, \bar{K}_\eps $}     \big\},\\
\Aa_j^{\odot}:=& \big\{ u \colon |u(x_k) -1| \leq \frac12 \text{ for } x_k \in \bar{I}_j \big\},\\
\Aa_{-1,j}^{\odot}:= &\big\{ u \in\Aa_j^{\odot} \colon  \text{no hitting of $-1$ in $ \bar{I}_j$  } \big\}.
\end{align*}

We now write $\Aa_{y,4}^-$ as the intersection
\begin{align}
\Aa_{y,4}^-=\Aa^{\ominus}\cap\Aa^{\oplus}\cap\Aa_{bc}^\odot \cap \bigg( \bigcap_{j=0}^{\bar{K}_\eps -1}  \Aa^{\odot}_{-1,j} \bigg) ,\label{x4m}
\end{align}
and apply the Markov property (Lemma \ref{le:Markov1b}) $\bar{K}_\eps$ times to deduce
\begin{align}
 &\mu^{-1,1}_{\eps,(-\Le,\Le)} \Big(  \Aa^{\ominus} \cap \Aa^{\oplus} \cap \Aa_{bc}^{\odot} \cap \Big( \bigcap_{j=0}^{\bar{K}_\eps -1}  \Aa^{\odot}_{-1,j}  \Big)  \Big) \notag\\
&=  \mathbb{E}_{(-\Le,\Le)}^{\mu_\eps,-1,1} \Big( \mathbf{1}_{ \Aa^{\ominus}}(u) \,    \mathbf{1}_{ \Aa^{\oplus}}(u) \, \mathbf{1}_{ \Aa_{bc}^{\odot}}(u)   \, \prod_{j=0}^{\bar{K}_\eps-1}      \mathbb{E}_{(x_{\bar{k}_j},x_{\bar{k}_{j+1}})}^{\mu_\eps,\mathbf{u}} \big( \mathbf{1}_{ \Aa^{\odot}_{-1,j} } (u)\big)  \,  \Big).\label{mk1}
\end{align}
According to Lemma \ref{l:hittingzero}, we have
\begin{equation*}
\mathbb{E}_{(x_{\bar{k}_j},x_{\bar{k}_{j+1}})}^{\mu_\eps,\mathbf{u}} \big( \mathbf{1}_{ \Aa^{\odot}_{-1,j} } (u)\big)   \leq  \,\lambda \,  \mathbb{E}_{(x_{\bar{k}_j},x_{\bar{k}_{j+1}})}^{\mu_\eps,\mathbf{u}} \big( \mathbf{1}_{ \Aa^{\odot}_j } (u)\big),
\end{equation*}
uniformly over all paths $\mathbf{u}$ that satisfy $ \mathbf{u}(x_{\bar{k}_j}), \mathbf{u}(x_{\bar{k}_{j+1}}) \in [-3/2, -1/2]$.
We insert this bound into~\eqref{mk1} and then use the Markov property once more to recover
\begin{align}
 &\mu^{-1,1}_{\eps,(-\Le,\Le)} \Big(  \Aa^{\ominus} \cap \Aa^{\oplus} \cap \Aa_{bc}^{\odot} \cap \Big( \bigcap_{j=0}^{\bar{K}_\eps -1}  \Aa^{\odot}_{-1,j}  \Big)  \Big) \notag\\
 & \leq \lambda^{\bar{K}_\eps} \,  \mathbb{E}_{(-\Le,\Le)}^{\mu_\eps,-1,1} \Big( \mathbf{1}_{ \Aa^{\ominus}}(u) \,     \mathbf{1}_{ \Aa^{\oplus}}(u) \,\mathbf{1}_{ \Aa_{bc}^{\odot}}(u)    \, \prod_{j=0}^{\bar{K}_\eps-1}      \mathbb{E}_{(x_{\bar{k}_j},x_{\bar{k}_{j+1}})}^{\mu_\eps,\mathbf{u}} \big( \mathbf{1}_{ \Aa^{\odot}_j } (u)\big)  \,  \Big)\notag\\
 & =\lambda^{\bar{K}_\eps}   \mu^{-1,1}_{\eps,(-\Le,\Le)} \Big(  \Aa^{\ominus} \cap \Aa^{\oplus} \cap \Aa_{bc}^{\odot} \cap \Big( \bigcap_{j=0}^{\bar{K}_\eps -1}  \Aa^{\odot}_j  \Big)  \Big).\label{mk2}
\end{align}
Since
\begin{equation*}
 \Aa^{\ominus} \cap \Aa^{\oplus} \cap \Aa_{bc}^{\odot} \cap \Big( \bigcap_{j=0}^{\bar{K}_\eps -1}  \Aa^{\odot}_j  \Big)  \subseteq \Be,
\end{equation*}
the combination of~\eqref{x4m} and~\eqref{mk2} completes the proof of~\eqref{done}.
\end{proof}

\section*{Acknowledgements}
We  thank the Max-Planck Institute for Mathematics in the Sciences in Leipzig, where we had the pleasure of working jointly on this project.

The second author would like to thank Volker Betz for explaining to him some background about Schr\"odinger operators. He would also like to thank Martin Hairer and Andrew Stuart for many discussions about this project and  related topics.

The third author would like to thank Eric Vanden--Eijnden for insightful discussions related to ideas developed in this project.

\medskip

Hendrik Weber was partially supported by ERC grant \emph{AMSTAT -- Problems at the Applied Mathematics and Statistics Interface} and by a Philip Leverhulme prize.  Maria G. Westdickenberg was partially supported as an Alfred P. Sloan Research Fellow and by the National Science Foundation under Grant No. DMS--0955051.


\begin{thebibliography}{{Hai}09}
\expandafter\ifx\csname url\endcsname\relax
  \def\url#1{\texttt{#1}}\fi
\expandafter\ifx\csname urlprefix\endcsname\relax\def\urlprefix{URL }\fi

\bibitem[A89]{arrhenius}
\textsc{S. Arrhenius}.
\newblock Ueber die Reaktionsgeschwindigkeit bei der Inversion von Rohrzucker durch S\"auren.
\newblock \emph{Z. Phys. Chem.} \textbf{4}, (1889), 226--248.


\bibitem[BBM10]{BBM}
\textsc{F. Barret}, \textsc{A. Bovier}, and  \textsc{S. M\'el\'eard}.
\newblock{Uniform estimates for metastable transition times in a coupled bistable system.}
\newblock \emph{Preprint}  (2010).



\bibitem[B12]{B12}
\textsc{F. Barret}.
\newblock{Sharp asymptotics of metastable transition times for one dimensional SPDEs.}
\newblock \emph{Preprint}  (2012).



\bibitem[BG12]{BG}
\textsc{N. Berglund} and~\textsc{B. Gentz}.
\newblock{Sharp estimates for metastable lifetimes in parabolic SPDEs: Kramers' law and beyond}.
\newblock \emph{Preprint, 2012}.


\bibitem[BBB08a]{BBB}
\textsc{L. Bertini}, \textsc{S. Brassesco}, and \textsc{P. Butt{\`a}}.
\newblock Soft and hard wall in a stochastic reaction diffusion equation.
\newblock \emph{Arch. Ration. Mech. Anal.} \textbf{190}, no. 2, (2008),
  307--345.



\bibitem[BBB08b]{BBB08}
\textsc{L. Bertini}, \textsc{S. Brassesco}, and \textsc{P. Butt{\`a}}.
\newblock Dobrushin states in the {$\phi^4_1$} model.
\newblock \emph{Arch. Ration. Mech. Anal.} \textbf{190}, no. 3, (2008),
  477--516.

\bibitem[Bog98]{Bog}
\textsc{V. I. Bogachev}.
\newblock \emph{Gaussian measures}, vol. 62 of \emph{Mathematical Surveys and
  Monographs}.
\newblock American Mathematical Society, Providence, RI, 1998.

\bibitem[BEGK04]{BEGK04}
\textsc{A. Bovier}, \textsc{M. Eckhoff}, \textsc{V. Gayrard}, and \textsc{M. Klein}.
\newblock{Metastability in reversible diffusion processes. I. Sharp asymptotics for capacities and exit times.}
\newblock \emph{J. Eur. Math. Soc.} \textbf{6}, no. 4, (2004), 399--424.

\bibitem[BGK05]{BGK05}
\textsc{A. Bovier}, \textsc{V. Gayrard}, and \textsc{M. Klein}.
\newblock{Metastability in reversible diffusion processes. II. Precise asymptotics for small eigenvalues.}
\newblock \emph{J. Eur. Math. Soc.} \textbf{7}, no. 1, (2005), 69--99.

\bibitem[BDMP95]{P95}
\textsc{S. Brassesco}, \textsc{A. De Masi}, and  \textsc{E. Presutti}.
\newblock{Brownian fluctuations of the interface in the D=1 Ginzburg-Landau equation with noise.}
\newblock \emph{Ann. Inst. H. Poincar\'e Probab. Statist.} \textbf{31},  no. 1, (1995), 81--118.


\bibitem[B93]{B}
\textsc{H. B. Braun}.
\newblock Thermally activated magnetization reversal in elongated ferromagnetic particles.
\newblock \emph{Phys. Rev. Lett.} \textbf{71}, no. 21, (1993), 3557--3560.

\bibitem[COP93]{COP93}
\textsc{M. Cassandro}, \textsc{E. Orlandi}, and \textsc{E. Presutti}.
\newblock Interfaces and typical {G}ibbs configurations for one-dimensional
  {K}ac potentials.
\newblock \emph{Probab. Theory Related Fields} \textbf{96}, no. 1, (1993),
  57--96.

\bibitem[dH00]{dH00}
\textsc{F. den Hollander}.
\newblock \emph{Large deviations}, vol. 14 of \emph{Fields Institute
  Monographs}.
\newblock American Mathematical Society, Providence, RI, 2000.

\bibitem[Ei10]{einstein}
\textsc{A. Einstein}.
\newblock Theorie der Opaleszenz von homogenen Flu\"ussigkeitsgemischen in der N\"ahe des kritischen Zustandes.
\newblock \emph{Ann. Physik} \textbf{33}, (1910), 1275--1298.

\bibitem[Ey35]{eyring}
\textsc{H. Eyring}.
\newblock The activated complex in chemical reactions.
\newblock \emph{J. Chem. Phys.} \textbf{3}, (1935), 107--115.

\bibitem[Fa27]{farkas}
\textsc{L. Farkas}.
\newblock Keimbildungsgeschwindigkeit in \"ubers\"attigten D\"ampfen.
\newblock\emph{Z. Phys. Chem.} \textbf{125}, (1927), 236--242.

\bibitem[FJ82]{FJ}
\textsc{W. G. Faris} and \textsc{G. Jona-Lasinio}.
\newblock Large fluctuations for a nonlinear heat equation with noise.
\newblock{J. Phys. A} \textbf{15}, no. 10, (1982), 3025--3055.

\bibitem[Fe48]{feynman}
\textsc{R. P. Feynman}.
\newblock Space-time approach to non-relativistic quantum mechanics.
\newblock \emph{Rev. Mod. Phys.} \textbf{20}, no. 2, (1948), 367--387.

\bibitem[Fr88]{F}
\textsc{M. I. Freidlin}.
\newblock Random perturbations of reaction-diffusion equations: the quasi-deterministic approximation.
\newblock{Trans. Amer. Math. Soc.} \textbf{305}, no. 2, (1988), 665--697.

\bibitem[FW98]{FW}
\textsc{M. I. Freidlin} and \textsc{A. D. Wentzell}.
\newblock \emph{Random perturbations of dynamical systems}, vol. 260 of
  \emph{Grundlehren der Mathematischen Wissenschaften [Fundamental Principles
  of Mathematical Sciences]}.
\newblock Springer-Verlag, New York, second ed., 1998.
\newblock Translated from the 1979 Russian original by Joseph Sz{\"u}cs.


\bibitem[Fu95]{Fu95}
\textsc{T. Funaki}.
\newblock The scaling limit for a stochastic PDE and the separation of phases.
\newblock \emph{Probab. Theory Related Fields} \textbf{102}, no. 2, (1995), 221--288.



\bibitem[GLE41]{GLE}
\textsc{S. Glasstone}, \textsc{K. J. Laidler}, and \textsc{H. Eyring}.
\newblock\emph{The theory of rate processes}.
\newblock McGraw-Hill, New York, 1941.


\bibitem[{Hai}09]{Ha09}
\textsc{M. {Hairer}}.
\newblock {An Introduction to Stochastic PDEs}.
\newblock \emph{ArXiv e-prints} .

\bibitem[HTB90]{HTB}
\textsc{P.~H\"anggi}, \textsc{P. Talkner}, and \textsc{M. Borkovec}.
\newblock{Reaction rate theory: fifty years after Kramers}.
\newblock\emph{Rev. Mod. Phys.} \textbf{62}, no.~2, (1990), 251--341.

\bibitem[Ki74]{kifer}
\textsc{Yu. I. Kifer}
\newblock{Nekotorye rezul'taty, kasayushchiesya malykh sluchaynykh vozmushchenii dinamicheskikh sistem}.
\newblock \emph{Teor. Veroyatnost. i Primenen.} \textbf{19}, no~2, (1974), 514--532.
\newblock English translation of title: Certain results concerning small random perturbations of dynamical systems.

\bibitem[Kr40]{kramers}
\textsc{H. A. Kramers}.
\newblock Brownian motion in a field of force and the diffusion model of chemical
reactions.
\newblock \emph{Physica} \textbf{7}, (1940), 284--304.

\bibitem[LS61]{LS}
\textsc{R. Landauer} and \textsc{J. A. Swanson}.
\newblock{Frequency factors in the thermally activated process}.
\newblock \emph{Phys. Rev.} \textbf{121}, (1961), 1668--1674.

\bibitem[L69]{L}
\textsc{J. S. Langer}.
\newblock Statistical theory of the decay of metastable states.
\newblock \emph{Ann. Phys.} \textbf{54}, (1969), 258--275.

\bibitem[MS77]{MS}
\textsc{B. J. Matkowsky} and \textsc{Z. Schuss}.
\newblock{The exit problem for randomly perturbed dynamical systems}.
\newblock\emph{SIAM J. Appl. Math.} \textbf{33}, (1977), 365--382.

\bibitem[MM77]{MM}
\textsc{L. Modica} and \textsc{S. Mortola}.
\newblock Un esempio di {$\Gamma ^{-}$}-convergenza.
\newblock \emph{Boll. Un. Mat. Ital. B (5)} \textbf{14}, no. 1, (1977),
  285--299.

\bibitem[MoMc89]{pav_book}
\textsc{F. Moss} and \textsc{P. V. E. McClintock}.
\newblock \emph{Noise in Nonlinear Dynamical Systems}.
Cambridge University Press, Cambridge, 1989.

\bibitem[OM53]{OM}
\textsc{L. Onsager} and \textsc{S. Machlup}.
\newblock Fluctuations and irreversible processes.
\newblock \emph{Phys. Rev.} \textbf{91}, no. 6, (1953), 1505--1512.

\bibitem[PAV33]{PAV}
\textsc{L. S. Pontryagin}, \textsc{A. A. Andronov}, and \textsc{A. A. Vitt}.
\newblock O statisticheskom rassmotrenii dinamicheskikh sistem.
\newblock \emph{Zh. Eksper. Teoret. Fiz.} \textbf{3}, no. 3, (1933), 165-180.



\bibitem[RY99]{RY}
\textsc{D. Revuz}, and {M. Yor}
\newblock Continuous martingales and Brownian motion.
\newblock Third edition. Grundlehren der Mathematischen Wissenschaften, 293.
\newblock Springer-Verlag, Berlin, 1999.


\bibitem[RVE05]{RVE05}
\textsc{M. G. Reznikoff} and \textsc{E. Vanden-Eijnden}.
\newblock Invariant measures of stochastic partial differential equations and
  conditioned diffusions.
\newblock \emph{C. R. Math. Acad. Sci. Paris} \textbf{340}, no. 4, (2005),
  305--308.

\bibitem[S79]{simon}
\textsc{B. Simon}.
\newblock\emph{Functional integration and quantum physics}.
\newblock Academic Press, New York, 1979.


\bibitem[S95]{S95}
\textsc{M. Sugiura}
\newblock Metastable behaviors of diffusion processes with small parameter.
\newblock \emph{J. Math. Soc. Japan} \textbf{47}, no. 4, (1995), 755--788.



\bibitem[VW08]{VW}
\textsc{E. Vanden-Eijnden} and \textsc{M. G. Westdickenberg}.
\newblock Rare events in stochastic partial differential equations on large spatial domains.
\newblock \emph{J. Stat. Phys.} \textbf{131}, (2008), 1023--1038.

\bibitem[VH84]{VH}
\textsc{J. H. Van't Hoff}.
\newblock{in Etudes de Dynamiques Chimiques, p. 114}.
\newblock F. Muller and Co., Amsterdam, 1884.
\newblock{Translated by T. Ewan as
Studies in Chemical Dynamics (London, 1896).}

\bibitem[Var84]{Var84}
\textsc{S. R. S. Varadhan}.
\newblock \emph{Large deviations and applications}, vol. 46 of \emph{CBMS-NSF
  Regional Conference Series in Applied Mathematics}.
\newblock Society for Industrial and Applied Mathematics (SIAM), Philadelphia,
  PA, 1984.

\bibitem[Web10]{W10}
\textsc{H. Weber}.
\newblock Sharp interface limit for invariant measures of a stochastic
  {A}llen-{C}ahn equation.
\newblock \emph{Comm. Pure Appl. Math.} \textbf{63}, no. 8, (2010), 1071--1109.

\bibitem[WF70]{wf_early}
\textsc{A. D. Wentzell and M. I. Freidlin}.
\newblock O malykh sluchainykh vozmuschcheniyakh dinamicheskikh sistem.
\newblock \emph{Uspekhi Mat. Nauk} \textbf{25}, no. 1, (1970), 3-55.
\newblock{English translation of title: On small random perturbations of dynamical systems}.


\bibitem[W30]{wiener}
\textsc{N. Wiener}.
\newblock Generalized harmonic analysis.
\newblock \emph{Acta Math.} \textbf{55}, no. 1, (1930), 117--258.

\bibitem[Zab89]{Z88}
\textsc{J. Zabczyk}.
\newblock Symmetric solutions of semilinear stochastic equations.
\newblock In \emph{Stochastic partial differential equations and applications,
  {II} ({T}rento, 1988)}, vol. 1390 of \emph{Lecture Notes in Math.},
  237--256. Springer, Berlin, 1989.

\end{thebibliography}

 \end{document}